\documentclass[11pt]{article}

\usepackage{amssymb}
\usepackage{amscd}
\usepackage{amsmath}
\usepackage{amsfonts}
\usepackage{amsthm}
\usepackage{color}
\usepackage{enumerate}
\usepackage{enumitem}
\usepackage{bbm} 

\usepackage{hyperref}
\usepackage{url}
\usepackage{stmaryrd}

\theoremstyle{plain}

\newtheorem{theo}{Theorem}[section] 
\newtheorem{prop}[theo]{Proposition}
\newtheorem{lemme}[theo]{Lemma}
\newtheorem{cor}[theo]{Corollary}
\newtheorem{defin}[theo]{Definition}

\theoremstyle{definition}

\newtheorem{rem}[theo]{Remark}

\newcommand{\R}{{\mathbb{R}}}
\newcommand{\N}{{\mathbb{N}}}
\newcommand{\C}{{\mathbb{C}}}

\newcommand{\be}{{\beta}}
\newcommand{\al}{{\alpha}}
\newcommand{\la}{{\lambda}}

\newcommand{\si}{{\sigma}}

\newcommand{\ga}{{\gamma}}

\newcommand{\om}{{\omega}}
\newcommand{\Om}{\Omega}

\newcommand{\La}{{\Lambda}}
\newcommand{\Si}{{\Sigma}}
\newcommand{\ep}{\epsilon}

\newcommand{\Ci}{{\mathcal{C}}^{\infty}} 

\newcommand{\op}{\operatorname}
\newcommand{\con}{\overline}
\newcommand{\bigo}{\mathcal{O}}

\newcommand{\biginf}{\mathcal{O}_{\infty}}

\newcommand{\D}{\mathcal{D}}
\newcommand{\symb}{\mathcal{S}}

\newcommand{\lag}{\mathcal{L}} 

\newcommand{\Hilb}{\mathcal{H}}

\newcommand{\wt}{\widetilde}

\newcommand{\pol}{\mathcal{P}}

\newcommand{\EE}{\bold{E}}

\newcommand{\Uu}{\rho}
\newcommand{\Vvv}{\tilde \rho}

\newcommand{\spec}{\op{sp}}

\newcommand{\PPP}{\widetilde{\Box}}
\newcommand{\PP}{\Box}

\newcommand{\bs}{\Lambda} 

\newcommand{\ac}{\mathfrak{a}}

\newcommand{\dom}{\mathcal{G}}

\newcommand{\dscal}{\Delta^{\op{scal}}}

\newcommand{\og}{\mathcal{G}}

\hypersetup{pdfborder={0000}, colorlinks=true, linkcolor=blue, citecolor=red}

\begin{document}

\title{On the spectrum of non degenerate magnetic Laplacian}
\author{L. Charles}

\maketitle

\begin{abstract} 
  We consider a compact Riemannian manifold with a Hermitian line bundle whose
  curvature is non degenerate. Under a general condition, the Laplacian acting
  on high tensor powers of the bundle exhibits gaps and clusters of
  eigenvalues. We prove that for each cluster, the number of eigenvalues that
  it contains, is given by a Riemann-Roch number. We also give a pointwise
  description of the Schwartz kernel of the spectral projectors onto the
  eigenstates of each cluster, similar to the Bergman kernel asymptotics of positive
  line bundles.  Another result is that  gaps and clusters
also appear in local Weyl laws.

\end{abstract}


\section{Introduction} 

Consider a Hermitian line bundle $L$ on a compact Riemannian manifold with a
connection $\nabla$ whose curvature is non degenerate. We will be concerned with the
eigenvalues and eigenstates of  Bochner Laplacians $\Delta_k =
\frac{1}{2} \nabla^*
\nabla + k V$ acting on positive tensor powers $L^k$ of the bundle, $V$ being a
real function, in the limit where $k$ tends to infinity.
Physically, $k^{-2} \Delta_k$ is a magnetic Schrödinger
operator with $k$  the inverse of the Planck's constant, $\nabla$ the
magnetic potential and $k^{-1}V$ the electric potential. 

A very particular case is the $\con \partial$-Laplacian of high powers of a positive line
bundle on a complex manifold. Its ground states are the holomorphic sections
which play obviously a central role in algebraic/complex geometry, but also in
mathematical physics: in K\"ahler quantization, the space of holomorphic
sections is the quantum space and the large $k$ limit is the semi-classical
limit. Starting from the paper \cite{GuUr88} by Guillemin and Uribe, it has
been understood that for a manifold not necessarily complex, the holomorphic
sections can be replaced by the bounded
states of the Bochner Laplacian $\Delta_k$ where the potential $V$ is suitably
defined, bounded here means that the eigenvalues are bounded independently of $k$. 
These ``almost'' holomorphic sections have been used with success in various
problems on symplectic manifolds from their projective embeddings to their
quantizations \cite{BoUr96}, \cite{BoUr00}, \cite{MaMa}. 

In the larger regime where we consider all the eigenvalues smaller than $k \bs$, $\bs$
being arbitrary large but independent of $k$, 
few results are known: a general Weyl law has been established by
Demailly \cite{De} that we will recall later, and for a specific class of connection
$\nabla$, Faure and Tsuji have shown that the spectrum of $\Delta_k$ exhibits some
gaps and clusters \cite{FaTs15}, the first cluster consisting of the bounded
states of \cite{GuUr88}.

A natural question is to determine the number of eigenvalues in each cluster.
For the first cluster, in the case of holomorphic sections of a positive line
bundle, the answer is provided by the Riemann-Roch-Hirzebruch theorem and the
Kodaira vanishing theorem. More generally, when $k$ is sufficiently large, the
number of bounded states of the Bochner Laplacian of \cite{GuUr88} is still
given by the Riemann-Roch number of $L^k$.
One of our main results is that the number of eigenvalues in each higher
cluster is given as well by a Riemann-Roch number, associated to $L^k$
tensored with a convenient auxiliary bundle $F$ defined in terms of the cluster. 

We are also concerned with results of local nature: we show that   gaps and clusters
 appear as well in the local Weyl laws of $\Delta_k$, local here means that each eigenvalue
is counted with a weight
given by the square of the pointwise norm of the corresponding eigensection. Furthermore
we give a pointwise description of the Schwartz kernel of the spectral projectors
associated to each cluster, generalising the Bergman kernel asymptotics for
positive line bundles.

The picture emerging from these results is that the restriction of the Bochner
Laplacian $\Delta_k$ to each cluster
is essentially a Berezin-Toeplitz operator with principal symbol an endomorphism of the
auxiliary bundle $F$.

\subsection{The magnetic Laplacian}  \label{sec:magnetic-laplacian-1}

Let us turn to precise statements. 
Let $M^{2n}$ be a closed manifold equipped with a Riemannian metric $g$, a
volume form $\mu$, a Hermitian line bundle $L$ with a connection compatible
with the metric,  a Hermitian
vector bundle $A$ over $M$ having an arbitrary rank $r$ with a
connection, and a section $V \in \Ci (M , \op{End} A)$ such that $V(x)$ is
Hermitian for any $x\in M$. Define the Laplacian
\begin{gather} \label{eq:delta_k}
  \Delta_k = \tfrac{1}{2} \nabla^* \nabla + k V : \Ci ( L^k \otimes A) \rightarrow \Ci
(  L^k \otimes A) .
\end{gather} 
Here $k \in \N$, $\nabla$ is the covariant derivative of $L^k \otimes A$,
$\nabla^*$ is its adjoint, the scalar products of sections of $L^k \otimes A$
or $ L^k \otimes A \otimes T^*M$ are defined by integrating
the pointwise scalar products against the volume form $\mu$. The metric of
$T^*M$ is induced by the Riemannian metric.

We have introduced the bundle $A$
with the endomorphism valued section $V$ to include some important 
Laplacians as the $\overline{\partial}$-Laplacian acting on $p$-forms or the
square of some Dirac operators. Furthermore our results holds for a slightly
more general class of operators than \eqref{eq:delta_k}, which are defined in
Section \ref{sec:class-magn-lapl} and are locally of the form
\eqref{eq:ass_op}. 

$\Delta_k$ being a formally self-adjoint elliptic operator on a compact
manifold, it is essentially
self-adjoint, its spectrum $\spec (
\Delta_k)$ is a discrete subset of $[ k \inf V_1 , + \infty ) $ and consists  only of
eigenvalues with finite multiplicities, the eigenfunctions are smooth sections
of $L^k \otimes A$. Here $V_1(x) $ is the lowest eigenvalue of $V(x)$. 

The curvature of $L$ has the form $\frac{1}{i} \om$ with $\om \in \Om^2 ( M ,
\R)$ a closed form. Let us assume that:
\begin{gather} \label{eq:hyp_non_degenere} \tag{A} 
 \text{ $\om$ is non degenerate at each point of $M$.} 
\end{gather}
Thus $\om $ is a symplectic form. Associated to $\om$ is the Liouville volume form
$\mu_{\op{L}} = \om^n /n!$. We will assume that $\mu = \mu_{\op{L}}$. This is
not a restrictive assumption because if we multiply $\mu$ by a positive
function $\rho$ and the metric of $A$ by $\rho^{-1}$ we do not change the
scalar products of $\Ci ( L^k \otimes A)$ and $\Om^1 ( L^k \otimes A)$.  Working with $\mu_{\op{L}}$ will
simplify several statements.

\subsection{Pointwise data}  \label{sec:pointwise-data}

We now introduce several pointwise datas, that will enter in our asymptotic
description of the spectrum of $\Delta_k$. 
 Denote by $j_B $ the section of $\op{End}(TM)$ such that $ \om ( \xi , \eta ) = g ( j_B
\xi , \eta )$. Then $M$ has an almost-complex structure $j$ compatible with
$\om$ defined by 
$$j_y  := |j_{B,y}|^{-1}
j_{B,y} , \qquad \forall \; y \in M. $$ 
So the vector bundle $T^{1,0} M = \op{Ker} ( j -
i \op{id}_{TM\otimes \C}) \subset TM \otimes \C$ has a Hermitian metric $h$ given by $h ( \xi, \eta) = \frac{1}{i} \om (
\xi, \con{\eta})$. 

Moreover, the complexification of $\frac{1}{i} j_{B,y}$ restricts to a positive
endomorphism of $(T^{1,0}_yM, h_y)$. Denote its eigenvalue by $0< B_1(y)
\leqslant \ldots \leqslant
B_n(y)$. Introduce an orthonormal basis $(u_i)$ of $(T^{1,0}_yM, h_y) $ such that 
$j_{B,y} u_i = i B_i (y) u_i$.

Consider the space $\D (T_yM) = \C [T^{0,1}_yM]$ of antiholomorphic
polynomials of $T_yM$. If $(z_i)$ are the linear complex coordinates of $T_yM$
dual to the  $u_i$'s, then $\D (T_yM) = \C [ \con{z}_1, \ldots, \con{z}_n]$.
Define the endomorphism 
\begin{gather} \label{eq:defintro_PP_y}
\PP_y = \textstyle{\sum}_i B_i(y) ( \ac_i ^{\dagger} \ac_i + \tfrac{1}{2} ) + V(y) : \D
(T_yM ) \otimes A_y \rightarrow \D(T_yM) \otimes A_y 
\end{gather}
where $\ac_i$ and $\ac_i^{\dagger}$ are the endomorphisms of $\D (T_yM)$
acting by derivation with respect to $\con{z}_i$ and multiplication by
$\con{z}_i$ respectively. 

Introduce an eigenbasis $(\zeta_j)$ of $V(y)$: $V(y) \zeta_i = V_i(y) \zeta_i$
with $V_1(y) \leqslant \ldots \leqslant V_r(y)$.
Then $\PP_y$ is diagonalisable, with eigenbasis $(\con{z}^\al \otimes
\zeta_j$, $(\al, j) \in \N^n \times \{ 1, \ldots, r\})$
$$ \PP_y ( \con{z}^{\al} \otimes \zeta_j) = \bigl( \textstyle{\sum}_i B_i ( y) ( \al (i) +
\tfrac{1}{2}) + V_j(y) \bigr) \;  \con{z}^\al \otimes
\zeta_j .$$  
Let $\la_1(y) \leqslant \la_2(y) \leqslant \ldots $ be the eigenvalues of
$\PP_y$ ordered and repeated according to their multiplicities. 

The operators $\PP_y$ depend smoothly on $y$ even if it is not obvious
from \eqref{eq:defintro_PP_y}, because in general there is no local smooth frame $(u_i)$ of
$T^{1,0}M$ which is an eigenbasis of $j_{B,y}$ at each $y$. The various
eigenvalues $B_i(y)$, $V_j(y)$ and $\la_\ell (y)$ depend continuously on $y$.

\subsection{Weyl laws}

In \cite{De}, Demailly proved a Weyl law for the operators $k^{-1} \Delta_k$. It
says roughly that in the semiclassical limit $k \rightarrow \infty$, the spectrum of $k^{-1}
\Delta_k$ is an aggregate of the spectrums of the $\PP_y$'s. More precisely,
introduce the counting functions $N_y ( \la) = \sharp \{ \ell ; \; \la_\ell(y)
\leqslant \la \}$ of $\PP_y$  and the one  of $k^{-1} \Delta_k$ 
$$ N(\la, k ) = \sharp \spec ( k^{-1} \Delta_k) \cap ] -\infty , \la ].$$
Here and in the sequel, an eigenvalue with multiplicity $m$ is counted $m$
times. Let $v : \R \rightarrow \R$ be the non decreasing function $v(\la) := \int_M N_y ( \la) \; d\mu_{\op{L}}
(y)$. Let $D $ be the
set of discontinuity points of $v$. Then for any $\la
\in \R \setminus D$, we have
\begin{gather} \label{eq:weyl_law_Demailly}
N(\la, k ) = \Bigl( \frac{k}{2\pi} \Bigr)^n v( \la)   + \op{o} (k^{n})
\end{gather}
as $k$ tends to infinity.
We have slightly reformulated the original result \cite[Theorem 0.6]{De}, which holds more
generally for $\om$ not necessarily non degenerate and $M$ not necessarily compact. 

The subset $D$ is in general non empty.
As an easy example, if $j_B = j$ and
$V =0$, then $D = \frac{n}{2} + \N$, $v $ is locally constant on $\R \setminus
D$ and for any $\ell \in \N$, $v(\frac{n}{2} + \ell + 0 ) = v ( \frac{n}{2} +
\ell - 0)+  r \mu_{\op{L}}(M) { n+ \ell -1 \choose \ell -1 }$. 

Our goal is to understand the corrections to the Weyl law \eqref{eq:weyl_law_Demailly}, in other
words what is
hidden in the remainder $\op{o}(k^n)$. For instance, if the function $v$ is constant on a compact interval $J$, then
\eqref{eq:weyl_law_Demailly} implies that $\sharp \spec (
k^{-1} \Delta_k) \cap J = \op{o} ( k^n)$. Actually, as we will see, in this
situation, when $k$ is sufficiently large, $J$ contains no eigenvalue of
$k^{-1} \Delta_k$. Furthermore, the numbers of eigenvalues between such intervals is
given by Riemann-Roch numbers.

To state our result, introduce the set $\Si = \bigcup_j \la_j (M)$. Since $\la_j \geqslant cj -C$
for some positive constants $c$ and $C$, $\Si$ is a locally finite union of
closed disjoint intervals. The function $v$ is locally  constant on $\R
\setminus \Si$.

If $B$ is complex vector bundle of $M$, we denote by $\op{RR} (B)$ the Riemann-Roch
number of $B$, that is the integral of the product of the Chern character of
$B$ by the Todd form of $(M, j)$. 

\begin{theo} \label{theo:intro_weyl_global}
  Let $a, b \in \R \setminus \Si$ with $a < b$. Then when $k$ is
  sufficiently large,
  \begin{gather} \label{eq:counting}
    \sharp \spec ( k^{-1} \Delta_k) \cap [a,b] = \begin{cases} \op{RR} (L^k
    \otimes F )  \text{ if } [a,b] \cap \Si \neq \emptyset \\
    0 \text{ otherwise } 
  \end{cases} 
\end{gather}
where $F$ is the vector bundle with fibers $F_y = \op{Im} 1_{[a,b]} (
  \PP_y)$, $y \in M$.
\end{theo}

$\op{RR} ( L^k \otimes F)$ depends polynomialy on $k$, with leading
term
$$ \op{RR} ( L^k \otimes F)  = (\op{rank} F ) \Bigl(\frac{k}{2 \pi } \Bigr)^n
\mu_{\op{L}} (M) + \bigo ( k^{n-1}). $$
The result is consistent with the Weyl law \eqref{eq:weyl_law_Demailly}
because when $a,b \in \R \setminus \Si$, $N_y ( b ) = N_y (a) + \op{rank} F$ for any $y \in M$.    

Theorem \ref{theo:intro_weyl_global} holds not only for the magnetic
Laplacian \eqref{eq:delta_k}, but also for other remarkable geometric operators, as for instance
the holomorphic Laplacian or the square of spin-c Dirac operators. The
corresponding results are
stated in Theorem \ref{theo:holomorphic_Laplacian} and Theorem
\ref{theo:semicl-dirac-oper}. In these cases, $\Si=\N$, so the spectrum of
$k^{-1} \Delta_k$ consists of clusters at non-negative integers, the dimension
of each cluster being given by the  Riemann-Roch number $\op{RR} ( L^k \otimes F)$
where $F$ is a sum of tensor products of symmetric and exterior powers of
$T^{1,0} M$, cf. part 3 of Theorem \ref{theo:semicl-dirac-oper}.

Theorem \ref{theo:intro_weyl_global} is relevant only when $\Si$ has
several components. This happens when $j = j_B$ and $V=0$ as explained
above. Observe as well that the sets of  $(\om, g, V)$ such that $\Si$ is non
connected, is open in $\mathcal{C}^0$-topology.

$\Si$ being the support of the Lebesgue-Stieltjes measure $d v$, the Weyl law
\eqref{eq:weyl_law_Demailly} implies that for any $\la \in \Si$, the distance $d ( \la,
\spec ( k^{-1} \Delta_k))$ tends to $0$ as $k \rightarrow \infty$. To the
contrary, if $\la \notin \Si$,  by the second case of \eqref{eq:counting},
there exists $\ep >0 $ such that $d ( \la , \spec ( k^{-1} \Delta_k) )
\geqslant \ep$ when $k$ in sufficiently large.

The following theorem gives more precise estimates.
\begin{theo}  \label{theo:intro_spec}
For any $\bs >0$, there exists $C>0$ such that for any $\la \leqslant \bs $, 
  \begin{gather} \label{eq:spectr_approx_1}
  \la \in \Si \Rightarrow  \op{dist} ( \la,  \spec (k^{-1} \Delta_k ) )
  \leqslant C k^{-\frac{1}{2}} , \\  \label{eq:spectr_approx_2}
  \la \in  \spec ( k^{-1} \Delta_k) \Rightarrow  \op{dist} ( \la , \Si ) \leqslant C k^{- \frac{1}{2}}.
\end{gather}
\end{theo}

Interestingly, some local Weyl laws hold with a similar gaped
structure. Instead of $\Si$, the local law at $y\in M$  involves the spectrum $\Si_y  = \{ \la_i ( y) ; \; i \in
\N\}$ of $\PP_y$, which is a discrete subset of $\R$. Clearly, $\Si = \bigcup_y
\Si_y$.

For any $k \in \N$, choose an orthonormal eigenbasis $(\Psi_{k,i})_{i \in \N}$ of
$k^{-1} \Delta_k$ such that $k^{-1} \Delta_k \Psi_{k,i} = \la_{k,i}
\Psi_{k,i}$ with $\la_{0,k} \leqslant \la_{1,k} \leqslant \ldots$. 
For any $y\in M$ and real numbers $a <b$, define
$$ N(y,a,b,k) = \sum _{i, \; \la_{k,i}  \in [a,b] } | \Psi_{k,i} ( y) |^2, $$
so we count the eigenvalues in $[a,b]$ with weights given by the square of the
pointwise norm at $y$ of the corresponding eigenvectors. 

\begin{theo}  \label{theo:intro_weyl_local}
  For any $\bs \in \R \setminus \Si$, $y \in M$ and $ a, b \in ]-\infty,
  \bs] \setminus \Si_y$ such that $a<b$, the following holds: if $[a,b] \cap \Si_y$ is empty,
  then $N(y,a,b, k) = \bigo ( k^{-\infty})$. Otherwise we have an asymptotic expansion
\begin{gather}  \label{eq:local_weyl_law}
  N(y,a,b,k) = \Bigl( \frac{k}{2 \pi } \Bigr)^n \sum_{\la \in \Si_y \cap
  [a,b] } \sum_{\ell =0 }^{\infty} m_{\ell, \la} k^{-\ell} + \bigo (
k^{-\infty}), 
\end{gather}
where the coefficients $m_{\ell, \la}$ do not depend on $a$, $b$, $k$. In particular,
$m_{0, \la}$ is the multiplicity of the eigenvalue $\la$ of $\PP_y$. 
\end{theo}

We believe that the same result holds without the assumption that $a,b $ are
smaller than $\bs \in \R \setminus \Si$. Observe that the first order term  $\sum_{\la \in [a,b]}  m_{0, \la} $
in \eqref{eq:local_weyl_law} is merely the number of eigenvalues of $\PP_y$
in $[a,b]$. In particular we recover the same structure as in the counting law
\eqref{eq:counting} of Theorem \ref{theo:intro_spec}: when the leading order
term is zero, then $N(y,a,b, k) = \bigo ( k^{-\infty})$. We interpret this as
a gap in the local Weyl law. 

Besides these gaps and clusters, another notable aspect in Theorem
\ref{theo:intro_spec} and \ref{theo:intro_weyl_local} is that we have full
asymptotic expansions. For the Laplace-Beltrami operator or the Schrödinger
operator without magnetic field, the remainders in Weyl laws have a completely
different behavior, cf. for instance the survey \cite[Section 8]{Ze_survey}.
Another situation where clusters and gaps occur is for the pseudo-differential operators
whose principal symbol has a periodic Hamiltonian flow. This has been studied
in many papers, see for instance \cite{Do} and references therein for a semi-classical result and \cite{Boutet_bourbaki},
\cite[Section 1]{BoGu} for earlier results, with Riemann-Roch numbers
already. For our magnetic Laplacian,
the gaps are also connected to periodic Hamiltonians: the quantum harmonic
oscillators $\ac_i ^{\dagger} \ac_i$ of \eqref{eq:defintro_PP_y}. In dimension 2,
this lies at the origin of the cyclotron motion or resonance of a charged particle in a magnetic field.

\subsection{Schwartz kernels of spectral projectors} \label{sec:schw-kern-spectr}

Another result we would like to emphasize in this introduction is the asymptotic description of the Schwartz
kernel of $g(k^{-1} \Delta_k)$ where $g: \R \rightarrow \C$ is a bounded
function with compact support satisfying some  assumptions.  These Schwartz kernels are by definition given at $(x,y) \in M^2$ by
$$ g ( k^{-1} \Delta_k ) (x,y) = \sum_i g ( \la_{k,i} ) \Psi_{k,i} ( x) \otimes
\con {\Psi_{k,i} (y)} \in L^k_x \otimes A_x \otimes \con{L}^k_y \otimes
\con{A}_y   .$$
We will prove that $g(k^{-1} \Delta_k)$ belongs to the operator algebra $\lag ( A)$
introduced in \cite{oimpart1}. Let us recall the main characteristics of $\lag ( A)$, the complete definition will be given in Section \ref{sec:operator-class-lag}.

$\lag
(A)$ consists of families $(P_k)_{k \in \N}$ such that for any $k$,  $P_k$ is
an endomorphism of $\Ci (
M , L^k \otimes A)$ having a smooth Schwartz kernel in $\Ci ( M^2, (L^k
\otimes A ) \boxtimes (\con{L}^k \otimes \con{A}))$ satisfying the following
conditions. First,   for any compact subset $K$ of $M^2$ not intersecting the diagonal,
    for any $N$,  $P_k
    (x,y) = \bigo ( k^{-N})$ uniformly on $K$. Second, for any open set $U$ of
    $M$ identified with a convex open set of $\R^{2n}$ through a
    diffeomorphism, let $F \in \Ci (U^2,  L\boxtimes \con{L})$
    be the unitary frame such that $F(x,y) = u \otimes \con{v}$, where $v$ is
    any vector in $L_y$ with norm $1$ and $u \in L_x$ is the parallel
    transport of $v$ along the path $t \in [0,1] \rightarrow y + t (x-y)$.
    Introduce a unitary trivialisation of $A$ on $U$ and identify accordingly
    the sections of
    $A \boxtimes \con{A}$ over $U^2$ with the functions of $\Ci ( U^2, \C^r \otimes
    \con{\C}^r)$. Then the Schwartz kernel of $P_k$ has the following asymptotic
    expansion on $U^2$: for any $N \in \N$, for any $x \in U$ and $\xi \in T_xU$ such that $x+ \xi \in U$, 
\begin{xalignat}{2} \label{eq:beginsplit-p_k-x+}
\begin{split}     P_k (x+ \xi, x) = &\Bigl( \frac{k}{2\pi} \Bigr)^n
F^{k}(x+\xi,x)   e^{-\frac{k}{4} |\xi|_x^2} \sum_{\ell =0 }^N k^{-\ell}
a_{\ell} (x, k^{\frac{1}{2}} \xi ) \\  & + \bigo (
k^{n-\frac{N+1}{2}}) 
\end{split}
\end{xalignat}
where $|\xi|_x^2 = \om_x ( \xi, j_x \xi)$, the coefficients $a_{\ell} (x, \cdot)$ are polynomials map $T_xM
\rightarrow \C^r \otimes \con{\C}^r$
depending smoothly on $x$, and the $\bigo $ is uniform when $(x+ \xi, x)$ runs
over any compact set of $U^2$. 

Such an operator $P = (P_k)$ has a symbol $\si_0(P)$ which at $y \in M$ is
the endomorphism of $\mathcal{D} (T_yM) \otimes A_y$ defined by 
$$ \bigl( \si_0 ( P) (y) \bigr) ( f ) (u)  = (2 \pi)^{-n} \int_{T_y M} e^{
  (u -v) \cdot \con{v} } a_0 ( y, u-v) f (v) \; d \mu_y (v) $$  
Here, the scalar product $ u \cdot \con{v}$ and the measure $ \mu_y$ are
defined in terms of linear complex coordinates $z_i : T_yM \rightarrow \C$
associated to an orthonormal frame of $(T_y^{1,0}M, h_y)$ by
 $ u \cdot \con{v} = \sum z_i (u) \con{z_i(v)}$ and $\mu_y = | dz_1
 \ldots dz_n d\con{z}_1 \ldots d\con{z}_n |$.

As a result, for any $(P_k) \in \lag (A)$, $\| P_k \| = \bigo (1)$ and 
$$ \|
P_k \| = \bigo (k^{-\frac{1}{2}}) \Leftrightarrow (\si_0 (P)(y)  = 0 ,\; 
\forall y \in M) \Leftrightarrow (a_0 ( y, \cdot) = 0, \; \forall y \in M) .$$ 
Furthermore $\lag(A)$ is closed by product and the map $\si_0$ is an algebra
morphism. Here the product of the symbols at $y$ is the composition of
endomorphisms of $\D (T_yM) \otimes A_y$, which is not commutative. 

\begin{theo} \label{theo:intro_proj} 
\begin{enumerate} 
\item 
For any $a, b \in \R \setminus \Si$, the spectral projector
  $\Pi_k:= 1_{[a.b]} ( k^{-1} \Delta_k)$   and $k^{-1} \Delta_k \Pi_k$ belong
  to $\lag (A)$ and their symbols at $y$ are equal to $1_{[a,b]} (\PP_y)$ and $\PP_y
  1_{[a,b]} ( \PP_y) $ respectively. 
\item For any $\bs \in  \R \setminus
\Si$, for any $g \in \Ci ( \R, \C)$ such that $\op{supp} g \subset ]-\infty,
\bs]$, $ (g ( k^{-1} \Delta_k) )_k $ belongs to $\lag (A)$ and its symbol at $y$ is $g(
\PP_y)$.  
\end{enumerate}
\end{theo}

The second assertion is actually a generalisation of the first one because
choosing $ \bs >b$ such that $[b, \bs] \cap \Si = \emptyset$, one has 
$1_{[a,b]} = g$ on an open neighborhood of $\Si$ with $g \in \Ci ( \R)$
supported in $]-\infty, \bs]$, and by  Theorem
\ref{theo:intro_weyl_global}, 
$1_{[a,b]} (\la) = g ( \la)$ for
any $\la \in \spec ( k^{-1} \Delta_k)$ when $k$ is sufficiently large.

\subsection{Comparison with earlier results}

This work started as a collaboration with Yuri Kordyukov and some of the
results presented here appeared also in \cite{Yuri_1}: the existence of
spectrum gaps, that is \eqref{eq:counting} when $[a,b ] \cap \Si = \emptyset$, and a
weak version of \eqref{eq:spectr_approx_2} with a $\bigo ( k^{-\frac{1}{4}}) $
instead of the $\bigo ( k^{-\frac{1}{2}})$ are proved in \cite[Theorem
1.2]{Yuri_1}. Moreover, under the assumption of Theorem \ref{theo:intro_proj},
the Schwartz kernel  of the spectral projector $\Pi_k = 1_{[a,b]} ( k^{-1}
\Delta_k)$ is described in \cite[Theorem 1.6]{Yuri_1} in a way similar to our result. 

In the case where $j_B = j$ and $V$ is constant, the existence of spectrum
gaps, that is \eqref{eq:counting} when $[a,b ] \cap \Si = \emptyset$,  was proved in
\cite[Theorem 10.2.2]{FaTs15}. Our proof will follow
the same line as in \cite{FaTs15} and is similar to the proof in \cite{Yuri_1}.

In the case again where $j_B =j$ and $V=0$, the first gap  and the
asymptotic description of  the first cluster has a long history. When $j$ is integrable so that $M$ is a complex manifold and $\om$
is K\"ahler, the gap follows from Kodaira vanishing Theorem, the first
cluster consists of the holomorphic sections of $L^k$, its dimension is given
by the Riemann-Roch-Hirzebruch theorem, the Schwartz kernel of the
corresponding spectral projector is the Bergman kernel, whose asymptotic can
be deduced from \cite{BoSj} and has been used in many papers starting from
\cite{Ze}. The extension to almost-complex structure has been done in
\cite{GuUr88}, \cite{BoUr07}, \cite{MaMa08}. Parallel results for
spin-c Dirac operators have been proved in \cite{BoUr96}, \cite{MaMa02},
\cite{MaMa}. 

The main tool we use in this paper is the algebra $\lag (A)$ introduced in
\cite{oimpart1}, a first weaker version was proposed in \cite{oim}. The
asymptotic expansions \eqref{eq:beginsplit-p_k-x+} or similar versions have
been used before by several authors to describe spectral projector on the first
cluster and corresponding Toeplitz operators \cite{ShZe}, \cite{oimbt},
\cite{MaMa} for instance. In \cite{oimpart1}, besides establishing the main properties of
$\lag (A)$, we considered some projectors $(\Pi_k)$ in $\lag (A)$ whose symbol at
$y \in M$ is the projector onto the $m$-th level of a Landau Hamiltonian $\sum
 \ac_i^{\dagger} \ac_i$. In particular we computed the rank of $\Pi_k$ as a
 Riemann-Roch number and we studied the corresponding Toeplitz algebra.  By
 the results of the current paper, particular instances of such projectors are the spectral
 projectors on the $m$-cluster of a magnetic Laplacian with $j_B= j$ and $V=0$.

In a different context, many works have been devoted to the magnetic
Schrödinger operator in $\R^n$, cf. \cite{NiRa} for a general overview. 
The most significant result  is a semiclassical description of the bottom of the spectrum in terms of 
effective operators whose principal symbols are the functions we denoted by $\la_i$,
cf. for instance \cite[Theorem 6.2.7]{Ivrii}, \cite[Theorem 1.6]{SaNi} or 
\cite[Theorem 2]{LM} for a statement in the manifold setting. These works differ
in at least two ways from the current paper: the global gap assumption is
generally replaced by a
confinement hypothesis, typically
the function we denote by $\la_0$
is assumed to have a non-degenerate minimum. Moreover, the general strategy is
to put the Schr\"odinger operator on a normal form by conjugating it with a
convenient Fourier integral operator. 

$$ $$

\subsection{Outline of the paper}

The main idea in the first part of the paper is to approximate the Laplacian
$\Delta_k$ locally by a family of Laplacians $\Delta_{y,k}$, $y \in M$ obtained
from $\Delta_k$ by ``freezing'' the coordinates at $y$. In
Section \ref{sec:linear-data} we introduce these operators, recall the basic results regarding their
spectrum  and explain the relationship with the operators
$\PP_y$ of Section \ref{sec:pointwise-data}. In Section \ref{sec:class-magn-lapl}, we introduce a
class of Laplacian slightly more general than the magnetic Laplacian
$\Delta_k$ and which are well approximated by the $\Delta_{y,k}$. This class
contains the holomorphic Laplacian and some of its generalisation without
integrable complex structure. In section \ref{sec:spectral-estimates}, we
prove a weak version of Theorem
\ref{theo:intro_spec} which says that $\spec (k^{-1} \Delta_k) \rightarrow \Si$ in
the limit $k \rightarrow \infty$, by constructing on one hand some peaked sections which are
approximate eigenmodes of $\Delta_k$, and on the other hand, by inverting 
 $\la - k^{-1} \Delta_k$ up to a $\bigo ( k^{-\frac{1}{4}})$ when $\la \notin
 \Si$. 
 
In the second part of the paper, Sections \ref{sec:operator-class-lag} and
\ref{sec:proof-theor-refth}, we introduce the algebra  $\lag(A)$ and prove
that the spectral projector $1_{[a,b]}( k^{-1} \Delta_k)$ belongs to $\lag
(A)$ when $a, b \in \R \setminus \Si$.  The proof is divided in three steps:
from the resolvent estimate of Section \ref{sec:spectral-estimates}, we deduce
that any operator of $\lag (A)$ having for symbol $1_{[a,b]} ( \PP)$ is an approximation of
$1_{[a,b]} ( k^{-1} \Delta_k)$ up to a $\bigo ( k^{-\frac{1}{4}})$. We then
prove that $\lag (A) / \bigo ( k^{-\infty})$ has a unique self-adjoint projector having for
symbol $1_{[a,b]} ( \PP)$ and commuting with $\Delta_k$. Finally
we prove that this operator is indeed the spectral projector.

In the last part, Section \ref{sec:glob-spectr-estim}, we establish some spectral properties for the Toeplitz
operators associated to the projectors of $\lag (A)$, including a sharp
G{\aa}rding inequality and the functional calculus. Then we deduce Theorem
\ref{theo:intro_weyl_global}, Theorem \ref{theo:intro_weyl_local} and the
second part of Theorem \ref{theo:intro_proj}. 

\subsection*{Acknowledgment} I would like to thank Yuri Kordyukov for his collaboration at an early stage of this
work. I benefited from Colin Guillarmou insight to establish the pointwise
estimate of spectral projectors. And I had the chance to present this work to
Pierre Flurin, L\'eo Morin and San V{\~u} Ng{\d{o}}c, which helped me to improve it.

\section{The linear pointwise data} \label{sec:linear-data}

In this section we consider a compact manifold $M^{2n}$ equipped with a symplectic
form $\om$ and a Riemannian metric $g$. Let $A \rightarrow M$ be a Hermitian
vector bundle with a section $V$ of $\Ci ( M , \op{End} A )$ such that $V(x)$
is Hermitian for any $x\in M$.  We choose a point $y \in M$.

\subsection{The complex structure}

Let  $j_{B,y} $ be the endomorphism of $T_yM$ such that
$\om_y  ( \xi , \eta) = g_y ( j_{B,y}  \xi, \eta)$. It will be useful to work
with the following normal form.  
\begin{lemme} \label{lem:normal_form}
  There exists $0< B_1(y) \leqslant \ldots \leqslant B_n(y)$ such
  that $T_yM$ has a basis $(e_i,f_i)$ satisfying
\begin{gather*} 
\om_y
(e_i, e_j) = \om_y (f_i , f_j) =0, \qquad  \om_y ( e_i, f_j) = \delta_{ij} \\
j_{B,y} e_i = B_i(y) f_i, \qquad j_{B,y} f_i = -B_i(y) e_i
\end{gather*}
The vectors $u_i= \frac{1}{\sqrt 2} ( e_i - i f_i)$, $\con{u}_i = \frac{1}{\sqrt 2}
( e_i + i f_i) $ are a basis of $T_yM \otimes \C$ and
\begin{gather*} 
\tfrac{1}{i} \om_y ( u_i, u_j) = \tfrac{1}{i} \om ( \con{u}_i, \con{u}_j) =
0 \qquad \tfrac{1}{i} \om ( u_i , \con{u}_j) = \delta_{ij} \\
j_{B,y} u_i  = i B_i ( y) u_i, \qquad j_{B,y} \con{u}_i  = - i B_i ( y)
\con{u}_i
\end{gather*} 
\end{lemme}  
\begin{proof} Since $j_{B,y}$ is a $g_y$-antisymmetric invertible endomorphism
  of $T_yM$, there
  exists a $g_y$-orthonormal basis $(\tilde{e}_i, \tilde{f}_i)$ such that
  $j_{B,y} \tilde{e}_i = B_i(y) \tilde{f}_i$ and $j_{B,y} \tilde{f}_i = -
  B_i(y) \tilde{e}_i$, where the $B_i(y)$'s are positive. We set  $e_i =
  (B_i(y))^{-\frac{1}{2}} \tilde{e}_i$  and $f_i =  (B_i(y))^{-\frac{1}{2}}
      \tilde{f}_i$, and the result follows by direct computations.  
\end{proof}
We can interpret this result as follows: first, $\frac{1}{i} j_{B,y}$ is $\C$-diagonalisable with only nonzero real
eigenvalues, denoted by $\pm B_i(y)$. Second, the subspace $W$ of $T_yM \otimes \C$ spanned
by the $u_i$'s is the sum of the eigenspaces of $\frac{1}{i} j_{B,y}$ with a
positive eigenvalue. $W$ is Lagrangian and the sesquilinear form $h_y$ of
$T_yM \otimes \C$ given by $h_y(u,v) = \frac{1}{i} \om_y ( u, \con{v})$  is positive on $W$.
Equivalently the endomorphism $j_y$ of $T_yM$ such that $j_y = i$ on $W$ is a
complex structure of $T_yM$ compatible with $\om_y$. So from now on, we will
denote $W = \op{ker} (j_y - i)$ by $T^{1,0}_yM$ and by the definition of $j_y$, the restriction
of $\frac{1}{i} j_{B,y}$ to $T^{1,0}_yM$ is a positive endomorphism of
$(T^{1,0}_yM, h_y )$ with eigenvalues the $B_i(y)$. Hence the vectors $(u_i)$ in
Lemma \ref{lem:normal_form}
are nothing else than a $h_y$-orthonormal eigenbasis of $T^{1,0}_yM$. 

An important remark is that $j_y$ depends smoothly on $y$, so it defines an
almost complex structure of $M$. Indeed, the space $T^{1,0}_yM$ depends
smoothly on $y$ because $\frac{1}{i} j_{B,y}$ being invertible, no eigenvalue can cross
$0$. Another reason is that $j_y = |j_{B,y}|^{-1} j_{B,y}$ where $|j_{B,y}|$
is the positive square root of the $g_y$-positive endomorphism $-j_{B,y}^2$. Actually,
the construction of $j$ is the classical proof of the fact that any symplectic manifold admits a compatible almost-complex
structure, cf  \cite[Proposition 2.5.6]{MS}.

To the contrary, in general, we cannot choose a local continuous symplectic frame
$(e_i,f_i)$ of $TM$ such that
$j_B e_i = B_i f_i$, $j_B f_i = B_i e_i$, even if we
renumber the eigenvalues $B_i(y)$ in a way depending on $y$. Indeed, as is
well known, it is not possible in general to diagonalise smoothly a symmetric
matrix, the symmetric matrix being $-(j_{B,y})^{2}$ in our case. More specifically,
consider on $\R^2 \otimes \R^2$ with its usual Euclidean structure the
endomorphism $j_B(s,t) = M (s,t)\otimes j_2$ where
$$ M (s,t) = \begin{pmatrix} 1+s & t \\
  t & 1 -s 
\end{pmatrix}, \qquad j_2 = \begin{pmatrix} 0 & -1 \\
  1 & 0 
\end{pmatrix},$$
$s$ and $t$ being parameters in a neighborhood of $0$. Then $j_B$ is non
degenerate and antisymmetric, and we can choose for each $(s,t)$ a basis
$(e_i,f_i)$ satisfying the previous conditions, but not continuously with respect
to $(s,t)$.  Indeed, $- j_B^2 (s,t) = M^2(s,t) \otimes \op{id}$ and for $s=0$,
$t$ small non zero, the eigenspaces of $M(s,t)$ are $(1,1)\R$ and $(1,-1)\R$ whereas
for $t=0$ and $s$ small non zero, they  are $(1,0)\R$ and
$(0,1) \R$.

This example appears on $\R^4$ equipped with its usual Euclidean metric and the 
closed form 
$$ \om = (1+ p_1) dp_1 \wedge dq_1 + (1-p_2) dp_2 \wedge dq_2 + q_1 dq_1 \wedge dp_2
- q_2 dp_1 \wedge dq_2 $$
which is symplectic on a neighborhood of the origin. On the plane $\{ p_1 =
p_2, \; q_1 = q_2 \}$, the matrix of $j_B$ is $M(p_1, q_1) \otimes j_2$. 

We have also to be careful that the metric $\tilde{g}$ determined by
$(\om, j)$
\begin{gather} \label{eq:tildeg}
\tilde{g}_y ( \xi , \eta ) := \om_y (  \xi , j_y \eta) = g_y ( |j_{B,y}| \xi, \eta) .
\end{gather}
 is equal to $g_y$ only when $B_1(y) = \ldots = B_n(y) = 1$, that is when
$j_{B,y}$ is itself a complex structure. 

\subsection{The scalar Laplacian of $T_yM$} \label{sec:laplacian-linear}

Consider now the covariant derivative
\begin{gather}  \label{eq:nabla-linear}
\nabla = d + \tfrac{1}{i} \al : \Ci (
T_yM ) \rightarrow \Om^1 ( T_yM ) 
\end{gather}
where
$\al \in \Om^1(T_y M, \R)$ is given by $\al_\xi ( \eta) = \frac{1}{2} \om_y ( \xi, \eta)$.  Since $d \al = \om_y$, the curvature of $\nabla$
is $ \frac{1}{i} \om_y$. We then define the scalar Laplacian of $T_yM$ by
\begin{gather} \label{eq:delta_x}
\dscal_y  :=  \tfrac{1}{2} \nabla^* \nabla   : \Ci (
T_yM) \rightarrow \Ci (T_yM).
\end{gather}
Here the scalar
products of $\Ci (T_yM)$ and $\Om^1 (T_yM)$ are defined by integrating the
pointwise scalar products against a fixed constant volume form, the pointwise
scalar product of $\Om^1 (T_yM)$ is defined from the metric $g_y$.

We can explicitly compute the spectrum and eigenfunctions of
$\dscal_y$ as
follows. Introduce a basis $(e_i,f_i)$ of $T_yM$ as in Lemma \ref{lem:normal_form}. This basis is
$g_y$-orthogonal and $g_y(e_i, e_i) = g_y(f_i, f_i) =
B_i (y) ^{-1}$, so we have
\begin{xalignat*}{2} 
\dscal_y & = - \tfrac{1}{2} \sum_{i=1}^n B_i(y)  (
\nabla_{e_i} ^2 + \nabla_{f_i}^2 )  \\
& =  \sum_{i=1}^{n} B_i(y)  \bigl( - \nabla_{u_i}
\nabla_{\con{u}_i} + \tfrac{1}{2} \bigr)  
\end{xalignat*}
where $u_i = \frac{1}{\sqrt 2} ( e_i - i f_i)$, $\con{u}_i = \frac{1}{\sqrt 2}
( e_i + i f_i) $. Denote by $z_i$ the linear complex coordinates dual to the
$u_i$. If $(p_i,q_i)$ are the real linear coordinates of $T_yM$ in the basis
$(e_i,f_i)$, then $z_i = \frac{1}{\sqrt 2} ( p_i + i q_i)$. 
Since $\om_y = i \sum dz_i \wedge d \con{z}_i$, we have  $$\nabla = d + \tfrac{1}{2} \sum_{i=1}^n
(z_i d\con{z}_i - \con{z}_i dz_i)
.$$
Introduce the function $s( \xi )  := \exp (
-\frac{1}{4} |\xi|_y^2)$, $\xi \in T_yM$, where
$$ |\xi|_y^2 = \textstyle{\sum}_i (p_i^2 + q_i^2) = 2 \textstyle{\sum}_i |z_i|^2 = \tilde{g}_y( \xi, \xi) $$
 Since $s =  \exp ( - \frac{1}{2}
|z|^2) $, we have $\nabla_{\con{u}_i} s = 0$, so $s$ is $\nabla$-holomorphic.

Let us consider $\Ci(T_yM)$ as the space of sections of the trivial line
bundle over $T_yM$
and let us use $s$ as a global frame. Introduce the operators
\begin{gather} \label{eq:a_i_def}
\ac_i = \partial_{\con{z}_i}, \qquad \ac_i^{\dagger} = \con{z}_i -
\partial_{z_i}. 
\end{gather}
Then   $\nabla_{\con{u}_i} (f s
) = (\ac_i f ) s$ and $ \nabla_{u_i}(f s) = - (\ac_i^\dagger f) s$, so that
\begin{gather} \label{eq:def_PPP}
  \dscal_y (fs ) = (\PPP^{\op{scal}}_y f)s \qquad \text{with} \qquad  \PPP_y^{\op{scal}} := \sum_{i=1}^n
B_i(y) ( \ac_i^{\dagger} \ac_i + \tfrac{1}{2} ) . 
\end{gather}
Let $\pol (T_yM )$ be the space of polynomial
functions $T_yM  \rightarrow \C$, non necessarily holomorphic or anti-holomorphic. With the
coordinates $(z_i)$, $\pol (T_yM) = \C [z_1, \ldots, z_n , \con{z}_1, \ldots,
\con{z}_n]$. Observe that $\ac_i$ and $\ac_i^\dagger$ preserves $\pol (T_yM )$ and
the same holds for $\PPP_y^{\op{scal}}$.   

Since the  $\ac_i$, $\ac_i^{\dagger}$ satisfy the so-called canonical
commutation relations
$$ [ \ac_i ,
\ac_j ] = [\ac_i^{\dagger}, \ac_j^{\dagger} ] = 0 , \qquad [\ac_i , \ac_j^{\dagger}] =
\delta_{ij}, $$   
we deduce by a classical argument that the endomorphisms $\ac_i^{\dagger} \ac_i$ of
$\pol ( T_yM )$ are mutually commuting endomorphisms, each of them
diagonalisable with spectrum $\N$, cf. for instance \cite[Proposition 4.1]{oimpart1}. So we have a decomposition into joint eigenspaces
\begin{gather} \label{eq:dec_spectral}
\pol ( T_yM ) = \bigoplus_{\al \in
    \N^n} \mathfrak{L}_\al \quad \text{ with } \quad \mathfrak{L}_{\al} = \textstyle{\bigcap}_{i=1}^{n} \op{ker} (  \ac_i^\dagger \ac_i - \al
  (i) ).
\end{gather}
Furthermore, $ \mathfrak{L}_0 = \C [z_1, \ldots ,z_n]$ and $\mathfrak{L}_{\al}
=  (\ac^\dagger)^\al \mathfrak{L}_0 $, for all $\al \in \N^n $
  where $(\ac^{\dagger} )^{\al} = (\ac_1^{\dagger} )^{\al (1)} ...(\ac_n^{\dagger} )^{\al (n)}$.
  Consequently $\PPP^{\op{scal}}_y$ is a diagonalisable endomorphism of
$\pol ( T_yM)$ with spectrum $\Si_y^{\op{scal}}$ given by 
\begin{gather} \label{eq:def_si_y_scal}
 \Si_y^{\op{scal}} = \bigl\{ \textstyle{\sum}_{j=1}^n B_j(y)( \al (j)  + \tfrac{1}{2} ),\; \al
\in \N^n \bigr\} .
\end{gather}
Moreover the eigenspace $\mathcal{E} ( \la)$ of the eigenvalue $\la \in \Si_y$ is the sum of the
$\mathfrak{L}_{\al}$ where $\al$ runs over the multi-indices of $\N^n$  such that $\sum B_i(y) (\al (i) +
\frac{1}{2} )  = \la$. 

We can deduce from these algebraic facts the $L^2$-spectral theory of
$\dscal_y$.  First of all, the space $\exp ( -\frac{1}{4} |\xi|_y^2)  \pol ( T_yM)$ is dense in $L^2 ( T_yM)$ by the same
proof that Hermite functions are dense. So we deduce from
\eqref{eq:dec_spectral} a decomposition of $L^2 (T_yM)$ in a Hilbert  sum of orthogonal subspaces
\begin{gather} \label{eq:dec_Hilbert} 
L^2 ( T_yM ) = \bigoplus_{\al \in \N^n} \mathcal{K}_{\al} , \qquad
 \mathcal{K}_{\al} = \con{ e^{-\frac{1}{4}|\xi|_y^2} \mathfrak{L}_\al} ^{L^2(T_y
   M)}, \; \forall \al \in
 \N^n
\end{gather}
Let $\dom $ be the subspace of $L^2 (T_y M)$ consisting of the
$\psi $ having a decomposition $\sum \psi_{\al}$ in \eqref{eq:dec_Hilbert}
such that $\sum |\al|^2 \| \psi_{\al} \|^2 $ is finite. As a differential
operator,  $\dscal_y $ acts on the distribution space $\mathcal{C}^{-\infty} (T_yM)$ and in particular on
$\dom$. 

\begin{lemme} \label{lem:spectre}
  The space $\dscal_y ( \dom )$ is contained in $L^2
  (T_yM)$ and  the restriction of $\dscal_y$
to $\dom $ is a self-adjoint unbounded operator of $L^2 (T_yM)$.
Furthermore, its spectrum is $\Si_y$ and consists only of
eigenvalues, the eigenspace of $\la \in \Si_y $ being the closure of $ \exp (
-\frac{1}{4} |\xi|_y^2)  \mathcal{E} ( \la)$.
\end{lemme}
This follows from $\PPP^{\op{scal}}_y = e^{\frac{1}{4} |\xi|_y^2} \dscal_y e^{-\frac{1}{4} |\xi|_y^2}
  $ by elementary
standard arguments, cf. \cite[Lemma 1.2.2]{Davies} for instance.  Observe also
that the unboun\-ded operator $(\dscal_y, \dom )$ is the closure of
$(\dscal_y$, $ e^{-\frac{1}{4} |\xi|_y^2}\pol ( T_yM)$).

\subsection{The $A_y$-valued Laplacian $\Delta_y$} \label{sec:a_y-valued-laplacian}
We now consider the full Laplacian
\begin{gather} \label{eq:lap_model}
\Delta_y := \Delta_y^{\op{scal}} + V(y) : \Ci ( T_y M , A_y ) \rightarrow
\Ci ( T_yM, A_y) 
\end{gather}
We deduce from the properties of $\Delta_y^{\op{scal}}$ that $(\Delta_y , \dom \otimes
A_y)$ is a selfadjoint unbounded operator of $L^2(T_yM) \otimes  A_y$ with discrete spectrum
\begin{gather}  \label{eq:si_y}
\Sigma_y = \bigl\{   \textstyle{\sum}_{j=1}^n B_j(y)( \al (j)  + \tfrac{1}{2}
) + V_\ell (y) ,\; \al
\in \N^n, \; \ell = 1, \ldots, r \bigr\} 
\end{gather}
where $V_1(y) \leqslant \ldots \leqslant V_r(y)$ are the eigenvalues of
$V(y)$. Let $(\zeta_\ell)$ be an eigenbasis of $V(y)$, $V(y) \zeta_\ell = V_\ell(y)\zeta_\ell$. Then any $\la \in \Sigma_y$ is an eigenvalue of $\Delta_y$
with eigenspace the closure of the sum of the $ \exp (
-\frac{1}{4} |\xi|_y^2) \mathcal{E}(
\la') \otimes \C \zeta_\ell$ such that $\la' + V_\ell (y) = \la$.  

In the sequel, we will mainly work with
\begin{gather} 
\PPP_y =e^{\frac{1}{4}|\xi|_y^2} \Delta_y  e^{-\frac{1}{4}|\xi|_y^2}= \PPP_y^{\op{scal}}+ V(y)
\end{gather}
acting on $\pol (T_yM) \otimes A_y$. 

\subsection{The restriction $\PP_y$ of $\PPP_y$ to anti-holomorphic polynomials} 
\label{sec:restr-delt-anti}

By its definition \eqref{eq:delta_x}, $\Delta_y$ depends smoothly on $y$,
acting  on a infinite dimensional space though. Nevertheless, in many
arguments, we can work in finite dimension by replacing $\pol (T_yM)$ by the
subspace  $\D (T_y M) \subset \pol (T_yM)$ of anti-holomorphic
polynomials. With the coordinates $(z_i)$ introduced previously, $\D (T_y M) =
\C [ \con{z}_1, \ldots , \con{z}_n ]$.

First, the annihilation and
creation operators $\ac_i$, $\ac_{i}^\dagger$ preserves the subspace $\D
(T_yM)$ in which they act respectively
by $\partial_{\con{z}_i}$ and $\con{z}_i$. Moreover the joint eigenspaces
$\mathfrak{L}_{\al}$ of the $ \ac_i^{\dagger} \ac_i$ satisfy $\mathfrak{L}_{\al} \cap \D
(T_xM) = \C \, \con{z}^\al$.
So $\PPP_y$ preserves $\D (T_yM) \otimes A_y $ and its restriction $\PP_y \in \op{End}
( \D (T_yM) \otimes A_y )$  has the same
spectrum as $\PPP_y$. For any eigenvalue $\la$, the corresponding eigenspaces of
$\PPP_y$ and $\PP_y$ are $\bigoplus \mathfrak{L}_\al \otimes \C \zeta_\ell $ and
$\bigoplus \C \con{z}^\al \otimes \C \zeta_\ell$
respectively where in both cases we sum over the $(\al, \ell)$ such that $\sum B_i (y) (
\al (i ) + \frac{1}{2} ) + V_{\ell}(y) = \la$. 

For any $p \in \N$, the endomorphism $\PP_y$ preserves the subspace
$\D_{\leqslant p} (T_yM) $ of
$\D(T_yM)$ of  polynomials with degree smaller than
$p$. These spaces are obviously finite-dimensional and their union
$\D_{\leqslant p}( TM) = \bigcup_y \D_{\leqslant p } (T_yM)$ is a genuine
vector bundle over $M$. Moreover $ y
\mapsto \PP_y |_{\D_{\leqslant p } (T_y M)}$ is a smooth section of $ \op{End} (
  \D_{\leqslant p} (TM))$.

\begin{lemme} \label{lem:smooth}  $ $ \begin{enumerate} 
 \item  
  For any $\bs >0$, there exists $p \in \N $ such that for any $ y \in M$ and $\la$ $\in$
  $\op{sp} (\PP_y) \cap (-\infty, \bs ]$, the eigenspace $\op{ker} ( \PP_y - \lambda)$ is contained
  in $\mathcal{D}_{\leqslant p} (T_yM) \otimes A_y$. 
\item For any compact interval $I$ whose endpoints do not belong to $\Si$, the spaces
  $$ F_y = \bigoplus_{\la \in \op{sp} ( \PP_y ) \cap I } \op{ker} ( \la -
  \PP_y), \qquad y \in M $$
  are the fibers of a subbundle $F$ of $\mathcal{D}_{\leqslant p} (TM) \otimes
  A$, with
  $p$ a sufficiently large integer. 
\end{enumerate}
\end{lemme}

\begin{proof}
The functions $B_i$ being continuous, $B_i(y) \geqslant \ep$ on
$M$ for some $\ep>0$,  so $\sum B_i(y) \al (i)  \geqslant \ep |\al|$ with
$|\al| = \al(1) + \ldots + \al (n)$. So $\sum B_i(y) ( \al (i) + \frac{1}{2})
+ V_{\ell} (y) \leqslant \bs$ implies $ \ep |\al| \leqslant \bs - \op{inf} V_1$, which
proves the first assertion with $p$ any integer larger than $\ep^{-1} ( \bs - \op{inf}
V_1)$. 

$I$ being bounded, by the first part, $F_y \subset \mathcal{D}_{\leqslant p}
(T_y M) \otimes A_y$ for any $y$, when $p$ is sufficiently large. The projector of
$\op{End} ( \mathcal{D}_{\leqslant p } (T_yM) \otimes A_y )$ onto $F_y$ is given by the
Cauchy integral formula   
\begin{gather} \label{eq:proj_symbol}
  (2 \pi i )^{-1} \int_{\ga}  (\la -
\PP_{y,p})^{-1} d\lambda ,
\end{gather}
where $\PP_{y,p}$ is the restriction of $\PP_y$ to $\D_{\leqslant p} (T_yM)
\otimes A_y$ and
$\ga$ is a loop of $\C \setminus \Si_y$ which encircles $I$. By the assumption
that the endpoints of $I$ do not belong to $\Si$, we can choose $\ga$
independent of $y$. Hence \eqref{eq:proj_symbol} depends smoothly on $y$ and its image $F_y$ as
well. 
\end{proof}

\section{A class of magnetic Laplacians} \label{sec:class-magn-lapl}

Consider a compact Riemannian manifold $(M,g)$ equipped with a Hermitian
line bundle $L$ with a connection $\nabla$ of curvature
$\frac{1}{i} \om$, and a Hermitian vector bundle $A$
with a section $V \in \Ci ( M, \op{End} (A))$ such that $V(x)$ is Hermitian
for any $x\in M$.

The results we will prove later hold for families of differential operators 
$$(\Delta_k : \Ci ( M , L^k
\otimes A) \rightarrow \Ci ( M ,L^k \otimes A), \ k \in \N)$$
 having the following local form:  for any
coordinate chart $(U, x_i)$ of $M$ and trivialisation $A|_U \simeq U \times
\mathbb{A}$, we have on $U$ by identifying $\Ci (U, L^k \otimes A)$ with $\Ci
(U, L^k \otimes \mathbb{A})$ that 
\begin{gather} \label{eq:ass_op} \tag{B}
\Delta_k = - \tfrac{1}{2} \sum g^{ij} \nabla_{i,k} \nabla_{j,k}  + k V + \sum a_i
\nabla_{i,k}  + b 
\end{gather}
where $g^{ij} = g ( dx_i, dx_j)$, $\nabla_{i,k}$ is the covariant derivative of
$L^k$ with respect to $\partial_{x^i}$, and $a_i$, $b$ are in $\Ci ( U,\op{End} (\mathbb{A})) $ and do not depend on $k$.

 So on a semi-classical viewpoint with  $\hbar = k^{-1}$,  $k^{-2} \Delta_k
= \si ( x , \frac{1}{ik} \nabla_i^k )$ with $\si$ the (normal) symbol given
by
\begin{gather} \label{eq:semi_class}
\si ( x, \eta)
= -\tfrac{1}{2} \sum g^{ij}(x)
\eta_i \eta_j + \hbar V(x) + \sum \hbar a_i(x) \eta_i + \hbar^2 b(x).
\end{gather}
As we will see, for the small eigenvalues, the variable $\eta_i$ has the same
weight as $\hbar^{1/2}$, so in the above sum, the two first terms have weight
$\hbar$ and the third and fourth terms are lower order terms with weight
$\hbar^{3/2}$ and $\hbar^2$ respectively. 

In this section, we will prove that various operators have the form 
\eqref{eq:ass_op}: the magnetic Laplacian defined in the introduction
\eqref{eq:delta_k}, the holomorphic Laplacian and also some generalised
Laplacians associated to semi-classical Dirac operators.

\subsection{About Assumption \eqref{eq:ass_op} }
The
proof that some operators satisfy Assumption \eqref{eq:ass_op} consists in each case of establishing a Weitzenb\"ock type formula. Since we
don't need to give a geometric definition of the coefficients $a_i$ and $b$ in
\eqref{eq:ass_op}, the computations will be rather simple once we know
which terms to neglect. To give a
systematic treatment and to have a better understanding of the approximations
we do, we will introduce non-commutative symbols for  the
differential operator algebra generated by the $\nabla_{i,k}$ and $\Ci ( U,
\op{End} \mathbb{A})$. Instead of the full algebra, we will only work with
second order operators. Everything in this section works
without assuming that $\om$ is degenerate, the dimension of $M$ could be odd
as well, but we will not insist on that. 

Let $(e_i)$ be a frame of $TM$ on an open set $U$
of $M$ and $\mathbb{A}$ be a Hermitian vector space. 
Let $\nabla_{i,k} $ be the covariant
derivation of $\Ci ( U,L^k)$  with respect to
$e_i$. For any $y \in M$, let $\nabla_{y,i}$ be the covariant derivative of $\Ci (
T_yM)$ for the connection \eqref{eq:nabla-linear} with respect to $e_i(y)$.

We say that a family  $P=(P_k : \Ci ( U , L^k \otimes \mathbb{A} ) \rightarrow \Ci ( U , L^k
\otimes\mathbb{A}), \; k \in \N )$ of differential operators belongs to
$\og_2$ if it has the form
\begin{gather} \label{eq:defin_op_order2}
P_k=  \sum_{i\leqslant j} d_{ij} \nabla_{i,k} \nabla_{j,k} + kc + \sum b_i
   \nabla_{i,k} + a 
 \end{gather}
 for some coefficients $ d_{ij}$, $c$, $b_i$, $a \in \Ci ( U,
 \op{End} \mathbb{A} )$ independent of $k$. For such a family,  we define
 $$\si_2(P)(y)  = \sum_{i \leqslant j} d_{ij} (y) \nabla_{y,i}
 \nabla_{y,j} + c(y): \Ci (T_yM, \mathbb{A}) \rightarrow \Ci (T_yM, \mathbb{A}) $$
 Similarly we define the subspaces $\og_0$ and $\og_1$ of
 $\og_2$ and the corresponding symbols as follows. Assume that  $P$
 satisfies \eqref{eq:defin_op_order2}. Then
\begin{xalignat*}{4}
  P \in \og_1 & \Leftrightarrow     d_{ij} = c = 0, && \si_1(P)(y)  = \sum b_i (y) \nabla_{y,i}\\     P \in \og_0 & \Leftrightarrow d_{ij} = c = b_i =0,  &&  \si_0(P) (y) = a(y)   
\end{xalignat*} 
The basic property we need is the following. 
\begin{lemme} \label{lem:adj_prod}
  Let $P \in \og_N$, $P' \in \og_{N'}$ with $N+N'
  \leqslant 2$. Then
  \begin{itemize}
\item $P^*:= (P_k^*)$ belongs to $\og_{N}$ and $\si_N(P^*) (y) = (\si_N(P)(y))^*$. 
\item  $PP' := (P_k P'_k) \in \og_{N+N'}$ and $\si_{N+N'} ( PP')
   (y) = \si_N(P)(y) \circ \si_{N'}( P')(y)$. 
 \end{itemize}
\end{lemme}
Here the formal adjoints $P_k^*$ are defined with respect to any volume form $\mu$
of $U$ which is independent of $k$, whereas the adjoint of $\si_N(P)(y)$ is
defined with respect to any constant volume form of $T_yM$. 
\begin{proof} This is easily proved, let us emphasize the main points. First $\nabla_{i,k}^* = - \nabla_{i,k} + \op{div}_\mu (e_i)$,
  so $(\nabla_{i,k}^*)$ belongs to $\og_1$ and $\si_1 ( \nabla_{i,k}^*)(y)
  = - \nabla_{y,i} = \nabla_{y,i}^*$. Second $\nabla_{i,k} a = a \nabla_{i,k}
  + \mathcal{L}_{e_i} a$, so $(\nabla_{i,k} a)$ belongs to $\og_1$ and
  has symbol $\si_1( \nabla_{i,k} a)(y) = a(y) \nabla_{y,i} = \nabla_{y,i}
  a(y)$. Third
  $$ \nabla_{i,k} \nabla_{j,k} = \nabla_{j,k} \nabla_{i,k} + \tfrac{k}{i} \om
  (e_i,e_j)$$ so when $i>j$, $ (\nabla_{i,k} \nabla_{j,k})$ belongs to
  $\og_2$ and $\si_2 ( \nabla_{i,k} \nabla_{j,k})(y) =\nabla_{y,j}
  \nabla_{y,i} + \frac {1}{i} \om(e_i, e_j) (y) = \nabla_{y,i} \nabla_{y,j}$. 
\end{proof}

Notice that for any vector field $X$ of $U$, $(\nabla_X^{L^{k}})$
belongs to $\og_1$ with symbol at $y$ given by the covariant
derivative of $\Ci ( T_yM)$ with respect  to $X(y)$. Using this and  Lemma \ref{lem:adj_prod}, we deduce that $\og_N$ and $\si_N$ do not depend on the choice of the frame $(e_i)$.
Let us make the dependence with respect to $(U, \mathbb A )$ explicit, so we
write $\og_N(U, \mathbb{A})$ instead of  $\og_N$.

Using again Lemma \ref{lem:adj_prod}, we see that if $u \in \Ci (U, \op{End} \mathbb{A})$ is invertible at
each point, then for any $P \in \og_N (U, \mathbb{A})$, $uPu^{-1}$ belongs to
$\og_N(U, \mathbb{A})$ and $\si_N( uPu^{-1} ) (y)= u(y) \si_N(P)(y) u(y)^{-1}$. So we
can define $\og_N(A)$ as the space of differential operator families $(P_k
)$  such that for any $k$, $P_k$ acts on $\Ci (M, L^k \otimes A)$ and for any trivialisation $A|_U
\simeq U \times \mathbb{A}$, the local representative of $(P_k)$ belongs to
$\og_N(U,\mathbb{A})$. The corresponding symbol $\si_N(P)(y)$ is
invariantly defined as a differential operator of $\Ci ( T_yM, A_y)$.  

It is also useful to consider  differential operators from $\Ci (M ,L^k \otimes A )$ to  $\Ci ( M ,L^k
\otimes B)$ where $B$ is a second auxiliary Hermitian vector bundle. To handle
these operators, we define the subspace $\og_N(A,B)$ of
$\og_N( A \oplus B)$ consisting of the $(P_k)$ such that
for any $k$,  $\op{Im} P_k \subset \Ci ( M , L^k \otimes B)
\subset \op{Ker} P_k$. The symbol at $y$ of an
element of $\og_N( A,B)$ is a differential operator $\Ci (T_yM , A_y)
\rightarrow \Ci (T_yM , B_y)$.

Observe now that  assumption \eqref{eq:ass_op} has the following reformulation
\begin{gather} \tag{B'}
(\Delta_k) \in \og_2(A) \qquad \text{and} \qquad \si_2 (
\Delta_k)(y)  = \Delta_y^{\op{scal}} + V(y), \quad \forall y \in M.
\end{gather} 

\subsection{Magnetic Laplacian} \label{sec:magnetic-laplacian}

The simplest example of an operator satisfying condition \eqref{eq:ass_op} is the magnetic Laplacian
defined in Section \ref{sec:magnetic-laplacian-1}. So besides the line bundle $L$ with its
connection, the Riemannian metric $g$ and the section $V \in \Ci (M , \op{End}
A)$, we introduce a connection
on $A$ not necessarily preserving the Hermitian structure and a volume form
$\mu$ on $M$. Set
 $$ \Delta_k = \tfrac{1}{2}( \nabla^{L^k \otimes A} )^*   \nabla^{L^k \otimes
   A} + k V : \Ci ( M , L^k \otimes A ) \rightarrow \Ci ( M , L^k \otimes A )$$
 where the formal adjoint of $\nabla^{L^k \otimes A}$ is defined from the
 scalar product obtained by integrating pointwise scalar products against $\mu$.
 
 \begin{prop} $(\Delta_k)$ satisfies assumption \eqref{eq:ass_op}.
 \end{prop}
 \begin{proof} This follows from Lemma \ref{lem:adj_prod} and the fact that $(\nabla^{L^k \otimes
     A})$ belongs to $\og_1( A, A \otimes T^*M)$ with symbol at $y$
   equal to the covariant derivative $\nabla$ of $\Ci (T_yM)$ tensored with
   the identity of $A_y$. To see this, write locally 
   $$\nabla^{L^k \otimes
   A} = \sum_i \ep (e_i^*) \nabla_{e_i} ^{L^k \otimes A} = \sum_i \ep (
 e_i^*) (\nabla_{i,k} + \ga_i) $$
 where $(e_i^*)$ is the dual frame of $(e_i)$, $\ep (e_i^*)$ is the exterior product
 by $e_i^*$ and the $\ga_i \in \Ci (U, \op{End} \mathbb{A})$ are the
 coefficients of the connection one-form of $\nabla^A$ in a trivialisation
 $A|_U \simeq U \times \mathbb{A}$.
 \end{proof}

 \subsection{Holomorphic Laplacian} \label{sec:holom-lapl}

 Assume that $M$ is a complex manifold and $L$, $A$ are holomorphic
 Hermitian bundles, $L$ being positive in the sense that the curvature of its
 Chern connection if $\frac{1}{i} \om$ where $\om \in \Om^{1,1}(M)$ is
 a K\"ahler form. Equip $T^{0,1} M$ with the metric $|u |^2 = \frac{1}{i} \om (
 \con{u}, u)$, $u \in T^{0,1}M$ and let $\mu = \om^n/n!$ be the Liouville
 volume form. Define the holomorphic Laplacian
 $$ \Delta''_k = (\con{\partial}_{L^k \otimes A} )^* \con{\partial}_{L^k
   \otimes A} : \Ci ( M , L^k \otimes A) \rightarrow \Ci (M, L^k \otimes
 A). $$
By Hodge theory,  $\op{ker} \Delta''_k$ is isomorphic with the Dolbeault
cohomology space $H^0 ( L^k \otimes A)$. When $k$ is
sufficiently large, the dimension of $H^0 ( L^k \otimes A)$ is the
Riemann-Roch number $\op{RR}(L^k \otimes A)$ defined as the evaluation of the
product of the Chern
character of $L^k \otimes A$ by the Todd Class of $M$. 
  $\Delta''_k$  satisfies assumption \eqref{eq:ass_op}, which
 lead to the following description of its spectrum. 
 
 \begin{theo} \label{theo:holomorphic_Laplacian}
   For any   $\La>0$, there exist $C>0$ such that $\spec (
k^{-1}  \Delta''_k ) \cap [0, \La]$ is contained in $ \N + Ck^{-1} [-1,1] $
 For any $m \in \N$, 
   $$ \sharp \spec ( k^{-1}\Delta''_k) \cap [m-\tfrac{1}{2}, m+ \tfrac{1}{2} ]  = \op{RR} (
   L^k \otimes A \otimes \op{Sym}^m (T^{1,0}M)) ,$$ when $k$ is sufficiently
   large. 
 \end{theo}
 Notice that the first eigenvalue cluster is degenerate in the sense that $ \spec ( \Delta''_k) \cap [0,\frac{1}{2} ] \subset \{0\}$ when $k$ is
        sufficiently large. 
\begin{proof} $\con{\partial}_{L^k \otimes A}$ belongs to $\og_1 ( A, A \otimes
 (T^*M)^{0,1})$ and its symbol at $y$  is the $(0,1)$ component of the
 connection  $\nabla$ defined in
 \eqref{eq:nabla-linear}. Using the same notations $(u_i)$ and $(z_i)$ as in Section
 \ref{sec:laplacian-linear},  $\nabla^{0,1} = \sum \ep ( d\con{z}_{i} ) \otimes
 \nabla_{\con{u}_i}$. Since the adjoint of $\ep (d\con{z}_i)$ is the interior product by
 $\con{u}_i$, $\ep (d\con{z}_i) ^* \ep (d \con {z}_i) =1$ so that
 $$\si_2 ( \Delta''_k) (y) = - \textstyle{\sum}_i   \nabla_{u_i} \nabla_{\con{u}_i} $$
 so $\Delta''_k$ satisfies assumption \eqref{eq:ass_op} with $V(y) = -
 \frac{n}{2}$ and $\Si_y = \N$. The
 results follow now from Corollary \ref{cor:glob-spectr-estim} with a
 $k^{-1}$ instead of $k^{-\frac{1}{2}}$ by Remark \ref{rem:improvment}. 
\end{proof}
Similarly we can consider the Laplacian acting on $(0,q)$-forms and prove the
same result where $\N$ is replaced by $q+ \N$ and the number of eigenvalues in
$q+m + [-\frac{1}{2}, \frac{1}{2}] $ is the Riemann-Roch number of $L^k
\otimes A \otimes \wedge ^{0,q} ( T^*M) \otimes \op{Sym}^m ( T^{1,0}M)$. 

We can also generalise this to the case where the complex structure is not
integrable. So assume that $(M, \om)$ is a symplectic manifold with a
compatible almost-complex structure $j$, that $L \rightarrow M$ is a Hermitian
line bundle with a connection of curvature $\frac{1}{i} \om$ and $A$ a
Hermitian vector bundle with a connection.  Then Theorem
\ref{theo:holomorphic_Laplacian} holds with the operator
$$ \Delta''_k = (( \nabla^{L^k \otimes A})^{(0,1)})^* (\nabla^{L^k \otimes A})^{(0,1)} : \Ci (L^k
\otimes A) \rightarrow \Ci ( L^k \otimes A)$$
and the proof is exactly the same. However, it is no longer true that the
first eigenvalue cluster is non degenerate.
Using Dirac operators, one can generalise the previous result and still have
the degeneracy of the first cluster, as explained in the next section. 

\subsection{Semiclassical Dirac operators} 

In this section, $(M, \om , j)$ is a symplectic manifold with an almost complex structure, $(L, \nabla) $ is a Hermitian line bundle on $M$  with a
connection having curvature $\frac{1}{i} \om$ and $A$ an auxiliary Hermitian
vector bundle.

Let $S = \wedge ^{0, \bullet } T^*M$ be
the spinor bundle and $S^+$, $S^-$ be the subbundles of even (resp. odd) forms. 
For any $y \in M$, extend the covariant derivative $\nabla$ defined in 
\eqref{eq:nabla-linear} to $\Om^\bullet (T_yM)$ in
the usual way and denote by $\nabla^{0,1}$ the restriction of its $(0,1)$
component to $\Om^{0, \bullet} (T_yM) = \Ci (T_yM \otimes S_y)$.   
\begin{defin} A semi-classical Dirac operator is a family $(D_k) \in \og_1 (A \otimes S)$ with symbol
  \begin{gather*}   \si_1 (D_k) (y)  = \nabla^{0,1} + (\nabla^{0,1})^* :
    \Om^{0, \bullet} (T_yM)  \rightarrow \Om^{0, \bullet} (T_yM)
    , \qquad \forall y \in M 
  \end{gather*}
  such that for any $k$, $D_k$ is formally self-adjoint and odd. 
\end{defin}
Such an operator can be constructed as follows: introduce a connection on $S$
preserving $S^+$ and $S^-$, a connection on $A$ and set
$$ D_k = \sum_i  \ep (\con{\theta}_i) \nabla_{\con{u}_i}^{L^k \otimes A
  \otimes S}  + \bigl( \ep (\con{\theta}_i) \nabla_{\con{u}_i}^{L^k \otimes A
  \otimes S} \bigr)^*   $$ 
where $(u_i)$ is any orthonormal frame of $T^{1,0}M$, $(\theta_i)$ is the
dual frame of $(T^*M)^{1,0}$ and the exterior product $\ep (\con{\theta}_i)$ acts on $S$. Another
example is provided by spin-c Dirac operators, cf.  \cite{Du}, \cite[Section 1.3]{MaMa}. Observe as
well that the semi-classical Dirac operator is unique up to a self-adjoint odd
operator of $\og_0 ( A \otimes S)$. 
We denote by $D^\pm_k: \Ci ( L^k \otimes A \otimes S^{\pm}) \rightarrow \Ci (
L^k \otimes A \otimes S^{\mp}) $ the restrictions of $D_k$ and observe that
$D^-_k$ is the formal adjoint of $D^+_k$. 
\begin{theo} \label{theo:semicl-dirac-oper}
  Let $(D_k)$ be a semi-classical Dirac operator. Then the operator  $\Delta_k =
  D_k^- D_k^+$ satisfies  
  \begin{enumerate}
    \item  for any $\La >0$, there exists $C>0$ such that $\spec ( k^{-1} \Delta_k) \cap
 [0,\La]$
 is contained in $\N + Ck^{-\frac{1}{2}} [-1,1]$. 
\item $\spec (k^{-1} \Delta_k) \cap [0,\frac{1}{2}] \subset \{0 \}$ and
  $\op{ker} \Delta_k $ has dimension $\op{RR} ( L^k
  \otimes A)$ when $k$ is sufficiently large.
\item for any $m\in \N$, when $k$ is sufficiently large, 
  $$\sharp \spec ( k^{-1} \Delta_k) \cap [m -\tfrac{1}{2},
  m+ \tfrac{1}{2}] = \op{RR} ( L^k \otimes A_m) $$
  where $A_m = \bigoplus_{(\ell,p)} A \otimes \op{Sym}^\ell ( T^{1,0}M) \otimes
  \wedge^{2p} (T^{1,0}M) $, the sum being over the $(\ell, p) \in \N^2$ such
  that $\ell + 2p =m$ and $p\leqslant n$.
  \end{enumerate}
\end{theo}

\begin{proof} As in section \ref{sec:laplacian-linear}, let $(u_i)$ be an orthonormal basis 
  of $T_y^{1,0}M$ and $(z_i)$ be the associated linear complex coordinates. We
  have $\nabla_{\con u_i}^* = -\nabla_{u_i}$, $\ep (d
  \con{z}_i)^* = \iota ( \con{u}_i)$ so that
   $$\si_1(D_k)(y) = \sum \ep
  (d\con{z}_i)  \nabla_{\con{u_i}} - \iota( \con{u}_i) \otimes \nabla_{u_i}. $$ 
  A standard computation using that $\nabla_{u_i}$, $\nabla_{\con{u}_i}$
  commute with $\ep (d\con{z}_j)$, $\iota ( \con{u}_j)$ and  
  $ [\nabla_{u_i}, \nabla_{u_j} ] = [\nabla_{\con{u}_i}, \nabla_{\con{u}_j} ]
  =0$, $ [\nabla_{u_i} , \nabla_{\con{u}_j}] =
  \delta_{ij}$ leads to
$$  \si_2(D_k^2)(y) = \sum (-  \nabla_{u_i}\nabla_{\con{u}_i} +  \ep (
d\con{z}_i) \iota(
  \con{u}_i) ) = \Delta^{\op{scal} }_y - \tfrac{n}{2} + \op{N}_y$$
  where $\op{N}_y$ is the number operator of $S_y$, that is $N_y \al =
 ( \op{deg} \al ) \al $. Restricting to $S^+$, we deduce that $(\Delta_k)$
 satisfies assumption \eqref{eq:ass_op} with $V(y) = -\frac{n}{2} + \op{N}_y$.  
 So $\Si_y = \N$ and the first assertion of the theorem follows from the second
 part of Corollary \ref{cor:glob-spectr-estim}.

 In the same way, $(D_k^+ D_k^-)$ has the form \eqref{eq:ass_op} with $V(y) = -\frac{n}{2} +
 \op{N}_y$ as well, but the number operator  takes odd value on
 $S^-$. Thus  $$\spec (k^{-1} D_k^+ D_k^-) \subset [1 -Ck^{-\frac{1}{2}}, \infty[$$ for
 some positive $C$. Since for any $\la \neq 0$, $D_k^+$ is an
 isomorphism between $ \op{ker} ( D_k^-D_k^+ -\la )$ and $\op{ker} ( D_k^+ D_k
 ^- -\la)$, this proves that $ \spec (k^{-1} \Delta_k) \cap ]0, \frac{1}{2}] $
 is empty when $k$ is sufficiently large and the first part of the second assertion follows.  

 The second part of the second assertion and the third assertion follow from
 Corollary \ref{cor:glob-spectr-estim}. Indeed, for $V(y) = -\frac{n}{2} +
 \op{N}_y$ acting on $A_y \otimes S^+_y$,  the bundle $F$ with fiber
 $ F_y =  \ker (
 \PP_y - m ) $ is isomorphic with $A_m$ as a complex vector bundle.
\end{proof}

\section{Spectral estimates}  \label{sec:spectral-estimates}

Let $(\Delta_k : \Ci ( M , L^k \otimes A) \rightarrow \Ci ( M , L^k \otimes
A))$ be a differential operator family satisfying \eqref{eq:ass_op}. We assume
that the curvature $\frac{1}{i} \om $ is non-degenerate.
We assume as well that $\Delta_k$ is formally self-adjoint where the scalar
product of section of $\Ci ( M, L^k \otimes A)$ is defined from the measure $\mu =\om^n/n!$.

For any $y \in M$, by Darboux Lemma, there exists a coordinate system
$(U,x_i)$ of $M$
centered at $y$ such that $\om$ is constant in these coordinates, that is $ \om = \frac{1}{2} \sum
\om_{ij} dx_i \wedge dx_j$ with $\om_{ij} = \om ( \partial_{x_i}, \partial_{x_j})$ constant functions. We identify $U$
with a neighborhood of the origin of $T_y M$ through these coordinates. We
assume that this neighborhood is convex. 

Introduce a unitary section $F_y $ of $L \rightarrow U$ such that for any
$\xi \in U$, $F_y$ is flat on the segment $[0,\xi]$. Then
\begin{gather} \label{eq:nabla-f_y}
 \nabla F_y = \tfrac{1}{2i} \textstyle{\sum}_{i,j}  \om _{i,j} x_i dx_j \otimes
 F_y 
\end{gather}
Indeed $\nabla F_y =  \frac{1}{i} \al \otimes F_y$ with $\al$ satisfying $d \al = \om$ and $\int_{[0,\xi]} \al = 0 $ for any $\xi \in
U$. We easily see that these conditions determine a unique $\al$ and that
they are satisfied by $\al = \frac{1}{2} \sum \om_{ij} x_i dx_j$. 

So trivialising $L$ on $U$ by using this frame $F_y$, $L|_U \simeq U \times \C$
and $\nabla $ becomes the linear connection defined in \eqref{eq:nabla-linear}.
Moreover trivialising $L^k$ on $U$ with $F^k_y$,  the covariant derivative
$\nabla_{j,k}$ of $L^k$ with respect to $\partial_{x_j}$ is
\begin{gather} \label{eq:nabl-i-k-linear}
\nabla_{j,k} = \partial_{x_j} + \tfrac{ik}{2}  \textstyle{\sum}_i  \om _{i,j}
x_i 
\end{gather}
Now introduce the Laplacian $\Delta_{y,k}$ of $\Ci ( T_yM, A_y)$ associated to this
covariant derivative, the constant metric $g_y$ of $T_yM$ and the constant
potential $kV(y)$, that is 
\begin{gather}  \Delta_{y,k} = - \tfrac{1}{2} \sum g^{ij}_y \nabla_{i,k} \nabla_{j,k} +
kV(y). 
\end{gather}
For $k=1$, we recover the Laplacian $\Delta_y$ defined in \eqref{eq:lap_model}.

Introduce a trivialisation of the auxiliary vector bundle $A|_U = U \times
A_y$ so that $\Ci ( U, L^k \otimes A) \simeq \Ci (U, A_y)$.
Then  assumption \eqref{eq:ass_op} tells us  that 
\begin{gather} \label{eq:main_observation}
\Delta_k - \Delta_{y,k} = \sum_{i,j}  a_{ij} \nabla_{i,k} \nabla_{j,k} + \sum_i a_i
\nabla_{i,k} + k c + b 
\end{gather}
where  $a_{ij}= - \frac{1}{2}g^{ij} + \frac{1}{2}g^{ij}_y$ and $c = V - V(y)$
are both equal to zero at
the origin $y$. The identity \eqref{eq:main_observation} will be used later to
compare the spectrums of $\Delta_k$ and $\Delta_{y,k}$, cf. the proofs of
Proposition \ref{prop:peaked-sections} and Lemma \ref{lm:main-computation}.

Before that, let us compute the spectrum of $\Delta_{y,k}$.  The Laplacian $k^{-1} \Delta_{y,k}$ is unitarily
conjugated to $\Delta_y$. Indeed,
introduce the rescaling map
\begin{gather} \label{eq:rescaling_map}
S_k: \Ci ( T_y M, A_y) \rightarrow \Ci (  T_y M, A_y), \qquad S_k (f) ( x) =
k^{\frac{n}{2}} f ( k^{\frac{1}{2}} x).
\end{gather}
Then, from the formula  \eqref{eq:nabl-i-k-linear}, we easily check that 
\begin{gather} \label{eq:rescale}
k^{\frac{1}{2}} S_k
\nabla_i = \nabla_{i,k} S_k, \qquad k^{-1} \Delta_{y,k} S_k =  S_k \Delta_y .
\end{gather}
Consequently, the spectrum of $k^{-1} \Delta_{y,k}$ is $\Si_y$
for any $k$.

\subsection{Peaked sections} \label{sec:peaked-sections}

As above, we identify a neighborhood $U$ of $y$ with a
neighborhood of the origin in $T_yM$ through Darboux coordinates, we introduce
the frame $F_y$ of $L$ on $U$ with covariant derivative given by
\eqref{eq:nabla-f_y}, and we work with a trivialisation $A|_U \simeq U \times A_y$. 
Choose a function $\psi \in \Ci_0(U, \R)$ such that $\psi =1$ on a
neighborhood of $y$. Then to any polynomial $f \in \pol ( T_yM) \otimes A_y$, we associate
the smooth section $\Phi_k (f)$ of $L^k \otimes A$ defined on $U$ by 
\begin{gather} \label{eq:peaked_Section}
  \Phi_k (f) (\xi  ) =  k^{\frac{n}{2} } F^k _y (\xi) e^{ - \frac{k}{4}
  |\xi|_y^2 } f ( k^{\frac{1}{2}} \xi) \psi ( \xi) 
\end{gather}
and equal to $0$ on $M \setminus U$. 
\begin{prop} \label{prop:peaked-sections}
  We have 
  \begin{enumerate} 
\item 
  $ \| \Phi_k(f) \|^2 = \int_{T_yM} e^{-\frac{1}{2} |\xi|_y^2 } | f (\xi) |^2
  d\mu_y (\xi)   + \bigo ( e^{-C/k})   $ with $\mu_y = \om_y^n/n!$ the
  Liouville form of $T_yM$,
  \item 
 $ k^{-1} \Delta_k  \Phi_k(f)  = \Phi_k ( g) + \bigo (
 k^{-\frac{1}{2}})$ with $g  = \PPP_y (
  f  )$.
\end{enumerate}
\end{prop}

The peaked sections of \cite{oimpart1}  are defined without using the Darboux
coordinates, and for this reason the $\bigo ( e^{-k/C})$ in the norm estimate
is replaced by a $\bigo ( k^{-\frac{1}{2}})$. Actually, the Darboux
coordinates are not essential in this subsection, they only simplify slightly
some estimates, whereas in Sections \ref{sec:local-appr-resolv}, \ref{sec:globalisation} it will be necessary to use them. 

\begin{proof}  $\Phi_k (f)$ being supported in $U$, we can view it as  a
  function of $T_yM$, so  $$\Phi_k (f) =\psi  S_k (sf)  $$ where $s( \xi ) = e^{-\frac{1}{4}
    |\xi|_y^2}$ as in Section \ref{sec:laplacian-linear} and $S_k$ is the rescaling map \eqref{eq:rescaling_map}. Since we work with Darboux coordinate, the volume
  form $\mu$ of $M$ coincide on $U$ with $\mu_y$. So
  $$\| \Phi_k (f) \| ^2 = \int_{T_yM} |S_k ( sf)|^2 \psi^2 \, d
  \mu_y.$$  
We
  will need several times to estimate an integral having the form 
$$ I_k (\wt{\psi} ) = \int_{T_yM} |S_k ( sf)|^2 \wt{\psi} \, d \mu_y = k^{n}
\int_{T_yM} e^{-\frac{k}{2} |\xi|_y^2 } |f( k^{\frac{1}{2} } \xi) |^2 \, \widetilde{\psi} (
\xi )  \,d \mu_y (\xi)$$
with $\wt{\psi} \in \Ci (  T_yM)$ satisfying  $\wt{\psi} ( \xi ) = \bigo (
|\xi|^m)$ on $T_yM$ for $m \geqslant 0$. We claim that    $I_k ( \wt{\psi} ) = \bigo
( k^{-\frac{m}{2}})$ and in the case where $\wt{\psi} =0$ on a neighborhood of
the origin,   $I_k ( \wt{\psi}) = \bigo ( e^{-k/C})$
for some $C>0$.

The first claim follows from the change of variable
$\sqrt k \xi = \xi'$.  For the second one, we use that $e^{- \frac{k}{2}
  |\xi|_y^2} \wt {\psi} (\xi) = \bigo ( e^{- k/C} |\xi|^m e^{-\frac{k}{4} |\xi|_y^2})$
and do the same change of variable. 

The first assertion of the proposition is an immediate consequence of the
second claim with $\wt { \psi} = 1 - \psi^2$. For the second assertion, we
start from \eqref{eq:main_observation} and using that $[\nabla_{i,k}, \psi ] =
\partial_{x_i} \psi$ repetitively, we obtain that 
\begin{gather} \label{eq:delta_k-psi-=} 
\Delta_k  \psi = \psi \bigl ( \Delta_{y,k} + a_{ij} \nabla_i^k \nabla_j^k +
\tilde{b}_i \nabla_i  + k c + \tilde{c} )
\end{gather}
where $a_{ij}$, $c$ are the same functions as in \eqref{eq:main_observation},
$\tilde{b}_i$ and $\tilde{c}$ do not depend on $k$.  

Now, by \eqref{eq:delta_k-psi-=}, $\Delta_k ( \psi S_k ( sf))$ is a sum of
five terms, the first one being 
$$ \psi \Delta_{y,k} S_k ( fs) = k \psi S_k (
\Delta_y (sf)) = k \Phi_k (g), \quad \text{ with } sg =
\Delta_y( sf) $$
by \eqref{eq:rescale}.  
We will prove that the four other terms are in $\bigo ( k^{\frac{1}{2}})$,
which will conclude the proof.

Each time, we will apply the preliminary integral estimate with the
convenient function $\wt{\psi}$. First $|\psi \tilde{c}|$ being bounded,  $\psi  \tilde{c}\, S_k (sf)  = \bigo (1)$.
Second, $c$ vanishing at the origin, $|\psi(\xi
) c (\xi) |^2 = \bigo ( |\xi|^2)$ so that $\psi c\, S_k ( sf) = \bigo (
k^{-\frac{1}{2}})$. Third, by \eqref{eq:rescale},
$$\nabla_{i,k} S_k ( sf ) =
k^{\frac{1}{2}} S_k ( \nabla_i (sf))= k^{\frac{1}{2}} S_k ( sf_i)$$ with a new
polynomial $f_i$, and since  $\psi \tilde b_i$ is bounded, it comes that 
$$\psi  \tilde b_i\,
\nabla_{i,k} S_k (sf) = \bigo ( k^{\frac{1}{2}}).$$  Similarly, 
$ \nabla_{i,k}
\nabla_{j,k} S_k (sf) = k S_k ( s f_{ij})$
with new polynomials $f_{ij}$, and
$a_{ij}$ vanishing at the origin, we obtain $$\psi a_{ij} \nabla_{i,k}
\nabla_{j,k} S_k (sf) = \bigo ( k^{\frac{1}{2}}) $$
as was to be proved. 
\end{proof}

\begin{theo} \label{theo:spectr2}
  Let $(\Delta_k)$ be a family of formally self-adjoint differential operators
  of the form 
  \eqref{eq:ass_op}. 
  Then, if $\la \in \Si_y$,  there exists $C(y,\la)$ such that 
$$\op{dist} ( \la ,  \spec ( k^{-1} \Delta_k )
  ) \leqslant C(y,
\la) k^{-\frac{1}{2}}, \qquad  \forall k. $$ 
Furthermore, for any $\bs >0$, $C(y,\la)$
stays bounded when $(y, \la)$ runs over $M \times (-\infty,\bs ]$. 
\end{theo}
This proves the first assertion of Theorem \ref{theo:intro_spec}. 

\begin{proof}
  By Section \ref{sec:laplacian-linear}, any eigenvalue $\la$ of $\PPP_y$ has an
  eigenfunction $f \in \pol
  (T_yM) \otimes A_y$. Normalising conveniently $f$, we get by Proposition \ref{prop:peaked-sections},
  $$ \| \Phi_k (f ) \| = 1 + \bigo ( e^{-k/C}), \qquad  k^{-1}
  \Delta_{k} \Phi_k (f) = \la \Phi_k(f) + \bigo
  (k^{-\frac{1}{2}}) $$ 
  which proves that $\op{dist} ( \la ,  \spec (k^{-1} \Delta_k )
  ) = \bigo (k^{-\frac{1}{2}})$. To get a uniform $\bigo$ when $\la \leqslant \bs$,
remember that by the first assertion of Lemma \ref{lem:smooth}, 
we can choose $f \in \D_{\leqslant p} (T_yM) \otimes A_y$ where $p$ is sufficiently large
and independent of $y \in M$. Furthermore, for any $p \in \N$, the $\bigo$'s in Proposition
\ref{prop:peaked-sections} are uniform with respect to $f$ describing the
compact set $ \{ f \in \D_{\leqslant p } ( TM) \otimes A, \| f \| = 1 \}$.   Here we
can use any metric of $\D_{\leqslant p } ( TM)$, the natural one in our situation being $\| f \|^2 =
\int_{T_yM} e^{-\frac{1}{2} |\xi|_y^2 } | f (\xi) |^2
  d\mu_y (\xi)  $ for $f \in \D (T_yM)$.
\end{proof}

\subsection{A local approximated resolvent} \label{sec:local-appr-resolv}

Recall that $k^{-1} \Delta_{y,k} = S_k \Delta_y S_k^*$ so that  $k^{-1} \Delta_{y,k}$   has
the same spectrum $\Si_y$ as $\Delta_y$. For any $\la \in \C
\setminus \Si_y$, we denote by $$R_{y,k}(\la) := (\la - k^{-1}
\Delta_{y,k})^{-1} : L^2 (T_yM) \otimes A_y \rightarrow L^2 (T_yM) \otimes A_y$$ the resolvent.  
We will need the following basic elliptic estimates. 
\begin{prop}
  For any $\la \in \C \setminus \Si_y$, the resolvent $R_{y,k}(\la)$ sends $\Ci_0$ to
  $\Ci$ and satisfies
  \begin{xalignat}{2} \label{eq:estimation_model}
    \| k^{-\frac{1}{2}} \nabla_{i,k}  R_{y,k} (\la) \|
\leqslant C_\bs d^{-1} , \quad  \| k^{-1} \nabla_{i,k} \nabla_{j,k}  R_{y,k}
(\la)  \| \leqslant C_\bs d^{-1} 
\end{xalignat}
if $|\la|\leqslant \bs $ with $d = \op{dist} (\la , \Si_y)$ and the constant $C_\bs$  independent of $k$.
\end{prop}
Here and in the sequel, the norm $\| \cdot \|$ is the operator norm associated
to the $L^2$-norm. 
\begin{proof} The first assertion follows from elliptic regularity: for any
  distribution $\psi$ of $T_yM$, if $( \la - k^{-1}\Delta_{y,k}) \psi$
  is smooth then $\psi$ is smooth.

  Since $R_{y,k} ( \la) = S_k R_{y,1} ( \la) S_k^*$ and $k^{-\frac{1}{2}}
  \nabla_{i,k} = S_k \nabla_i S_k^*$, it suffices to prove the inequalities
  \eqref{eq:estimation_model} for $k=1$. We can assume that
the frame  $(\partial/\partial x_i)$ is $g$-orthonormal at $y$, so $g^{ij}_y =
  \delta_{ij}$, so $\Delta_y = - \frac{1}{2} \sum_i \nabla_i^2 + V(y) $. Since
  $\langle \Delta_y u, u \rangle = \frac{1}{2} \sum \| \nabla_i u \|^2 +
  \langle V(y) u, u \rangle $, we have by Cauchy-Schwarz inequality
\begin{gather} \label{eq:ineq1}
 \| \nabla_i u \|^2 \leqslant C \| u \| ( \| \Delta_y u \| + \| u \|)
\end{gather}
Since $[\nabla_i, \nabla_j] = \frac{1}{i} \om_{i,j}$, we have
\begin{xalignat}{2} \notag
& \| \nabla_i \nabla_j u \|^2 = \langle \nabla_j \nabla_i^2 \nabla_j u, u
\rangle = \langle \nabla_i^2 \nabla_j^2 u,u  \rangle +  \tfrac{2}{i} \om_{ji}
\langle \nabla_i \nabla_j  u, u \rangle  \\
&   \label{eq:local-appr-resolv-1} = \langle \nabla_j^2 u, \nabla_i^2 u  \rangle +  \tfrac{2}{i} \om_{ij}
\langle  \nabla_j  u, \nabla_i u \rangle
\end{xalignat}
Moreover,
\begin{gather} \label{eq:tfrac14-sum_i-j}
\tfrac{1}{4} \sum_{i,j} \langle \nabla_i^2 u , \nabla_j^2 u \rangle =
\| \Delta_y u - V(y) u \|^2 \leqslant C ( \| \Delta_y u \| +  \| u \| )^2 .
\end{gather}
Estimating the first term of \eqref{eq:local-appr-resolv-1} with
\eqref{eq:tfrac14-sum_i-j} and the second one with
\eqref{eq:ineq1}, it comes that
\begin{gather} \label{eq:ineq2}
\| \nabla_i \nabla_j u \|^2 \leqslant C ( \| \Delta_y u \| + \| u \| )^2 
\end{gather}

To conclude the proof, we use that the norm of $R_y( \la) = ( \la -
\Delta_y)^{-1}$ is $d ^{-1}$ and $\Delta_y R_y ( \la) = \la R_y ( \la) -
\op{id} $ so when $|\la | \leqslant \bs$, 
$$\| \Delta_y R_{y} ( \la) \|
\leqslant \bs d^{-1} + 1 \leqslant C_\bs d^{-1}$$ because $d$ stays bounded
when $\la$ is. Hence it follows from \eqref{eq:ineq1} and \eqref{eq:ineq2} that
$$ \| \nabla_i R_y ( \la) v \| \leqslant C_\bs d^{-1}\| v\| , \qquad \| \nabla_j \nabla_i R_y (
\la) v\| \leqslant C_\bs d^{-1} \|v \|  $$
which corresponds to \eqref{eq:estimation_model} for $k=1$. 
\end{proof}

Recall that we identified a neighborhood  of $y \in M$ with a neighborhood $U$
of
the origin of $T_yM$ through Darboux coordinates. Introduce a smooth function $\chi : T_yM  \rightarrow [0,1]$ such that $\chi ( \xi)
= 1$ when $|\xi| \leqslant 1$ and $\chi ( \xi) =0$ when $|\xi| \geqslant 2$. Define $\chi_r
(\xi) := \chi (\xi /r)$. In the sequel we assume that $r$ is sufficiently small so
that $\chi_r$ is supported in $U$. Then for any differential operator $P$
acting on $\Ci ( U )$,
$\chi_r P$ and $P \chi_r$ are differential operators with  coefficients supported
in $U$, so we can view them as operators acting on $\Ci ( T_yM)$. 

In the following Lemma, we prove that the resolvent $R_{y,k} (\la)$ of $k^{-1}
\Delta_{y,k}$ is a local right-inverse of $ ( \la - k^{-1} \Delta_k)$ up to some error.

\begin{lemme} \label{lm:main-computation}
  For any $\la \in \C \setminus \Si_y$ such that $|\la |
  \leqslant \bs $, we have with $d = d( \la ,
  \Si_y)$  
\begin{gather} \|  ( \la - k^{-1} \Delta_k)  \chi_r R_{y,k} ( \la) - \chi_r \|
  \leqslant C_\bs F(r,
 k^{-1}, d)
\end{gather}
 where $F(r,\hbar, d) =  ( r + \hbar^{1/2} + \hbar r^{-2} + \hbar^{1/2} r^{-1}
 )d^{-1} $.
\end{lemme}
\begin{proof} 
We compute
\begin{xalignat}{2} \notag &  (\lambda - k^{-1} \Delta_k ) \chi_r R _{y,k}( \la) -
  \chi_r \\ \notag
  = &  -  k^{-1} [ \Delta_k , \chi_r ] R_{y,k}
  ( \lambda) + \chi_r ( \lambda - k^{-1} \Delta_k ) R_{y,k}(\lambda) - \chi_r \\
  \label{eq:deux_termes}
= &   -  k^{-1}  [ \Delta_k , \chi_r ] R_{y,k}
( \lambda) + \chi_r k^{-1} ( \Delta_{y,k}  -  \Delta_{k} ) R_{y,k} ( \lambda)
\end{xalignat}
To estimate the first term, we start from assumption \eqref{eq:ass_op}, which gives us 
\begin{xalignat*}{2}
  [\Delta_k, \chi_r ] & = - \tfrac{1}{2}g^{ij} [\nabla_{i,k} \nabla_{j,k} , \chi_r ] + a_j [
\nabla_{j,k}, \chi_r ] \\ & =  -  \tfrac{1}{2} g^{ij} \bigl( (\partial_j \partial_i \chi_r) + (\partial_i
\chi_r)\nabla_{j,k} + ( \partial_j \chi_r) \nabla_{i,k} \bigr) + a_j (\partial_j
\chi_r)  .
\end{xalignat*}
Applying the estimates \eqref{eq:estimation_model}, we deduce that
\begin{xalignat*}{2}  \| k^{-1}  [  \Delta_k , \chi_r ] R_{y,k}
( \lambda) \|  & \leqslant C ( k^{-1} r^{-2} d^{-1} + k^{-\frac{1}{2}} r^{-1} d^{-1} +
k^{-1} r^{-1} d^{-1}) \\ & \leqslant  C  ( k^{-1} r^{-2} +k^{- \frac{1}{2}} r^{-1}
) d^{-1}
\end{xalignat*}
To estimate the second term of \eqref{eq:deux_termes}, we use the expression  \eqref{eq:main_observation} and the fact that the $a_{ij}$'s
and $c$ vanish at
the origin so that $|\chi_r a_{ij} | \leqslant Cr$ and $| \chi_r c| \leqslant Cr$.
By \eqref{eq:estimation_model} it follows that 
\begin{xalignat*}{2}
\| \chi_r  k^{-1}  ( \Delta_k - \Delta_{y,k} )R_{y,k} ( \lambda)\| & \leqslant C ( r 
+ k ^{-\frac{1}{2}}  + k^{-1} ) d^{-1} \leqslant C ( r 
+ k ^{-\frac{1}{2}}  ) d^{-1}
\end{xalignat*} 
which concludes the proof. 
\end{proof}

\subsection{Globalisation} \label{sec:globalisation}

The local approximation of the resolvent at $y$ in the previous section was based on a
choice of Darboux coordinates. To globalise this, we will first choose such
coordinate charts depending smoothly on $y$. 
All the constructions to come depend on an auxiliary Riemannian metric. 
For any $y \in M$ and $r>0$ let $B_y (r)$ be the open ball $\{ \xi \in T_yM,
\; \| \xi \| < r\}$.

\begin{lemme} \label{lem:Darboux}
  There exist $r_0 >0$ and a smooth family of embeddings $(\Psi_y:
  B_y (r_0) \rightarrow M, \; y \in M)$ such that for any $y
  \in M$, $\Psi_y (0) = y$, $T_0 \Psi_y =
  \op{id}_{T_yM}$ and $\Psi_y^* \om$ is constant on $B_y (r_0)$.  
\end{lemme}
The family $(\Psi_y , \; y \in M)$ is smooth in the sense that the map $ \Psi
(\xi) = \psi_y ( \xi)$, $ \xi \in B_y(r_0)$, from the open set $\bigcup_{y \in M}  B_y (r_0)$
of $TM$ to $M$, is smooth.

\begin{lemme} \label{lem:covering}
  There exists $N \in \N$, $r_1 >0$ and for any $0<r<r_1$ a finite
  subset $I(r)$ of $M$ such that the open sets $\Psi_y ( B_y( r)), \; y \in
  I(r)$ form a covering of $M$ with multiplicity bounded by $N$.
\end{lemme}
The multiplicity of a covering $\bigcup_{i \in I} U_i \supset M$ is the
maximal number of $U_i$ with non-empty intersection. The proofs of Lemmas \ref{lem:Darboux}
and \ref{lem:covering} are standard and postponed to Section \ref{sec:misceleanous-proofs}.

Recall that $\Si = \bigcup \Si_y$. So for any $\la \in \C \setminus \Si$, the
resolvents $R_{y,k} ( \la) : \Ci_0 ( T_yM, A_y) \rightarrow \Ci  ( T_yM, A_y)$ are
well-defined. As previously, introduce  a section $F_y$ of $L \rightarrow
\Psi_y ( B_y (r))$ satisfying \eqref{eq:nabla-f_y} and a trivialisation of $A$
on $\Psi_y ( B_y(r))$, from which we identify $\Ci ( \Psi_y ( B_y (r)),
L^k \otimes A )\simeq \Ci ( B_y(r), A_y)$. Let
$$ \wt{R}_{y,k} ( \la) : \Ci_0  (
\Psi_y ( B_y (r)), L^k \otimes A ) \rightarrow   \Ci  (
\Psi_y ( B_y (r)),L^k \otimes A).$$
be the map corresponding to $R_{y,k} (\la)$ under these identifications. 

For $r$ sufficiently small, define the function $\chi_{y,r}$ supported in
$\Psi_y ( B_y (r_0))$ and such that $\chi_{y,r} ( \Psi_y(\xi)) =  \chi ( \xi /r)$.
Introduce a partition of unity $(\psi_{r,y}, \; y \in I(r))$ subordinated to
the cover $(\Psi_y ( B_y(r)), \; y \in I(r))$. Then define the operator $R_k^r
(\la) $ acting on $\Ci ( M , L^k \otimes A )$ by
\begin{gather} \label{eq:def_resolvente_approchee}
R^r_k (\la) := \sum_{y \in I(r) } \chi_{y, r } \wt{R} _{y,k} ( \la) \psi_{r,y} .
\end{gather}

\begin{theo} \label{theo:resolv}
 Let $(\Delta_k)$ be a family of formally self-adjoint differential operators
  of the form 
  \eqref{eq:ass_op}.  Then for any $|\la | \leqslant \bs$,   
\begin{gather} \label{eq:global_estimate}
\| ( \la - k^{-1} \Delta_k) R^r_{k} ( \la) - 1 \| \leqslant C_\bs F( r, k^{-1},
d)  
\end{gather}
with $d = \op{dist}(\lambda, \Si )$ and $F$ the same function as in Lemma \ref{lm:main-computation}. 
\end{theo}
\begin{proof} Let $(U_i)$ be a covering of $M$ with multiplicity $N = \sup _x
  | \{ i / x \in U_i \}|$. Then
  \begin{enumerate}
    \item if $v_i$ is a family of sections such that
  $\op{supp} v_i \subset U_i$ for any $i$,  then $\| \sum v_i \| ^2 \leqslant N \sum \|
  v_i \|^2$
\item For any section $u$, $\sum \| u \|_{U_i}^2 \leqslant N \| u \|^2$. 
\end{enumerate}
To prove the first claim, $\| \sum v_i \| ^2 = \sum_{i,j} M_{ij} \langle v_i
, v_j \rangle \leqslant \sum M_{ij} \| v_i \| \| v_j\|$ where $M_{i,j} = 1$
when $U_i \cap U_j \neq \emptyset$ and $0$ otherwise. By Schur test applied to
the matrix $M$, $\langle M a , a \rangle \leqslant N \| a\|^2$ and the
result follows. 
To prove the second claim, set $m(x) = \sum 1_{U_i} ( x)$ which is bounded by
$N$ by assumption. Then $\sum \| u \|_{U_i}^2 = \int_M |u(x)|^2 m(x) d\mu (x)
\leqslant N \| u \|^2$. 

We now apply this to the covering $\Psi_y( B_{y} ( r))$, $y \in I(r)$. By
Lemma \ref{lm:main-computation}, for any $u \in \Ci ( M , L^k)$, we have  $\| S^r_{y,k}
\psi_{y,r} u \| \leqslant CF \| \psi_{y,r} u \|$
where 
$$S_{y,k}^r = 
(\la - k^{-1} \Delta_k) \chi_{y,r} \wt{R}_{y,k}  ( \la) - \chi_{y,r}$$
$F =F( r, k^{-1},
d) $ and the constant $C$ can be chosen independently of $y$ because everything depends continuously on $y$ and $M$ is compact. 
Since $R^r_k ( \lambda) - 1=  \textstyle{\sum}_{y \in I(r)} S_{y,k}
\psi_{y,r}$, we have 
\begin{gather*} 
\| R_k^r ( \la)  u -u  \| ^2 \leqslant N
  \sum_{y \in I(r)} \|
S_{y,k}^r \psi_{y,r} u \|^2 \leqslant N(C F)^2 \sum_{y \in I(r)} \| \psi_{y,r}  u \|^2 \\  \leqslant N(C F)^2
\sum_{y \in I(r)} \|  u \|_{\Psi_y(B_y(r))}^2 \leqslant (NCF)^2  \| u \| ^2 .
\end{gather*}
which proves \eqref{eq:global_estimate}.
\end{proof}

Recall basic facts pertaining to the spectral theory of $\Delta_k$, cf. as
instance \cite[Section 8.3]{Sh}. As an
elliptic formally self-adjoint differential operator of order 2 on a compact manifold, $\Delta_k$  is a
self-adjoint unbounded operator with domain the Sobolev space $H^2 (M, L^k \otimes A)$. Its spectrum $\op{sp} ( \Delta_k)$ is a discrete subset of $\R$
bounded from below and consists only of eigenvalues with finite multiplicities.
\begin{cor} \label{cor:spectrum}
For any $\bs >0$, there exists $C>0$ such that for any $k$ we have
\begin{gather} \label{eq:conclusion}
\op{sp} ( k^{-1} \Delta_k ) \cap (-\infty, \bs ] \subset \Si + C k^{-\frac{1}{4}} [-1,
1].
\end{gather}
So any $\la \in \C$, satisfying $|\la| \leqslant \Lambda$
  and $d(\la, \Si ) \geqslant C k^{-\frac{1}{4}} $, does not belong to $ \op{sp} (
  k^{-1} \Delta_k)$. Moreover, for any such $\lambda$
  \begin{gather} \label{eq:resol_est} \| R_k^{r_k}  ( \la) - ( \la - k^{-1} \Delta_k)^{-1}  \| \leqslant C
d(\la, \Si)^{-2}   k^{-\frac{1}{4}} 
\end{gather}
with $r_k = k^{-\frac{1}{4}}$.
\end{cor}
\eqref{eq:conclusion} shows the second assertion of Theorem
\ref{theo:intro_spec} with a $k^{-\frac{1}{4}} $ instead of
$k^{-\frac{1}{2}}$. The improvement with a $k^{-\frac{1}{2}}$ will be proved
in Corollary \ref{cor:glob-spectr-estim}.

\begin{proof} First, since $\| \wt{R} _{y,k} ( \la) \| \leqslant d ( \la, \Si_y)^{-1}
  \leqslant d^{-1}$ with $d = d ( \la , \Si)$, we deduce from the first part of the proof of Theorem \ref{theo:resolv} that
  \begin{gather} \label{eq:-rr_k-}
    \| R^r_k ( \la) \| \leqslant C d^{-1} 
  \end{gather}
  where $C$ does not depend on $r$, $\la$ and $k$. From now on assume that $r
  = k^{-\frac{1}{4} }$.  So $F(r,k^{-1}, d) \leqslant C'
  k^{-\frac{1}{4} } d^{-1}$.
  By Theorem \ref{theo:resolv}, as soon as $ C_\bs C' k^{-\frac{1}{4}} d^{-1}
  \leqslant 1/2$, $( \la - k^{-1} \Delta_k) R^r_{k} ( \la)$ is invertible, so 
$\widetilde{R}_k := R^r_k ( \la) (( \la - k^{-1} \Delta_k) R^r_{k} ( \la)
)^{-1} $ is a bounded operator of $L^2$ satisfying
\begin{gather} \label{eq:inverse_a_droite}
  (\la - k^{-1} \Delta_k ) \widetilde{R}_k = \op{id}
\end{gather}
and by \eqref{eq:-rr_k-},
$$ \| \widetilde{R}_k - R^r_k ( \la) \| \leqslant 2 \| R^r_k ( \la) \| \; \| (
\la - k^{-1} \Delta_k) R^r_{k} ( \la) -1 \| \leqslant C''d^{-2} k^{-\frac{1}{4}} .$$
We claim
that $\widetilde{R}_k$ is actually continuous $L^2 \rightarrow
H^2$. Indeed, by classical result on elliptic operators \cite[Theorem 5.1]{Sh}, there exists a
pseudodifferential operator $P_k$ of order $-2$ which is a parametrix of $\la - k ^{-1}
\Delta_k$, that is $ P_k(\la - k^{-1} \Delta_k) = \op{id} + S_k $ where $S_k$ is a
smoothing operator. Then
multiplying by $\widetilde{R}_k$, we obtain $P_k = \widetilde{R}_k + S_k \widetilde{R}_k$, so
$\widetilde{R}_k = P_k - S_k \widetilde{R}_k$. Now, $P_k$ being of order
$-2$ and $S_k$ being smoothing, they are both continuous $L^2 \rightarrow H^2$,
so the same holds for $\widetilde{R}_k$. 

To finish the proof, we assume that $\la$ is real. Then $k^{-1}  \Delta_k
-\la$ is a Fredholm operator from $H^2
$ to $ L^2$ with index 0, because it is formally self-adjoint, cf.
\cite[Theorem 8.1]{Sh}. By \eqref{eq:inverse_a_droite}, $\la - k^{-1} \Delta_k$ sends $H^2$
onto $L^2$, so its kernel is trivial, so $\la$ is not an eigenvalue.  
\end{proof}

\section{The operator class $\lag (A)$} \label{sec:operator-class-lag}

\subsection{Symbol spaces} \label{sec:symbol-spaces}

Let $\EE$ be a $n$-dimensional Hermitian space. As we did in Section 
 \ref{sec:laplacian-linear} for $\EE = T_yM$, consider the spaces $ \pol ( \EE)$, $\D (\EE)$
consisting respectively of polynomial maps and antiholomorphic polynomial
maps from $\EE$ to $\C$. We will introduce two subalgebras $\symb (\EE)$ and
$\wt {\symb} (\EE)$ of $\op{End} ( \D (\EE))$ and $\op{End} ( \pol ( \EE))$
respectively. These algebras will be used later to define the symbols of the
operators in the class $\mathcal{L}$.

First we equip $\pol (
\EE)$ with the scalar product 
\begin{gather} \label{eq:scal_product_bargm}
 \langle f, g \rangle = (2\pi)^{-n}\int_{\EE} e^{-|z|^2} f (z)\, \con{g (z)} \; d \mu_{\EE}
(z) , 
\end{gather}
where $\mu_{\EE}$ is the measure $\prod dz_i d \con{z}_i$ if $(z_i)$
are linear complex coordinates associated to an orthonormal basis of $\EE$.
The Gaussian weight $e^{-|z|^2}$ appeared already in Section
\ref{sec:laplacian-linear} through the pointwise norm of the frame $s = \exp
(-\frac{1}{2} |z|^2)$.

Choose linear complex coordinates $(z_i)$ as above. Then the family $|\al \rangle := (\al !)^{- \frac{1}{2}} \con{z}^\al$, $ \al \in \N^n$ is an orthonormal
basis of $ \D ( \EE)$. For any $\al, \be \in \N^n$, introduce the endomorphism $ \Uu_{\al \be} := | \al \rangle
\langle \be |$ of $\D ( \EE)$.  Here we use the physicist notation, so $\Uu_{\al \be} ( \con{z}^\ga ) = 0 $ when $\ga \neq \be$
and $ \Uu_{\al\be} \bigl( |\be \rangle  \bigr) = |\al
\rangle $.

Consider the creation and annihilation operators $\ac_i$, $\ac_i^{\dagger}$
defined in \eqref{eq:a_i_def} as endomorphisms of $\pol (\EE)$. Note that
with the scalar product \eqref{eq:scal_product_bargm}, $\ac_i^{\dagger}$ is
the formal adjoint of $\ac_i$. 
Introduce the endomorphism $\Vvv_{\al \be}$ of $\pol (\EE)$
$$ \Vvv_{\al\be} := ( \al! \be ! ) ^{-\frac{1}{2}} (\ac^\dagger)^{\al} \Vvv_{00} \ac^\be$$
where $\ac^\be = \ac_1^{\be(1)} \ldots \ac_n^{\be (n)}$, $(\ac^\dagger)^{\al} =
(\ac_1^\dagger)^{\al (1)} \ldots (\ac_n ^\dagger)^{\al (n)}$ and $\Vvv_{00}$ is
the orthogonal projector onto the subspace $\mathfrak{L}_0$ of $\pol (\EE)$
consisting of holomorphic
polynomials. 

Observe that the restriction of $\Vvv_{\al \be}$ to $\D (\EE)$ is $\Uu_{\al
  \be}$. Furthermore, in the decomposition into orthogonal subspaces  $\pol ( \EE) = \bigoplus_{\al}
\mathfrak{L}_{\al} $ considered in \eqref{eq:dec_spectral}, $\Vvv_{\al\be}$ is
zero on $\mathfrak{L}_{\ga}$ with $\ga \neq \be$ and restricts to an
isomorphism from $ \mathfrak{L}_\be$ to $\mathfrak{L}_{\al}$. Also
$\Vvv_{\al \al}$ is the orthogonal projector onto $\mathfrak{L}_{\al}$. 

The algebras $\symb (\EE)$ and $\wt{\symb } (\EE)$ are defined as the
subalgebras of $ \op{End} ( \D(\EE))$ and $\op{End} (\pol ( \EE))$ with basis
the families $(\Uu_{\al,\be}, \al , \be \in \N^n)$ and $(\Vvv_{\al \be}, \;
\al, \be \in \N^n)$ respectively. As the notations suggest, these algebras do
not depend on the coordinate choice. This follows  from the following Schwartz
kernel description. 

Let $\op{Op} : \pol (\EE) \rightarrow
\op{End} ( \pol (\EE))$  be the linear map defined by
\begin{gather} \label{eq:def_op}
\op{Op} (q) ( f) (u)   = (2\pi)^{-n} \int_{\EE} e^{u\cdot \con v - |v|^2} q (u-v ) f(v) \;
d\mu_\EE (v)  
\end{gather}
where $u \cdot \con v$ is the scalar product of $u$ and $v$. By \cite[Lemma
4.3]{oimpart1},  $\Vvv_{\al,\be}= \op{Op} (p_{\al,\be})$ 
where $p_{\al\be}$ is the polynomial
\begin{gather} \label{eq:p_al-be_basis}
p_{\al,\be} := ( \al ! \be !)^{-\frac{1}{2}} (\con{z}-  \partial_z )^{\al}
(-z)^\beta, \qquad \al, \be \in \N^n
\end{gather} 
Since these polynomials form a basis of $\pol (\EE )$, $\op{Op}$ is an isomorphism form $\pol ( \EE)$ to $\wt{\symb} (
\EE )$. Furthermore, the map sending $q \in \pol (\EE)$ to $\op{Op} (q)|_{\D ( \EE)} $
is an isomorphism from $\pol ( \EE) $ to $\symb (\EE)$.

In the sequel we will tensor the space $\pol ( \EE)$ with an auxiliary vector space $\mathbb{A}$ and extend
the map $\op{Op}$ from $\pol ( \EE) \otimes \op{End} \mathbb{A}$ to $\wt{\symb} (
\EE ) \otimes \op{End} \mathbb A$.

\subsection{Eigenprojectors of Landau Hamiltonian}

Choose now $\EE = T_yM$ and  recall that for a convenient choice of complex
coordinate $(z_i)$,  the associated Landau
Hamiltonian $\PPP_y$ is given by
\begin{gather} \label{eq:laplacien_creat_annihil}
\PPP_y =e^{\frac{1}{4}|\xi|_y^2} \Delta_y  e^{-\frac{1}{4}|\xi|_y^2} = \sum B_i
(y) (\ac_i^\dagger
\ac_i + \tfrac{1}{2} )  + V(y)
\end{gather}
acting on $\pol ( T_yM) \otimes A_y$. 
Its spectrum $\Si_y$ and its eigenspaces were described in Section
\ref{sec:a_y-valued-laplacian} in terms of the $\mathfrak{L}_{\al}$ and an
eigenbasis $(\zeta_\ell)$ of $V(y)$, $V(y) \zeta_\ell = V_{\ell} (y) \zeta_{\ell}$. Consequently if $I$ is any bounded subset of $\R$, the spectral projector of
$\PPP_y$ for the eigenvalues in $I$ is $\op{Op} (
\si^I(y))$ where
\begin{gather*}
\si^I(y) = \sum_{(\al, \ell) \in \mathcal{I}_y} p_{\al \al} \otimes |\zeta_{\ell}
\rangle \langle \zeta_{\ell} |,
\end{gather*}
and $\mathcal{I}_y = \bigl\{ (\al, \ell) \in
\N^{n}\times\{ 1, \ldots, r \}/ \textstyle{\sum}_i B_i
(y) ( \al (i) + \tfrac{1}{2} ) + V_{\ell} (y) \in I \bigr\}$.

The map $ y \mapsto \si^I(y)$ is a section
of the infinite rank vector bundle $\pol (TM)$, not smooth in general, not even
continuous. In the sequel we will assume that
\begin{gather} \label{eq:ass_I} \tag{C}
  \text{$I$ is a compact interval with endpoints not belonging to $\Si$} 
\end{gather}
Let $\pol_{\leqslant p } (\EE)$ be the subspace of
$\pol (\EE)$  of polynomials with degrees in $z$ and in
$\con{z}$ smaller than $p$. Let $\pol_{\leqslant p} (TM) $ be the vector bundle
over $M$ with fiber at $y$ equal to  $\pol_{\leqslant p } (T_yM )$.

\begin{lemme} \label{lem:smooth_proj}
  If $I$ satisfies \eqref{eq:ass_I} and $p$ is sufficiently large, then  $y \mapsto \si^I (y)$ is a
  smooth section of $\pol_{\leqslant p } ( TM) \otimes \op{End} A$. 
\end{lemme}

\begin{proof}
Recall from Section \ref{sec:restr-delt-anti} that $\PP_y$ is the restriction of  $\PPP_y$ to $\D
(T_yM)$.    By Lemma  \ref{lem:smooth}, the spaces 
\begin{gather} \label{eq:def_F}
  F_y:= \op{Im} 1_I (\PP_y)  = \op{Span}
  ( \con{z}^\al \otimes \zeta_{\ell} , (\al, \ell) 
  \in \mathcal{I}_y)
\end{gather}
are the fibers of a subbundle of $  \D_{\leqslant p } (TM) \otimes A$ if $p$ is
sufficiently large. So the projector onto $F_y$ depends smoothly on $y$, in
other words,  the map $y \rightarrow  \op{Op} ( \si^I(y)) |_{\D (T_yM) \otimes
  A_y}$ is a smooth section of $ \op{End} ( \D _{\leqslant p } ( TM) \otimes A )$.

Now  we have an isomorphism 
\begin{gather*}
\pol_{\leqslant p }  ( \EE) \xrightarrow{\op{Op}_p}
  \op{End} ( \D _{\leqslant p } ( \EE) ), \qquad q \mapsto \text{the
    restriction of $\op{Op}(q)$ to $ \D _{\leqslant p } ( \EE)$}.
\end{gather*}
Indeed,  on one hand $(p_{\al \be}$, $|\al|, | \be| \leqslant p)$ is a basis of
  $\pol_{\leqslant p }  ( \C^n)$ and on the other hand $(\Uu_{\al \be}$, $|\al|, |\be| \leqslant p)$ is a basis of $\op{End}
  \mathcal{D}_{\leqslant p}(\C^n)$.
This gives a vector bundle isomorphism $\pol_{\leqslant p }  ( TM) \otimes
\op{End} A  \simeq \op{End} ( \D _{\leqslant p } ( TM) \otimes A )$, and concludes the proof. 
\end{proof}

Let  $\symb (TM) $ be the infinite rank vector bundle over $M$ with fibers
$\symb (T_yM)$ defined as in Section \ref{sec:symbol-spaces}. A
 section $U$ of $\symb (TM) \otimes \op{End} A$ is {\em smooth} if it has the form
\begin{gather}  \label{eq:symb_Smooth_op}
  U (y) = \op{Op} ( q (y))|_{\D (T_yM) \otimes A_y} 
\end{gather}
where $y \rightarrow q
 (y)$ is a smooth section of $\pol_{\leqslant p } ( TM) \otimes \op{End} A$ for some $p$. 
By Lemma \ref{lem:smooth_proj}, for any interval $I$ satisfying \eqref{eq:ass_I}, we have a
symbol $\pi^I \in \Ci ( M,  \symb (TM) \otimes \op{End} A)$ defined at $y$ by 
\begin{gather} \label{eq:def_pi_I} \pi^I( y ) = 1_I
(\PP_y) = \op{Op} (
\si^I(y))|_{\D (T_yM) \otimes A_y}
\end{gather}
which is the projector of $\D (T_y M) \otimes A_y$ onto the
subspace $F_y$ defined in Lemma \ref{lem:smooth}.

\subsection{Operators} \label{sec:operators}

The operator
class $\lag (A)$ was introduced in \cite{oimpart1}. It depends on $(M, \om , j)$, the prequantum
bundle $L$, that is $L$ with its metric and connection, and the auxiliary
Hermitian bundle $A$.

 $\lag (A) $ consists of families of operators $(P_k : \Ci (M, L^k \otimes A )
\rightarrow \Ci ( M , L^k \otimes A), \; k \in \N)$ having smooth Schwartz
kernels satisfying the following conditions. First, $P_k(x,y)$ is in $\bigo (
k^{-\infty})$ outside the diagonal. More precisely, for any compact subset $K$ of
$M^2 \setminus \op{diag} M$ and for any $N$, there exists $C>0$ such that
$$ |P_k (x,y) | \leqslant C k^{-N} , \qquad \forall k \in \N, \; \forall (x,y) \in K .$$
Second, for any open set $U$ of $M$ identified through a diffeomorphism with a convex open set of
$\R^{2n}$ and any unitary trivialisation $A|_U \simeq U \times \C^r$,  we have
on $U^2$ for any positive integers $N$, $k$
\begin{xalignat}{2} \label{eq:expansion}
\begin{split} P_k (x+ \xi ,x) = & \Bigl( \frac{k}{2\pi} \Bigr)^n F^{k} (x+ \xi , x)
e^{-\frac{k}{4} |\xi|_x^2} \sum_{\ell = 0 }^{ N} k^{-\ell} a_{\ell} (x,
k^{\frac{1}{2}} \xi ) \\ & + r_{N,k} (x+ \xi, x)  
\end{split}
\end{xalignat}
where the section $F: U^2 \rightarrow L \boxtimes \con{L} $ is defined as in
Section \ref{sec:schw-kern-spectr}, the coefficients $a_{\ell}(x,\xi) \in \C^r \otimes \con{\C}^r$ depend smoothly on $x$ and
polynomialy on $\xi$, with degree bounded independently of $x$, and the
remainder $r_{N,k}$ is in $\bigo (  k^{n-\frac{N+1}{2}})$ uniformly on
any compact subset of $U^2$.

The subspace $\lag ^+(A) $ of $\lag(A) $ consists of the operator families
$(P_k)$ where the coefficients $a_{\ell}$ in the local expansions
\eqref{eq:expansion} satisfy
$a_\ell (x, - \xi) = (-1)^\ell a_\ell (x,\xi)$.
The symbol map is the application $\si_0 : \lag \rightarrow
\Ci (M, \symb (TM) \otimes \op{End} A)$ given locally by
\begin{gather} \label{eq:sigma_0}
\si_0 (P) (x) = \op{Op} ( a_0 ( x, \cdot) )|_{\D (T_xM)}  \in \symb ( T_xM )
\otimes \op{End} A_x
\end{gather}
where we view $a_0 (x, \xi)$ in $ \C^r \otimes \con{\C^r} \simeq \op{End} \C^r
\simeq \op{End} A_x$. 

Recall that for any compact interval $I$ of $\R$, we  denote by $\Pi^I_k$ the corresponding
spectral projector of $k^{-1} \Delta_k$. The central result of this paper is
the following theorem.

\begin{theo} \label{theo:main_result_first}  Let $(\Pi^I_k)$ be the spectral
  projector of a formally self-adjoint operator family $(\Delta_k)$ of the form
 \eqref{eq:ass_op} with $I$ satisfying \eqref{eq:ass_I}.
 Then
  $(\Pi^I_k)$ belongs to $\mathcal{L}^+ (A)$ and has symbol $\pi^I$. 
\end{theo}
The proof is given in Section \ref{sec:proof-theor-refth}. We will actually prove a stronger result where we describe the Schwartz kernel
derivatives as well. 

\subsection{The class  $\mathcal{L}^{\infty} (A)$} \label{sec:class-mathcallinfty} 
We need first a few definitions.
Consider a real number $N$. We say that a sequence $(f_k)$ of $\Ci ( U)$ with $U$ an open set of
$M$ is in $\biginf (k^{-N})$ if for any $m \in \N$, for any vector fields $X_1,
\ldots , X_m$ of $U$, for any compact subset $K$ of $U$, there exists $C>0$
such that
$$ |X_1 \ldots X_m f_k (x) | \leqslant C k^{-N + m } , \qquad \forall x \in K,
\; k \in \N.$$
Let $s =(s_k \in \Ci ( M , L^k \otimes A), \; k
\in \N)$. We say that $s \in \biginf ( k^{-N})$ if for
any unitary frames $u$ and $(v_j)_{j=1}^{r}$ of $L$ and $A$ defined over the
same open set $U$ of $M$, the local representative sequences $(f_{k,j})$ such
that $s_k = \sum f_{j,k} u^k \otimes v_j$, are in $\biginf (k^{-N})$.
We say that $s$ belongs to $\biginf (k^{\infty})$ (resp. $\biginf (k^{-\infty})$)
if $s \in \biginf (k^{-N})$ for some $N$ (resp. for any $N$). So
$$ \biginf ( k^{-\infty}) \subset \biginf (k^{-N}) \subset \biginf ( k^{-N'})
\subset \biginf ( k^{\infty}) , \qquad \text{ if } N \geqslant N'$$
Replacing $M$, $L$ and $A$ by $M^2$, $L\boxtimes \con{L}$ and $A \boxtimes
\con{A}$, we can apply these definitions to Schwartz kernels of operator
families $(P_k : \Ci (M, L^k \otimes A )
\rightarrow \Ci ( M , L^k \otimes A), \; k \in \N)$.

By definition, $\lag ^{\infty}(A)$ and $\lag_{\infty}^\infty(A) $ are the subspaces
of $\lag(A)$ consisting of operator families with a Schwartz kernel in $\biginf (
k^{\infty})$ and $\biginf ( k^{-\infty})$ respectively. By \cite[Proposition 6.3]{oimpart1}, the difference between
$\lag^{\infty}(A)$ and $\lag (A)$ is rather small because for any $P \in \lag(A)$, there exists $P'
\in \lag^{\infty}(A)$ such that the Schwartz kernels of $P-P'$ is in $\bigo (
k^{-\infty})$, that is  $P_k(x,x') =   P'_k
(x,x') + \bigo (k^{-N})$ for any $N$, with a $\bigo$ uniform on $M^2$.
Furthermore $P'$ is unique modulo $\lag^{\infty}_{\infty} (A)$. 

By  \cite[Proposition 6.3]{oimpart1}, for any $(P_k) \in
\lag^{\infty}(A)$ the asymptotic expansion \eqref{eq:expansion} holds with a
remainder $r_{N,k} $ in $\biginf (k^{n - \frac{N+1}{2}})$. 

\begin{theo} \label{theo:main_result_second}
Under the same assumptions as in Theorem \ref{theo:main_result_first}, 
  $(\Pi^I_k)$ belongs to $\mathcal{L}^{\infty}(A)$. 
\end{theo}

The proof will be given in Section \ref{sec:proof-theor-refth}. 
To end this section, let us state the following corollary of Theorems
\ref{theo:main_result_first}, \ref{theo:main_result_second} and Lemma
\ref{lem:comp_Delta}.
\begin{cor} \label{cor:dsed}
Under the same assumptions as in Theorem \ref{theo:main_result_first}, 
  $(k^{-1} \Delta_k \Pi^I_k)$ belongs to $\lag^+ (A) \cap  \lag^{\infty}(A)$
  and has symbol $\si_0 ( k^{-1} \Delta_k \Pi_k) = \PP \circ \pi^I$. 
\end{cor}
So the first part of Theorem  \ref{theo:intro_proj} follows from Theorem
\ref{theo:main_result_first} and Corollary \ref{cor:dsed}.

\section{Proof of Theorems \ref{theo:main_result_first} and
  \ref{theo:main_result_second}} \label{sec:proof-theor-refth}

The first step, Lemma
  \ref{lem:first-approximation},  is to show that any
operator in $\lag (A)$ with symbol
$\pi^I$ is an approximation of $\Pi^I_k$ up to a $\bigo ( k^{-\frac{1}{4}})$.
  This will follow from the resolvent estimate given in \ref{cor:spectrum} and
  the Cauchy-Riesz formula. The second step, Lemma \ref{lem:formal-projector},  is the construction of a formal projector $(P_k)
  \in \lag^{+}(A)$ with symbol $\pi^I$ which almost commutes with $\Delta_k$. The third
  step, Section  \ref{sec:end-proof-theorem}, is to show that this formal
  projector $(P_k)$ is equal to $\Pi^I_k$ up to a $\bigo (k^{-\infty})$ and
  even up to a $\biginf (k^{-\infty})$ when $(P_k) \in \lag^{\infty}(A)$. 

\subsection{A first approximation} \label{sec:first-approximation}

\begin{lemme} \label{lem:first-approximation}
Under the same assumptions as in Theorem \ref{theo:main_result_first},  $\Pi_k^I = P_k + \bigo (
k^{-\frac{1}{4}})$ for any $(P_k)$ in $\lag (A)$ with symbol $\pi^I$.
\end{lemme}

\begin{proof} {\em Step 1}.
  The proof starts from the resolvent approximation given in Corollary \ref{cor:spectrum}.
Choose a loop $\ga$ of $\C \setminus \Si$ which encircles
$I$. When $k$ is sufficiently large, by Corollary \ref{cor:spectrum}, $\ga$
does not meet the spectrum of $k^{-1} \Delta_k$. So by Riesz projection
formula and \eqref{eq:resol_est},
\begin{gather} 
\Pi_k^I   = \frac{1}{2i\pi} \int_{\ga} ( \la - k^{-1} \Delta_k )^{-1} d\lambda =
\frac{1}{2i\pi} \int_{\ga}
R^{r_k}_k ( \lambda) \, d \lambda + \bigo ( k^{-\frac{1}{4}})
\end{gather}
with $r_k = k^{-\frac{1}{4}}$. Since $R^r_k (\la) := \sum_{y \in I(r)} \chi_{y, r } \wt{R} _{y,k} ( \la)
\psi_{r,y}$, it comes that 
\begin{gather} \label{eq:step_1_bilan} 
\Pi_k^I =  \sum_{y \in I(r_k)} \chi_{y, r_k} \wt{P}^I_{y,k}  \psi_{r_k,y} + \bigo ( k^{-1/4}) 
\end{gather}
where for any $y$
$$\widetilde{P}^I_{y,k} =  \frac{1}{2i\pi} \int_{\ga} \wt{R}_{y,k} ( \la)  d
\lambda.$$
Recall that $\wt{R}_{y,k} (\la) $ is the restriction of the resolvent $(\la -k^{-1}
\Delta_{y,k})^{-1}$ to $\Ci_0 ( B_y(r) , \C^r)$ identified with
$\Ci_0 ( \Psi_y( B_y(r)), L^k \otimes A)$. So by Riesz projection formula
again, $\wt{P}^I_{y,k}$ is the restriction of the spectral projection
$$P^I_{y,k} =  \frac{1}{2i\pi} \int_{\ga}
( \la - k^{-1} \Delta_{y,k})^{-1} \, d \lambda .$$

\noindent {\em Step 2}. Let $d : M^2 \rightarrow
\R_{\geqslant 0}$ be a distance locally equivalent to the Euclidean distance
in each chart and set $m_k(x',x) := k^{n} \exp (- k c d(x',x)^2 )  $ with $c>0$.
Then by Schur test, any operator family $(Q_k:\Ci ( M, L^k\otimes A) \rightarrow
\Ci(M,L^k \otimes A)$, $k \in \N)$ having a continuous Schwartz kernel satisfying $|Q_k(x',x)| =
\bigo ( m_k(x',x))$ uniformly with respect to $x,x'$ and $k$, has a bounded operator norm, cf. \cite[proof of Lemma
5.1]{oimpart1} for more details.
Given this and \eqref{eq:step_1_bilan}, it suffices now to prove that
\begin{gather} \label{eq:but2} P_k (x',x ) = \sum_{y \in I(r)} \chi_{y, r_k} (x')
  \wt{P}^I_{y,k}(x',x)  \psi_{r_k,y} (x) + ( m_k ( x',x)+ 1) \bigo ( k^{-\frac{1}{4}}).
\end{gather}
In the sequel, we will allow the constant $c$ entering in the definition of
$m_k$ to decrease from one line to another. With this convention, for any
$p>0$, we can
replace any $\bigo ( d ^p(x',x) m_k(x',x))$ by a $\bigo ( k^{\frac{p}{2}} m_k(x',x))$.

\noindent {\em Step 3.} 
 \eqref{eq:but2} follows from
\begin{gather} \label{eq:loc_est} P_k (x',x) = \wt{P}^I_{y, k}(x',x) +  ( m_k ( x',x)+ 1) \bigo (
k^{-\frac{1}{4}}) 
\end{gather}
for all $ (x',x) \in \Psi_y
(B_{y} (2r)) \times \Psi_y( B_y (2r))$ with a $\bigo$ uniform with
 respect to all the variables, $y$ included. Indeed,  since $\op{Supp}  \psi_{r,y} \subset \Psi_y(B_{y}(r)) \subset \{\chi_{y,r} =1 \}$,
we have
$$ \chi _{y,r} (x') \psi_{r,y}(x) = \psi_{r,y} ( x) + \bigo ( d(x',x) r^{-1}), \qquad
\forall x,x' \in\Psi_y( B_{y}(2r)).$$
Recall that by \cite[Lemma 5.1]{oimpart1}, $P_k (x',x)= \bigo ( m_k(x',x) ) + \bigo (
k^{-N})$ for any $N$. Applying this to $N =1/4$ and using that  $m_k d = \bigo (
k^{-\frac{1}{2}} m_k)$ as explained above, we obtain
$$  \chi _{y,r} (x') P_k (x',x) \psi_{r,y}(x) = P_k (x',x) \psi_{r,y} ( x) + \bigo
(k^{-\frac{1}{2}} m_k (x',x) r^{-1})+ \bigo ( k^{-\frac{1}{4}}) .$$
Assume now that  \eqref{eq:loc_est} holds. Multiplying  \eqref{eq:loc_est} by
$\chi _{y,r} (x') \psi_{r,y}(x)$ and using the last equality, we obtain
$$  P_k (x',x) \psi_{r,y} ( x) = \chi _{y,r} (x') \wt{P}_{y,k}^I (x',x) \psi_{r,y}(x)  + \bigo
(k^{-\frac{1}{2}} m_k (x',x) r^{-1}) + \bigo ( k^{-\frac{1}{4}}) $$
which holds for all $x',x \in M$. Recall that the covering $\bigcup
\Psi_y(B_{y} (r))$, $y \in I(r)$ has a multiplicity bounded independently on $r$. So we can sum
these estimates without multiplying the remainder by the number of summands and we obtain
$$  P_k (x',x)  = \sum_{y \in I(r)} \chi _{y,r} (x') \wt{P}^I_{y,k} (x',x) \psi_{r,y}(x)  + \bigo
(k^{-\frac{1}{2}} m_k (x',x) r^{-1})   + \bigo ( k^{-\frac{1}{4}}).$$
This proves \eqref{eq:but2} because $r_k = k^{-\frac{1}{4}}$.

\noindent {\em Step 4.} We give a formula for the Schwartz kernel of the
spectral projector $P^I_{y,k}$. First, by the
rescaling \eqref{eq:rescaling_map}, \eqref{eq:rescale}, we have
\begin{gather} \label{eq:rescaling_proj}
P^I_{y,k}
( \xi,  \eta ) = k^n P^I_{y} (k^\frac{1}{2} \xi, k^\frac{1}{2}
\eta)
\end{gather}
with $P^I_y := P^I_{y,1}$. Second, the Schwartz kernel of $P^I_{y}$ is given by 
\begin{gather} \label{eq:proj_spec_Landau}
P^I _{y} ( \eta + \xi , \eta  ) = (2 \pi)^{-n} e^{ \frac{i}{2}
    \om_y ( \eta, \xi ) - \frac{1}{4} |\xi|_y^2 } \pi^I (y, \xi) .
\end{gather}
Indeed, by \eqref{eq:laplacien_creat_annihil}, $P^I_{y} = e^{-\frac{1}{4} |\xi|_y}\op{Op} (
  \si^I(y))  e^{\frac{1}{4} |\xi|_y}$ and it follows from \eqref{eq:def_op}
  that
\begin{xalignat}{2} \notag
  P^I_{y}  ( \xi, \eta) & =  (2 \pi)^{-n}  e^{ - \frac{1}{2} |u|^2 + u \cdot \con v -
    \frac{1}{2} |v|^2} \si^I (y, u - v)  \\  \label{eq:proj_spec_coor_comp} & = (2 \pi)^{-n} e^{ \frac{1}{2} ( u \cdot
    \con v - \con u \cdot v ) -\frac{1}{2} |u - v|^2 } \si^I (y, u - v) 
\end{xalignat}
with $(u_i)$, $(v_i)$ the complex coordinates of $\xi$ and $\eta$ defined as in
section \ref{sec:laplacian-linear}, in particular $|\xi|_y^2 =
\frac{1}{2}|u|^2$ and $|\eta|_y^2 = \frac{1}{2} |v|^2$. Since $\om_y =  i
\sum_i du_i \wedge d \con{u}_i$, \eqref{eq:proj_spec_Landau} follows from \eqref{eq:proj_spec_coor_comp}.  
Inserting \eqref{eq:proj_spec_Landau} into \eqref{eq:rescaling_proj}, it comes
that
\begin{gather} \label{eq:pi_y_k_la_formule}
P^I_{y,k}( \eta + \xi , \eta)  = \Bigl( \frac{k}{2\pi} \Bigr)^n F_y^k
(\eta + \xi , \eta) e^{ - \frac{k}{4} | \xi|_y^2} \si^I(y, k^{\frac{1}{2}} \xi )
\end{gather}
with $F_y ( \eta + \xi, \eta) = e^{  \frac{i}{2}
    \om_y ( \eta, \xi )}$. $F_y$ has the same characterization
  as the section $F$ entering in the expansion \eqref{eq:expansion}, that is  $F_y ( \eta , \eta)
  =1$ and $\R \ni t \rightarrow F_y ( \eta + t \xi , \eta )$ is flat for any
  $\xi$, $\eta$. 

\noindent {\em Step 5.} The Schwartz kernel of $P_{k}$ has the local
expansion \eqref{eq:expansion}. By \cite[Lemma 5.1]{oimpart1}, the remainder $r_{N,k}$ is in $\bigo (
k^{-\frac{N }{2}} m_k) + \bigo ( k^{-N'} ) $ for any $N'$. So in particular, 
\begin{gather}  \label{eq:projkernel}
P_k ( x+ \xi , x) = \Bigl( \frac{k}{2\pi} \Bigr)^n F^{k} (x+ \xi , x)
e^{-\frac{k}{4} |\xi|_x^2} \sigma^I (x, k^{\frac{1}{2}} \xi ) + ( m_k +1) \bigo (
k^{-\frac{1}{4}}) . 
\end{gather}

\noindent {\em Step 6.}
We now prove \eqref{eq:loc_est} by comparing \eqref{eq:pi_y_k_la_formule} and
\eqref{eq:projkernel}. So let $x, x' \in \Psi_y(B_y(2r) )$ and $\xi = x'-x$.
We will use several times that
$$ d(x,y) \leqslant C r, \qquad C^{-1} d  \leqslant  |
 \xi | \leqslant C d \quad  \text{ where } d:= d (x',x).$$ 
Let $\Phi_y : \Psi_y(B_y(r_0) ) \rightarrow T_yM$ be the inverse of
$\Psi_y$.
We have to compare $P_k ( x+ \xi, x)$ with
$\wt{P}^I_{y,k} ( x+ \xi, x) = P_{y,k}^I ( \eta + \tilde \xi, \eta)$, where
$$ \eta = \Phi_y ( x), \qquad  \eta +
\tilde \xi = \Phi_y (x + \xi)  $$
We claim that 
\begin{gather} \label{eq:xi_tilde_xi} 
\tilde \xi = \xi + \bigo ( rd + d^2).
\end{gather}
To see this, write $\tilde{\xi} = \Phi_y (x+ \xi) - \Phi_y ( x)=  L_y (x,\xi) \xi$ where
$L_y (x,0) = T_x \Phi_y$. Since $L_y(y,0) = \op{id}_{T_yM}$, we have
$$ L_y(x, \xi ) = L_y (x,0) + \bigo ( |\xi|) = \op{id}_{T_y M} + \bigo (d(x,y)+
|\xi|) .$$
So $\tilde{ \xi} = \xi + \bigo (|\xi|(d(x,y) +
|\xi|) )  = \xi + \bigo ( d(r + d))$.

Consider now a smooth function $(x, \xi) \rightarrow q (x, \xi)$ which is
polynomial homogeneous in $\xi$ with degree $\ell$. Then
$$ q(x, \xi ) = q(y, \xi ) + \bigo ( d(x,y) \; | \xi |^\ell) =  q ( y, \xi ) + \bigo
( r d^\ell)$$
and by \eqref{eq:xi_tilde_xi}, $q ( y, \xi) = q ( y, \tilde{\xi}) + \bigo (
d^\ell ( r + d))$. So
\begin{gather} \label{eq:qx-kfrac12xi-} 
q(x,  k^{\frac{1}{2}}\xi ) = q ( y,  k^{\frac{1}{2}} \tilde \xi ) + \bigo ( ( k^{\frac{1}{2}} d )^\ell
(r+ d)).
\end{gather}
Consequently
\begin{gather} \label{eq:sii--x}
\si^I ( x,  k^{\frac{1}{2}} \xi ) = \si^I ( y , k^{\frac{1}{2}}
  \tilde{\xi} ) + \bigo ( r+ d) \textstyle{\sum}   (k^{\frac{1}{2}} d) ^\ell 
\end{gather}
the sum on the right being over $\ell$ and finite.    

By \cite[Section 2.6]{oim}, the section $E (x + \xi, x) := F(x+ \xi,x) e^{- \frac{1}{4}
  |\xi|_x^2}$ depends on the coordinate choice up to a section vanishing to
third order along the diagonal. So
$$ E(x+ \xi , x) = F_y (  \eta + \tilde{\xi} , \eta ) e^{-
  \frac{1}{4}|\tilde{\xi}|_x} e^{\bigo ( d^3)} = E_y ( \eta + \tilde{\xi} ,
\eta) e^{ \bigo ( d^3 + d^2 r)}$$  
with $E_y ( \eta + \tilde{\xi} ,
\eta) := F_y ( \eta + \tilde \xi , \eta) e^{- \frac{1}{4} | \tilde \xi | _ y}$
because $ |\tilde \xi |_y^2 = | \xi|_x^2 + \bigo ( d^2 ( r+ d))$ by
\eqref{eq:qx-kfrac12xi-}. So using that $|e^z - 1 | \leqslant |z| e^{
  |\op{Re} z|}$ and that $k^n E^k ( x+ \xi , x) = \bigo (m_k)$, it comes that 
\begin{xalignat}{2} \label{eq:beginsplit--kn}
\begin{split} 
  & k^n ( E^k ( x+ \xi, x) - E_y^k( \eta + \tilde{\xi}, \eta)) = \bigo ( d^2
(d+r) m_k ) e^{ k C d^2 (d+r)} 
\\ & =  \bigo ( d^2
(d+r) m_k ) e^{ k C d^2 (d+r)}= \bigo ( k^{-\frac{5}{4}} m_k )  e^{ k C d^2
  (d+r)}  = \bigo ( k^{-\frac{5}{4}} m_k )
\end{split}
\end{xalignat}
where we have used that $d$ and $r$ are both in  $\bigo ( k^{-\frac{1}{4}})$,
and always the same convention that the constant $c$ in $m_k$ can change from
one line to another so that  $ d^p m_k = \bigo ( k^{-\frac{p}{2}} m_k)$.  Using again that $k^n E^k ( x+ \xi , x) = \bigo
(m_k)$, it follows from \eqref{eq:sii--x}, 
\begin{xalignat*}{2}  k^n E^k ( x + \xi , x) \si^I(x,  k^{\frac{1}{2}} \xi) & = k^n E^k ( x + \xi ,
x) \si^I(y,  k^{\frac{1}{2}} \tilde \xi) + \bigo ( k^{-\frac{1}{4}} m_k) \\ &
= k^n E_y^k ( \eta + \tilde \xi ,
\eta) \si^I(y,  k^{\frac{1}{2}} \tilde \xi) + \bigo ( k^{-\frac{1}{4}} m_k)  
\end{xalignat*}
by \eqref{eq:beginsplit--kn}, which ends the proof of  \eqref{eq:loc_est}
\end{proof}

\subsection{A formal projector}  \label{sec:formal-projector}

This section is devoted to the proof of the following Lemma.
\begin{lemme} \label{lem:formal-projector}
Under the same assumptions as in Theorem \ref{theo:main_result_first},  there
exists $(P_k) \in \lag^{\infty}(A) \cap \lag^{+}(A)$ unique modulo
$\lag^{\infty}_{\infty}(A) $ such that $\si_0( P_k) = \pi^I$, $P_k = P_k ^*$ for any $k$,  $P_k \equiv
  P_k^2$ modulo $\lag^{\infty}_{\infty} (A)$ and $[\Delta_k , P_k ]\equiv 0$ modulo $\lag
  ^{\infty}_{\infty}(A)$. 
\end{lemme}

To show this, we will construct $(P_k)$  by successive approximations.
Introduce the filtration $\lag^{\infty}_p(A) := \lag^{\infty} (A)\cap \biginf (k^{-\frac{p}{2}})$, $p
\in \N$. For any $p \in \N$, we have a symbol map
$$ \si_p: \lag_p^{\infty}(A) \rightarrow
\Ci (M, \symb (TM) \otimes \op{End} A)$$
such that $\si_p ( P ) = \si_0 (k^{\frac{p}{2}}P)$ where $\si_0$ was defined in
\eqref{eq:sigma_0}. By \cite[Proposition 2.1 and Theorem 2.2]{oimpart1}, $\si_p$ is onto, $\op{ker} \si_p =
\lag_{p+1}^{\infty} (A)$ and for any sequence $( Q_p)$ of $\lag^{\infty}(A)$ such that $Q_p \in
\lag^{\infty}_p(A)$ for any $p$, there exists $Q \in \lag^{\infty}(A)$ such
that $Q = Q_0 + \ldots + Q_p$ modulo $\lag_{p+1}^{\infty}(A)$ for any $p$.  
Moreover,
\begin{enumerate}
\item if $Q$ and $Q'$
belong to $\lag^{\infty}_p(A)$ and $\lag^{\infty}_{p'}(A)$ respectively, then their
product  belongs to $\lag^{\infty}_{p+p'} (A)$. Furthermore, at any $x
\in M$, $\si_{p+p'}(QQ')(x)$ is the product of $\si_p(Q)(x) $ and
$\si_{p'}(Q')(x)$.
\item if $Q$ belongs to $\lag^{\infty}_p(A)$, then its adjoint $Q^*$ belongs
  to $\lag^{\infty}_p(A)$ with symbol $\si_p ( Q^*) (x) = \si_p (Q)(x)^*$. 
\end{enumerate}
By \cite[Theorem 2.5]{oimpart1},  $\lag^{+} (A)$ is a subalgebra of $\lag (A)$. 

\begin{lemme} \label{lem:comp_Delta}
For any $Q $ in $\lag^{\infty}_p(A)$, $(k^{-1} \Delta_k Q_k)$ and
  $(k^{-1} Q_k \Delta_k)$ belong both to $\lag^{\infty}_p(A)$ and their symbols at $x$ are
  $ \PP_x \circ \si_p(Q)(x)$ and $ \si_p(Q)(x) \circ \PP_x$. If
  $Q \in \lag^{+} (A)$ then the same holds for $(k^{-1} \Delta_k Q_k)$ and $(
  k^{-1} Q_k \Delta_k)$. 
\end{lemme}

\begin{proof} By \cite[Proposition 6.3, Assertion 3c and 3d]{oimpart1},  $(k^{-1} \Delta_k Q_k)$ and
  $(k^{-1} Q_k \Delta_k)$ belong both to $\lag^{\infty}_p(A)$. To compute the
  symbol, we can use the peaked sections of Section \ref{sec:peaked-sections}. Indeed, if $\Phi_k(f)$ is defined by
  \eqref{eq:peaked_Section} with $ f \in \D(T_xM ) \otimes A_x$ and $(P_k) \in
  \lag_0 (A)$ then by \cite[Proposition 2.4]{oimpart1}, $P_k
  \Phi_k (f) = \Phi_k (g) + \bigo ( k^{-\frac{1}{2}})$ with $g = \si_0 (P_k)
  (x) f$. So the symbol of any operator of $\lag_p (A)$ is characterized by
  its action on the peaked sections. Proposition \ref{prop:peaked-sections}
  tells us how $k^{-1} \Delta_k$ acts on the peaked section and the first part
  of the result
  follows. To show that the composition with $k^{-1} \Delta_k$ preserves the
  subspace $\mathcal{L}^+ (A)$ of even operators, one uses instead of the asymptotic expansion
  \eqref{eq:expansion} the alternative expansion
  $$ P_k (x,y) = \Bigl( \frac{k}{2\pi} \Bigr)^n E^k(x,y) \sum
  k^{-\frac{\ell}{2}} b_{\ell}(x,y) + \bigo ( k^{-\infty}) ,$$
  cf. \cite[Equation (45) and Proposition 5.6]{oimpart1}.  The fact that
  $(P_k)$ is even means that $b_\ell =0 $ when $\ell$ is odd. When $(P_k) \in
  \lag^{\infty}(A)$, this expansion holds for the $\Ci$ topology, so we can
  compute the Schwartz kernel of $k^{-1} \Delta_k P_k$ by letting $k^{-1}
  \Delta_k$ act on each term of the
  expansion. Doing this with the expression \eqref{eq:ass_op}, no half power
  of $k$ appears so $k^{-1} \Delta_k P_k$ is even. The same argument works for
  $k^{-1} P_k \Delta_k$.  
\end{proof}

In the sequel, to lighten the notations, we write $\pi$ instead of $\pi^I$. 
Let $L_1$ and $L_2$ be the endomorphisms of $\Ci ( M , \symb
(TM) \otimes \op{End} A)$ defined by 
\begin{xalignat*}{2}
  L_1 (f)(x) & = \pi(x) \circ f(x) + f(x) \circ \pi(x) - f(x)\\
  L_2(f) (x) & = [ \PP_x, f (x) ] 
\end{xalignat*}
Assuming that $I$ satisfies \eqref{eq:ass_I}, $\pi \in
\Ci ( M ,\op{End} ( \D_{\leqslant p_0 } (TM) \otimes \op{End} A) $ for some $p_0$,
so that $L_1$ is well-defined meaning that $L_1( f)$ is a smooth section of
$\symb (TM) \otimes \op{End} A$ when $f$ is. 

\begin{lemme} \label{lem:exact_Seq}
  The
following  sequence is exact
\begin{gather} \label{eq:ex_seq}
0 \rightarrow \op{Symb} \xrightarrow{L} \op{Symb}
  \oplus \op{Symb}  \xrightarrow{L'} \op{Symb}
  \rightarrow 0 
\end{gather}
where  $\op{Symb} = \Ci ( M , \symb (TM) \otimes \op{End} A)$, $L(f) =  (L_1(f), L_2(f))$ and $L'(f_1, f_2) = L_2 (f_1) -
  L_1 (f_2)$. 
\end{lemme}
\begin{proof}
$L ' \circ L =0$ is equivalent to $L_1 \circ L_2 = L_2 \circ L_1$, which
follows from $[\PP, \pi ]=0$. Indeed $[\PP , \pi ] =0$ implies that $
[\PP , f \pi ] = [\PP , f ] \pi$ and $[\PP , \pi f ] = \pi [\PP ,
f]$ so that
$$ L_2 ( L_1 ( f)) = [\PP, f \pi + \pi f - f ] = [\PP , f ] \pi + \pi [
\PP , f ] - [\PP , f] = L_1 ( L_2 (f))$$
Recall that
$ \op{Symb }  = \bigcup_{p \in \N}  \op{Symb}_p $ with  $\op{Symb}_p = \Ci ( M ,
 \op{End} ( \D_{\leqslant p } (TM) \otimes A)  ).$
 $L_2$ preserves each $\op{Symb}_p$ and the same holds for $L_1$ when $p$ is
  larger than $p_0$.  So we have to prove that for any $p \geqslant p_0$, the sequence \eqref{eq:ex_seq}
  with $\op{Symb}$ replaced by $\op{Symb}_p$ is exact.

  By Lemma \ref{lem:smooth}, the image of $\pi$ is a subbundle $F$ of
  $\D_{\leqslant p} (TM) \otimes A$. Let $F^{\perp}$ be the orthogonal subbundle, so
  that $\D_{\leqslant p} (TM)\otimes A = F \oplus F^{\perp}$. Write the elements of
  $\op{Symb}_p$ as block matrices according to this decomposition. The
  restrictions of $\pi$ and $\PP$ to $\op{Symb}_p$ have the particular forms
  $$ \pi = \begin{pmatrix} 1 & 0 \\ 0 & 0 \end{pmatrix} , \quad \PP
  = \begin{pmatrix} \PP_{\op{in}} & 0 \\ 0 & \PP_{\op{out}} 
  \end{pmatrix}
   $$
  Writing $ f = \begin{pmatrix} a & b \\ c & d 
  \end{pmatrix}$, we have   $$ L_1 (f) = \begin{pmatrix} a & 0 \\ 0 & -d 
  \end{pmatrix} ,  \qquad L_2 (f) = \begin{pmatrix} [\PP_{\op{in}} , a ] &
    E_1 ( b) \\ E_2( c) & [\PP_{\op{out}}, d ] 
  \end{pmatrix}$$
  with
  $$E_1 ( b) = \PP_{\op{in}} b - b\PP_{\op{out}}, \qquad E_2(c) =
  \PP_{\op{out}} c - c \PP_{\op{in}} = - E_1  ( c^{*})^{*} .$$

  Let us prove that $E_1 $ and $E_2$ are
  invertible endomorphisms of the spaces $\Ci ( M , \op{Hom} (F^{\perp}, F))$ and $\Ci (
  M , \op{Hom} ( F, F^{\perp}))$ respectively. For any $y \in M$, introduce
  an orthonormal eigenbasis $(e_i)$  of the restriction of $\PP_y$
  to $\D_{\leqslant p }
  (T_yM) \otimes A_y$.  So $\PP_y e_i = \la_i e_i$ and $F_y$ (resp. $F_y^{\perp}$) is
  spanned by the $e_{i} $ such that $\la_i \in I$ (resp. $\la_i \notin I$).
  Now the endomorphism
\begin{gather}  \label{eq:ophom--f_}
  \op{Hom} ( F^{\perp}_y, F_y) \rightarrow \op{Hom} (F^{\perp}_y, F_y),
  \qquad   b(y) \rightarrow    \PP_{\op{in}} (y) b(y) -
  b(y)\PP_{\op{out}}(y)
\end{gather}
is diagonalisable with eigenvectors $| e_i
  \rangle \langle e_j  |$ and eigenvalues $\la_i - \la_j$, where $\la_i \in I$
  and $\la_j \notin I$. Since $\la_i - \la_j \neq 0$, \eqref{eq:ophom--f_} is
  invertible for any $y$, so the same holds for $E_1$. The proof for $E_2$ is similar.

  From this, we deduce easily that the sequence is exact. In particular if
  $L'(f_1, f_2) = 0 $ with $f_i = \begin{pmatrix} a_i & b_i \\ c_i & d_i 
  \end{pmatrix}$ for $i=1$ or $2$, then $(f_1, f_2) = L(f)$ with

  $$ f = \begin{pmatrix} a_1 & E_1^{-1} (b_2)  \\ E_2^{-1} (c_2) & -d_1
  \end{pmatrix}.$$
  Observe as well that $f_1 = f_1^*$ and $f_2 = -
  f_2^*$ imply that $f = f^*$. 
\end{proof}

\begin{proof}[Proof of Lemma \ref{lem:formal-projector}]
Let $P \in \lag^{\infty} (A) $ be self-adjoint with symbol $\sigma_0 ( P ) = \pi$.
Then $R_{1} := P^2 - P$ and $R_{2} := k^{-1}
[\Delta_k, P]$ both belong to $\lag_1^{\infty} (A) $. Indeed their $\si_0$-symbols are
respectively $\pi^2 - \pi$ and $[\PP , \pi]$, and both of them  vanish.

Let us prove by induction on $m \geqslant 1$ that there exists $P$ as above such that $R_1$ and
$R_2$ are in $\lag_{m}^{\infty} (A)$.
Define $P' = P+S$ with $S \in \lag_m^\infty(A)$. Assume that $R_1$ and $R_2$ are in
$\lag^{\infty}_m(A)$. Then
\begin{gather*}
  (P')^2 - P' = R_{1} + S P + P S - S \mod  \lag_{m+1}^\infty(A) \\
  [k^{-1}\Delta_{k} , P' ] = R_{2} +  [k^{-1}\Delta_{k} , S ] 
\end{gather*}
So  $(P')^2 - P'$ and $ [k^{-1}\Delta_{k} , P' ]$ belong to
$\lag_{m} (A)$ and their $\si_m$-symbols are respectively $f_1 + L_1 (f)$ and $f_2
+ L_2(f)$
with $f = \si_m (S) $, $ f_1 = \si_m (R_1)$ and $f_2 = \si_m (R_2)$. Let us
prove that we can choose $f$ so that $f_1 + L_1(f) = 0$ and $f_2 +L_2(f) = 0$. By
Lemma \ref{lem:exact_Seq}, it suffices to check that $L_1 (f_2) = L_2 ( f_1)$.
But $L_2 (f_1)$ is the $\si_m$-symbol of $[k^{-1} \Delta _k , R_{1}]$,
$L_1(f_2)$ is the $\si_m$-symbol of $ P R_2 + R_2 P - R_2$,  and these operators
are equal as shows a direct computation. So $f$ exists. Furthermore $f =
f^*$ by the remark in the end of the proof of Lemma
\ref{lem:exact_Seq}. So we can choose $S$ self-adjoint. 

We conclude the proof with the convergence property with respect to the
filtration $\lag_m (A)$ recalled above.  Observe also that if we start with $P
\in \lag^{+} (A)$, then we end with a formal projector in $\lag^{+}(A)$.
\end{proof}

\subsection{Operator norm and pointwise estimates} 
\label{sec:end-proof-theorem}

Let us choose an operator $(P_k)$ satisfying the conditions of Lemma
\ref{lem:formal-projector}. Recall that for any operator $Q \in \lag_m(A)$, $Q_k
= \bigo ( k^{-\frac{m}{2}})$ in the sense that the operator norm of $Q_k$ is in $\bigo (
k^{-\frac{m}{2}})$ . So $P_k$ is self-adjoint, it is an almost
  projector $P_k ^2 = P_k + \bigo ( k^{-\infty})$ and it almost commutes with
  $\Delta_k$ in the sense that  $[\Delta_k , P_k ]= \bigo ( k^{-\infty})$. Furthermore, by
  Lemma \ref{lem:first-approximation},
  $P_k =
  \Pi^I_k + \bigo ( k^{-\frac{1}{4}})$. 
  
\begin{lemme} \label{lem:normeL2}
    $P_k =  \Pi_k^I + \bigo ( k^{-\infty})$.
\end{lemme}

\begin{proof}
  We omit the index $k$ to simplify the notations. Let $\mathcal{H}_+ =
  \op{Ran} \Pi^I$ and $\mathcal{H}_-$ be its orthogonal in $L^2 ( M, L^k \otimes
  A)$. Introduce the corresponding block decomposition of $P$ 
$$P= \begin{pmatrix}
    P_{++} & P_{+-} \\ P_{-+}  & P_{--} 
  \end{pmatrix}.$$
  We first prove that $P_{-+}$ and $P_{+-}$ are in $\bigo (
  k^{-\infty})$. 

By Corollary \ref{cor:spectrum} and assumption \eqref{eq:ass_I}, there exists $\epsilon$ such
that when $k$ is sufficiently large $\op{dist} ( I , \op{sp} ( k^{-1} \Delta_k )
\setminus I ) \geqslant \ep$. Let $\xi_{\la}$ and $\xi_{\mu}$ be two
eigenfunctions of $k^{-1} \Delta_k$ with eigenvalues $\la$ and $\mu $
respectively.  Then 
\begin{xalignat}{2} \label{eq:beginsplit--la}
  \begin{split} 
( \la - \mu ) \langle P \xi_{\la} , \xi_{\mu} \rangle & = k^{-1} \bigl( \langle P  
\Delta_k  \xi_{\la}, \xi_{\mu} \rangle - \langle P \xi_{\la} , 
\Delta_k  \xi_{\mu} \rangle \bigl) \\  & = k^{-1} \langle [P,  \Delta_k ] \xi_\la, \xi_\mu
\rangle \\ & = \bigo ( k^{-\infty}) \| \xi_\la \| \, \| \xi_{\mu} \| 
\end{split} 
\end{xalignat}
because $[P, \Delta_k ] = \bigo ( k^{-\infty})$.
Now for any $\xi_+ \in \mathcal{H}_+$ and $\xi_- \in \mathcal{H}_-$, write
their decomposition into eigenvectors $\xi_+ = \sum \xi_\la$ and $\xi_- = \sum
\xi_\mu$. So $ \| \xi_+ \|^2 = \sum \| \xi_{\la}\|^2$, $\|\xi_-\|^2 = \sum \|
\xi_\mu\|^2$ and $ \langle P \xi_-, \xi_+ \rangle =   \sum \langle P \xi_{\la}
, \xi_{\mu} \rangle$. So by \eqref{eq:beginsplit--la},
$$ \bigl| \langle P \xi_-, \xi_+ \rangle  \bigr| \leqslant \ep^{-1} \bigo(
k^{-\infty}) \sum \| \xi_{\la} \| \, \| \xi_{\mu} \| \leqslant \ep^{-1} \bigo
( k^{-\infty}) \| \xi_- \| \, \| \xi_+ \| $$
by Cauchy-Schwarz inequality. This proves that $ P_{+-}  = \bigo (
k^{-\infty})$. The same holds for its adjoint $P_{-+}$. 

Now the fact that $P^2 = P + \bigo ( k^{-\infty})$ implies  that $P_{++} ^2 =
P_{++} + \bigo ( k^{-\infty})$ and the same for $P_{--}$. Indeed,
\begin{xalignat*}{2}  
(\Pi^I P \Pi^I)^2 & = \Pi^I P \Pi^I P \Pi^I\\
 &  = \Pi^I P ^2 \Pi^I + \bigo ( k^{-\infty}) \qquad \text{ because } P_{-+} =
 \bigo ( k^{-\infty}) \\
& =  \Pi^I P
\Pi^I + \bigo ( k^{-\infty}) \qquad \text{ because } P^2 = P +
 \bigo ( k^{-\infty}) .
\end{xalignat*}
By Lemma \ref{lem:first-approximation}, $P = \Pi^I + \bigo ( k^{-\frac{1}{4}})$, so
$P_{--} = \bigo ( k^{-\frac{1}{4}})$. Then $P_{--}^2 = P_{--} + \bigo ( k^{-\infty})$
implies that
$$P_{--} =  \bigo ( k^{-\infty}).$$
In the same way, $ (\op{id}_{\mathcal{H}_+} -
P_{++}) ^2 = \op{id}_{\mathcal{H}_+} -
P_{++} + \bigo ( k^{-\infty})$ and $\op{id}_{\mathcal{H}_+} -
P_{++}  = \bigo( k^{-\frac{1}{4}})$ implies that $\op{id}_{\mathcal{H}_+} -
P_{++}  =   \bigo ( k^{-\infty})$. So
$$ P_{++} = \op{id}_{\mathcal{H}_+} + \bigo ( k^{-\infty})$$
which concludes the proof.
\end{proof}

\begin{lemme} \label{lem:norme_L2_plus}
  For any $\ell, m \in \N$, $ \Delta_k ^\ell( P_k - \Pi^I_k ) \Delta_k^m
  = \bigo ( k^{-\infty})$.
\end{lemme}

\begin{proof}
  On one hand, we have
\begin{gather} \label{eq:majoration}
\Delta_k^{\ell} P_k =
  \bigo ( k^{\ell}), \qquad \Delta_k^{\ell} \Pi^I_k = \bigo ( k^{\ell})
\end{gather}
the first estimate being  a consequence of $((k^{-1} \Delta_k)^{\ell} P_k )\in
  \lag^{\infty} (A)$, the second one is merely that $\Pi^I_k$ is
  the spectral projector of $k^{-1} \Delta_k$ for the bounded interval $I$. 

  On the other hand, since for any $Q \in \lag^{\infty}_{\infty} (A)$,
  $\Delta_k^\ell Q \Delta_{k}^m $ belongs to $\lag^\infty_\infty(A)$ as well, we
  have 
    \begin{gather} \label{eq:impr}
    \begin{split} 
  \Delta_k^\ell P_k^2 \Delta_k^m = \Delta_k^\ell P_k \Delta_k^m + \bigo (
  k^{-\infty})\\
  \Delta_k ^\ell [k^{-1} \Delta_k, P_k]  \Delta_k^m = \bigo (
  k^{-\infty})
\end{split}
\end{gather}
By the first equality,  $\Delta_k^\ell P_k \Delta_k^m =
\Delta_k^{\ell+m} P_{k} + \bigo ( k^{-\infty})$. Since $[\Delta_k , \Pi^I_k ]=
0$, it suffices to prove the final
result for $m=0$, that is $\Delta^\ell_k ( P_k -\Pi_k^I)  =\bigo ( k^{-\infty})$. We have
\begin{xalignat*}{2} 
\Delta_k^{\ell} ( P_k - \Pi_k^I ) & = \Delta_k^\ell ( P^2_k - \Pi_k^I) + \bigo(
k^{-\infty}) \qquad \text{ by
} \eqref{eq:impr} \\ 
& = \Delta_k^\ell P_k(P_k -  \Pi_k^I) +  \Delta_k^\ell ( P_k \Pi_k^I) - \Delta_k^\ell \Pi_k^I + \bigo(
k^{-\infty}) \\
& = \Delta_k^\ell P_k(P_k -  \Pi_k^I) +  P_k \Delta_k^\ell  \Pi_k^I - \Delta_k^\ell \Pi_k^I + \bigo(
k^{-\infty})  \quad \text{by
} \eqref{eq:impr} \\
& = \Delta_k^\ell P_k(P_k -  \Pi_k^I) +  (P_k - \Pi_k^I) \Delta_k^\ell  \Pi_k^I + \bigo(
k^{-\infty}) \\
& = \bigo ( k^{\ell} ) \bigo ( k^{-\infty}) 
\end{xalignat*}
by 
\eqref{eq:majoration}  and Lemma \ref{lem:normeL2}. 
\end{proof}

We are now ready to conclude the proof of Theorems
\ref{theo:main_result_first} and \ref{theo:main_result_second}: we
will show that the Schwartz kernel of $P_k - \Pi_k^I$ is in $\biginf (
k^{-\infty})$, in the sense of Section \ref{sec:class-mathcallinfty}.

Choose two open sets $U$ and $U'$ of $M$ equipped both with a set of
coordinates and  unitary trivialisations of $L$ and $A$, so that we can identify
the sections of $L^k \otimes A$ on $U$ with functions. Let $\varphi \in \Ci_0
( U)$, $\varphi' \in \Ci_0(U')$. Then $\varphi ( P_k - \Pi_k^I) \varphi'$ can be
viewed as an operator of $\R^{2n}$. Introduce the differential operator 
$$ \La_k = 1  -
k^{-2} \sum_{i=1}^{2n} \partial_{x_i}^2 $$
acting on $\Ci ( \R^{2n})$. 

\begin{lemme}   For any $\ell\in \N$,
\begin{gather} \label{eq:to_prove}
  \La_k^\ell \, \varphi ( P_k - \Pi_k^I)
  \varphi' \, \Lambda_k^\ell = \bigo ( k^{-\infty}).
\end{gather}
  Consequently, the Schwartz kernel of $\varphi ( P_k - \Pi_k ^I) \varphi'$ is
  in 
  $\biginf ( k^{-\infty})$. 
\end{lemme}
\begin{proof} We will use basic results on semiclassical pseudodifferential
  operators of $\R^{2n}$, with the semiclassical parameter usually denoted by
  $h$ equal here to $k^{-1}$. Choose $\psi_1, \psi_2 \in \Ci_0(U)$ such
  that $\op{supp} \varphi \subset \{ \psi_1 =1 \}$ and $\op{supp} \psi_1 \subset
  \{ \psi_2 = 1 \}$.  The operator $\psi_1 ( 1+ (k^{-2} \Delta_k)^{\ell})$, viewed as an operator
  of $\R^{2n}$, is a semiclassical differential operator with principal
  symbol $\psi_1 (H^\ell +1 ) $ where $H$ is the symbol of $\Delta_k$, so
 $$ H(x,\xi) = \sum  g^{ij}(x) ( \xi_i + \al_i(x) )(\xi_j + \al_j(x)) ,$$ 
with $-i \sum \al_i dx_i$ the
  connection one-form of $L$ in the trivialisation used to identify
  sections with functions. $\varphi \La_k^\ell$ is also a semiclassical differential
  operator with symbol $\varphi (x) \langle \xi \rangle^{2 \ell}$. The symbol $\psi_1 ( H^\ell +1)$ being elliptic on $\op{supp}
  \varphi \times \con{\R^n}$, we can factorise
  $$ \Lambda_{k}^{\ell} \varphi = Q_k \psi_1  ( 1+ (k^{-2} \Delta_k)^{\ell}) + S_k $$
with $Q_k$ a zero order semiclassical pseudodifferential operator and $S_k$ in
the residual class. To do this, we only need the pseudodifferential calculus
in the usual class $S^k_{1,0}(T^* \R^{2n}) $ of
symbols, cf. as instance \cite[Section E.1.5]{DZ}. Composing with $\psi_2$ 
\begin{gather} \label{eq:11}
 \La_{k}^{\ell} \varphi = Q_k \psi_1  ( 1+ (k^{-2} \Delta_k)^{\ell}) + S_k
\psi_2 
\end{gather}
Similarly, we have
\begin{gather} \label{eq:22}
 \varphi' \La_k^{\ell } = \psi_1' ( 1+ (k^{-2} \Delta_k)^{\ell}) Q'_k +  \psi_2'
S'_k 
\end{gather}
Now by Lemma \ref{lem:norme_L2_plus},
\begin{gather} \label{eq:33}
( 1+ (k^{-2} \Delta_k)^{\ell}) (P_k - \Pi^I_k )  ( 1+ (k^{-2} \Delta_k)^{m})
= \bigo ( k^{-\infty}) 
\end{gather}
and by the usual result on boundedness of pseudodifferential operators, cf.
\cite[Proposition E.19]{DZ}, $Q_k , Q'_k = \bigo ( 1)$ and $S_k ,
S_k' = \bigo (k^{-\infty})$. We deduce \eqref{eq:to_prove} easily with \eqref{eq:11},
\eqref{eq:22} and \eqref{eq:33}.

Now let $H^m_k $ be the Sobolev space $H^m ( \R^{2n})$ with the $k$-dependent norm $\| u
\|_{H^m_k} = \| \langle k^{-1} \xi \rangle \hat{u}(\xi) \|_{L^2(\R^{2n})}$.
Then $\La_k^\ell$ is an isometry $H^m_k \rightarrow H^{m-2\ell}_k$. So
\eqref{eq:to_prove} tells us that the operator norm $H^{-2 \ell}_k \rightarrow H^{2
  \ell}_k $ of $R_k = \varphi ( P_k - \Pi_k^I)
  \varphi'$ is in $\bigo ( k^{-\infty})$. The Schwartz kernel of $R_k$ at
  $(x,y)$ being equal to $\delta_x (  R_k \delta_y)$ and the Dirac $\delta_x$
  belonging to $H^{-m}_k$ with a norm in $\bigo ( k^{2n})$ for any $m >n$, it
  comes that $ R_k ( x,y) = \bigo ( k^{-\infty})$. Similarly, $\partial_x^\al
  \partial_y^\be R_k (x,y) = \bigo ( k^{-\infty})$ for any $\al, \be \in
  \N^{2n}$ because the $H_k^{-m}$-norm of $ \partial ^{\al} \delta_x$ is a $\bigo
  (k^{2n})$ as soon as $m \geqslant n + |\al|$.     
\end{proof}

\section{Toeplitz operators} \label{sec:glob-spectr-estim}

Let $F$ be  a vector subbundle of $\D
_{\leqslant p} ( T M) \otimes A$ for some $p$. 
Let $(\Pi_k) \in \lag (A)$ such that for each $k$, $\Pi_k$ is a self-adjoint
projector of $\Ci ( M , L^k \otimes A)$ and for any $x \in M$, the symbol
$\pi(x) = \si_0 ( \Pi_k) (x) $ is the
orthogonal projector onto $F_x$. Let $\Hilb_k$ be the image of $\Pi_k$. 

The corresponding Toeplitz operators are the  $(P_k) \in \lag
(A)$ such that $\Pi_k P_k \Pi_k = P_k$. The symbol $\si_0 (P)
(x)$ of such an operator satisfies
$$\pi (x) \si_0 (
P ) (x) \pi (x)= \si_0 ( P ) (x).$$
So $ \si_0 ( P) (x) =  f(x) \pi (x)$ with $f(x) \in \op{End}
F_x$. This section $f$ of $\op{End} F$ can be considered as the Toeplitz
symbol of $(P_k)$.

We will establish several spectral results for these Toeplitz operators. Applied to the
spectral projector $\Pi_k = 1_{[a,b]} ( k^{-1} \Delta_k)$ and $P_k = k^{-1}\Delta_k
 \Pi_k$, this will complete the proofs of Theorems \ref{theo:intro_weyl_global},
 \ref{theo:intro_spec}, \ref{theo:intro_weyl_local} and \ref{theo:intro_proj} stated in the introduction. 

\subsection{Global spectral estimates} 

\begin{theo} \label{theo:glob-spectr-estim}
  \begin{enumerate} 
\item   
  When $k$ is sufficiently large, $\op{dim} \Hilb_k = \op{RR} ( L^k \otimes F)$.
\item
  For any $(P_k) \in \lag (A)$ such that $P_k ^* = P_k$ and $\Pi_k P_k \Pi_k =
  P_k$ for any $k$, we have for any $\Psi \in \Hilb_k$ with $\| \Psi \| = 1
  $ that
\begin{gather} \label{eq:Garding}
  \inf_M  f_- + \bigo ( k^{-\frac{1}{2}}) \leqslant \langle P_k \Psi ,
  \Psi \rangle  \leqslant \sup_M f_+ + \bigo (k^{-\frac{1}{2}})
\end{gather} 
where the $\bigo$'s are uniform with respect to $\Psi$ and for any $x \in M$,
 $f_{-} (x)$ and $f_{+} (x)$ are the smallest and largest eigenvalues of the
 restriction of $\si_0 (P) (x)$ to $F_x$. 
\end{enumerate}
\end{theo}

The proof is based on the generalised ladder operators introduced in
\cite{oimpart1}: if $(A',F',\Pi_k', \Hilb_k')$ is a second set of data
satisfying the same assumption as $(A,F, \Pi_k, \Hilb_k)$ and $F$,
$F'$ are isomorphic vector bundles, then there exists isomorphisms $U_k : \Hilb_k \rightarrow \Hilb'_k$ when
$k$ is sufficiently large. Then defining  $\Hilb_k'$ as the kernel of a
well-chosen spin-c Dirac operator, $\dim \Hilb_k '$ is given by the
Atiyah-Singer Theorem, which will prove the first statement. For the second one, choosing $\Hilb'_k$ so that $F'=A'$,
$U_kP_kU_k^*$ is equal to a Toeplitz operator $\Pi_k ' f \Pi_k'$ up to a
$\bigo ( k^{-\frac{1}{2}})$. The inspiration here comes from the proof of the
sharp G{\aa}rding inequality for semiclassical pseudodifferential operator. 

\begin{proof}
Consider a second self-adjoint projector $\Pi' \in \lag ( A')$
with $\si_0 (\Pi')$ the orthogonal projector onto a vector bundle $F'$ of
$\mathcal{D}_{\leqslant p } (TM) \otimes A'$. Assume that $F$ and $F'$ are
isomorphic vector bundles. Then there exists $u \in \Ci ( M, \op{Hom} (F, F'))$
such that for any $x \in M$, $u(x)$ is a unitary isomorphim from $F_x$ to
$F'_x$. Extending $u (x)$ to a map $\D(T_xM ) \otimes A_x \rightarrow \D (T_xM)
\otimes A_x'$ which is zero on the orthogonal of $F_x$, we have $u^* (x) u(x)
= \si_0 ( \Pi ) (x)$ and $u(x) u^*(x) = \si_0 ( \Pi')(x)$. So if $(U_k) \in \lag (
A, A')$ has symbol $u$, then
\begin{gather} \label{eq:almost_unitary}
U^*_k U_k = \Pi_k + \bigo ( k^{-\frac{1}{2}}), \qquad U_k U_k^* = \Pi'_k +
\bigo ( k^{-\frac{1}{2}}). 
\end{gather}
Furthermore replacing
$U_k$ by $\Pi_k' U_k \Pi_k$ does not modify the symbol of $U_k$ so the same
property holds and moreover $\Pi_k' U_k \Pi_k
= U_k$. Consequently $U_k$ restricts to an isomorphism from $\Hilb_k$ to the
image $\Hilb'_k$ of $\Pi'_k$, when $k$ is sufficiently large.

Hence for large $k$, the dimension of $\Hilb_k$ only depends
on the isomorphism class of $F$. To compute it, we introduce a
spin-c Dirac operators $D_k$ acting on $L^k \otimes A'$ with $A' = F \otimes
\bigwedge^{0,\bullet} T^*M$ and define $\Hilb'_k$ as the kernel of $D_k$. Then
by a vanishing theorem \cite{BoUr96, MaMa02}, $\dim \Hilb_k ' $ is equal to the index of $D_k^+$ when
$k$ is sufficiently large. By Atiyah-Singer index theorem, $\dim \Hilb_k' =
\op{RR} ( L^k \otimes F)$. Furthermore, it follows from \cite{MaMa} that the
projector $(\Pi'_k)$ belongs to $\lag ( A')$ and $\si_0 (\Pi_k')$ is the
projector onto $\C \otimes F \otimes \C$. Alternatively the vanishing theorem
and the fact that $(\Pi_k') \in \lag (A')$ follows also from Corollary
\ref{cor:spectrum} and Theorem \ref{theo:main_result_first} applied to $D^-_k
D^+_k$ as in the proof of Theorem \ref{theo:semicl-dirac-oper}.

To prove the second part, we choose $A' = F'=F$, that is $(\Pi'_k) $ belongs to $\lag (F)$ and
its symbol is the projection onto $\D_0 (TM) \otimes F$. For
instance, we can choose $\Pi'_k = 1_{I}(k^{-1} \Delta_k)$ with $I =
\frac{n}{2} + [-\frac{1}{2} , \frac{1}{2} ]$ and $\Delta_k$ the
magnetic Laplacian acting of $\Ci ( M , L^k \otimes F)$ defined from any
connection of $F$ and the metric $\om ( \cdot, j \cdot)$ so that $\Si =
\frac{n}{2} + \N$.

Now let $P \in \lag (A)$ be selfadjoint and such that $\Pi_k P_k \Pi_k = P_k$.
Then the symbol $\si_0 (P) (x)$ is self-adjoint and has the form  $ \si_0 ( P) (x) =  f(x) \pi (x)$ with $f(x) \in \op{End}
F_x$. So $\si_0 ( P) (x) = u^*(x) f(x) u(x)$, so
\begin{gather} \label{eq:multiplicateur_f}
P_k = U_k ^* f U_k + \bigo (k^{-\frac{1}{2}} )
\end{gather}
where $f$ acts on $\Ci ( M , L^k \otimes F)$ by pointwise multiplication. 
For any $\Psi' \in \Ci ( M , L^k \otimes F)$,
$$ ( \inf_M f_{-} ) \; \| \Psi'\|^2 \leqslant \langle f \Psi' , \Psi' \rangle \leqslant (\sup_M f_+  ) \; \| \Psi'\|^2 $$
where $f_-(x)$ and $f_+(x)$ are the smallest and largest eigenvalues of $f(x)$
for any $x$. We conclude the proof by setting $\Psi ' = U_k \Psi$ and using
\eqref{eq:almost_unitary} and \eqref{eq:multiplicateur_f}.  
\end{proof}

\begin{cor} \label{cor:glob-spectr-estim}
 Let $(\Delta_k)$ be a family of formally self-adjoint differential operators
  of the form 
  \eqref{eq:ass_op}.  Let $a, b \in \R \setminus \Si$ with $a<b$. Then
  when $k$ is sufficiently large
\begin{gather} 
  \sharp \spec ( k^{-1} \Delta_k) \cap [a,b] = \begin{cases} \op{RR} (L^k
    \otimes F) \text{ if } [a,b] \cap \Si \neq \emptyset \\ 0 \qquad \text{
      otherwise,}
  \end{cases}
\end{gather}
with $F$ the bundle with fibers $F_x = 1_{[a,b]} ( \PP_x)$. Furthermore
\begin{gather} \label{eq:encad_spectre}
  \spec (  k^{-1} \Delta_k ) \cap [a,b] \subset [a,b] \cap \Si + \bigo
  ( k^{-\frac{1}{2}}).
\end{gather}
\end{cor}

\begin{proof} When $[a,b] \cap \Si$ is empty, we already know by Corollary
  \ref{cor:spectrum} that $\spec ( k^{-1} \Delta_k) \cap [a,b] $ is empty when
  $k$ is sufficiently large. When $[a,b] \cap \Si \neq \emptyset$, by Theorem
  \ref{theo:main_result_first}, the spectral
  projector $\Pi_k = 1_{[a,b]} ( k^{-1} \Delta_k)$ belongs to $\lag (A)$ with
    symbol $\pi = 1_{[a,b]} (\PP)$. So the dimension of $\op{Im} \Pi_k$ is
    given in the first assertion of Theorem \ref{theo:glob-spectr-estim}.

    Moreover,
    by Corollary \ref{cor:dsed}, $(k^{-1} \Delta_k )\Pi_k$ belongs to $\lag
    (A)$ and its symbol is $\PP 1_{[a,b]} (\PP)$. By the second assertion of Theorem
    \ref{theo:glob-spectr-estim}, 
$$ \spec (  k^{-1} \Delta_k ) \cap [a,b] = \spec (  k^{-1} \Delta_k \Pi_k )
\subset [\inf f_-, \sup f_+ ]  + \bigo
  ( k^{-\frac{1}{2}})$$ 
    where $f$ is the restriction of $\PP$
    to $F = \op{Im} \pi$.

    This proves the inclusion \eqref{eq:encad_spectre}
    when $[a,b] \cap \Si$ is connected. Indeed, $[a,b] \cap \Si_y = [f_-(y),
    f_+ (y)] \cap \Si_y $. So on one hand, $M$ being compact, $\inf f_- = f (
    y_-) $ and $\sup f_+ = f ( y_+)$ belongs to
    $[a,b] \cap \Si$. On the other hand $ [ a, b ] \cap \Si \subset [ \inf
    f _- , \sup f_+ ]$. Consequently
    $[a,b]\cap \Si = [\inf f_-, \sup f_+]$. 

    To treat the general case, we use that $[a,b] \cap
    \Si$ is a finite union of mutually disjoint compact intervals $I_1$,
    \ldots, $I_\ell$. So there exists $a_1 = a < a_2 < \ldots < a_{\ell+ 1} = b$
    in $\R \setminus \Si$ such that $I_i = [a_i, a_{i+1}] \cap \Si$ and by
    what we have proved, $\spec ( k^{-1} \Delta_k ) \cap [a_i, a_{i+1}] \subset
    I_i + \bigo ( k^{-\frac{1}{2}})$.
\end{proof}

\begin{rem} \label{rem:improvment}
  Decompose $\D (TM)$ into even and odd subspace
  $$ 
  \D  ^+ (TM) =
  \bigoplus_{p\in\N}  \D_{2p} (TM), \qquad \D ^{-} (TM) = \bigoplus_{p \in \N}
  \D_{2p+1} (TM) .$$ 
  Let us assume that $(\Pi_k)$ is even and that $F$ has a definite parity in the sense that  $F$ is a subbundle of
  $D^{\ep} (TM) \otimes A$ for $\ep = +$ or $-$. Then \eqref{eq:Garding} and
  \eqref{eq:encad_spectre} hold with $k^{-1}$ instead of $k^{-\frac{1}{2}}$.

  Indeed, by \cite[Theorem 2.5]{oimpart1}, the $\si_p$-symbol
  of $P_k \in \lag^{+}_p(A)$ has the same parity of $p$, meaning that $\si
  _p(P_k)$ sends $\D^{\ep} (TM) \otimes A$ into $\D^{ \ep'} (TM) \otimes A$
  with $\ep' = (-1)^p \ep$. 
 So if an operator $(P_k ) \in
  \lag_{1} ^+ (A)$ is such that $\Pi_k P_k \Pi_k = P_k$, then its symbol $\si_{1} ( P_k)$ is odd and has the form $ g \pi $ for some $g
  \in \op{End} F$. So $g$ is odd, but $F$ has a definite parity, so $g=0$.
  Consequently $(P_k ) \in \lag_{2}^+ (A)$. Moreover, by \cite[Theorem
  3.4]{oimpart1}, we can construct $(U_k) \in \lag (A, F)$  such that $U_k
  U_k^* = \op{id}$ when $k$ is sufficiently large and $(U_k)$ has the same
  parity as $F$. So if $(P_k) \in \lag^
  {+} (A)$, then $ (U_k P_k
  U_k^*) \in \lag ^{+} (F)$. Then in the proof of Theorem \ref{theo:glob-spectr-estim}, we can replace the
  $\bigo ( k^{-\frac{1}{2}})$ in \eqref{eq:almost_unitary} and
  \eqref{eq:multiplicateur_f} by a $\bigo ( k^{-1})$.
\qed  \end{rem}

\subsection{Local spectral estimates} \label{sec:local-spectr-estim}

\begin{theo} \label{theo:local-spectr-estim}
  Let $(P_k) \in \lag (A)$ be such that $\Pi_k P_k \Pi_k = P_k$ and
  $P_k^* = P_k$. Let $f \in \Ci ( M , \op{End} F)$ be the restriction of
  $\si_0(P_k)$ to $F$. 
\begin{enumerate} 
\item For any compact subsets $C$ of $M$ and $I$ of $\R$ such that $I
  \cap \spec (f(x)) = \emptyset$ for any $x \in C$, we have for any $N$
$$ (\Pi_k 1_{I} (P_k) \Pi_k)  (x,x) = \bigo (k^{-N}), \qquad \forall x \in
C $$ 
with a $\bigo$ uniform with respect to $x$. 
\item For any $g \in \Ci ( \R, \C)$, $(\Pi_k g(P_k) \Pi_k)$ belongs to
  $\mathcal{L}(A)$ and its $\si_0$-symbol is $(g \circ f) \pi$. Moreover, if
  $(\Pi_k)$ and $(P_k)$ are in $\lag^{+} (A)$, then the same holds for $(\Pi_k g(P_k) \Pi_k)$.
\end{enumerate}
\end{theo}

\begin{proof}
  Let $U$ be the open set $\{ x \in M; \; \spec (f(x) ) \cap I = \emptyset \}$.
  Let $\varphi \in \Ci_0(U)$ and $\la \in I$. Observe that $ \varphi  (f - \la )^{-1}
   \in \Ci ( M , \op{End} F)$. So if $(Q_k) \in \lag(A)$ has symbol $ \varphi (f -
  \la )^{-1}  \pi $, we have
\begin{gather} \label{eq:pi_k-q_k-pi_k}
  \Pi_k Q_k
  \Pi_k  (P_k-\la \Pi_k)  = \Pi_k \varphi \Pi_k - R_k \end{gather} 
  with $(R_k) \in \lag_1 (A)$. Let us improve this to obtain $(R_k) \in \lag_{\infty}
  (A)$.

  We need the following notion of support: for any $S \in \lag
  (A)$, $\op{supp} S$ is the closed set of $M$ such that $x \notin
  \op{supp} S$ if and only $S_k(y,z) = \bigo ( k^{-\infty})$ on a
  neighborhood of $(x,x)$. Using that the Schwartz kernel of $S \in \lag (A)$
  is in $\bigo ( k^{-\infty})$ on compact subsets of $M^2 \setminus \op{diag}
  M$ and in $\bigo ( k^n)$ on $M^2$, we prove that for any $S$, $S' \in \lag
  (A)$, $\op{supp} (SS') \subset (\op{supp} S) \cap ( \op{supp} S')$.

Assume now that $(Q_k) \in \lag (A)$ has the symbol $ \varphi  (f - \la )^{-1}
\pi$ as above and is supported in $U$. Then $(R_k) \in \lag_p (A)$ with $p
\geqslant 1$, $\Pi_k R_k \Pi_k = R_k$ so that the symbol $r = \si_p(R_k)$ satisfies $\pi
r \pi = r$. Furthermore, $(R_k)$ is supported in $U$, so the same holds for $r$, so that $r ( f - \la)^{-1} \in \Ci ( M , \op{End} F)$.  Let $(Q_k') \in \lag_p
(A)$ be supported in $U$ and have symbol $\si_p(Q_k') = r ( f - \la)^{-1} \pi$.
Then if we replace $Q_k$ in \eqref{eq:pi_k-q_k-pi_k} by $Q_k + Q_k'$, we have now $(R_k) \in \lag_{p+1}
(A)$. We deduce the existence of $(Q_k)$ such that \eqref{eq:pi_k-q_k-pi_k} holds
with $(R_k) \in \lag_{\infty} (A)$, so the operator norm of $R_k$ is in $\bigo
( k^{-\infty})$. 

We claim that this construction can be realized so that we obtain a $\bigo (
k^{-\infty})$ uniform with respect to $\la \in I$. To do this, we consider
families
\begin{gather} \label{eq:s_kla-in-lag}
  (S_k(\la)) \in \lag (A), \qquad \la \in I, 
\end{gather}
such that in the kernel
expansion \eqref{eq:expansion}, the coefficients $a_{\ell}$ depend
continuously on $\la$ and the remainders $r_{N,k}$ are in $\bigo (k^{n -
  \frac{N+1}{2}})$ on compact subsets of $U^2$ with a $\bigo$ independent of
$\la$. Then if $(S_k '( \la))$ is another family depending continuously on
$\la$ in the same sense, the same holds for the product $(S_k' ( \la)
S_k(\la))$. Furthermore, if $(S_k ( \la)) \in \lag_p (A)$ for any $\la \in I$, the
operator norm of $S_k ( \la)$ is in $\bigo ( k^{-\frac{p}{2}})$ with a $\bigo$
independent of $\la$. The proof of these claims is the same as the proof of
the same facts without $\la$.  Later in \eqref{eq:pi_k-s_k-z}, we will use these results again with the
parameter $\la$ describing a compact subset of $\C$. 

Now we deduce from \eqref{eq:pi_k-q_k-pi_k} with $\|R_k \| = \bigo (
k^{-\infty})$ that for any $k$, any normalised $\Psi \in \Hilb_k$ such that $P_k \Psi =
\la \Psi$ with $\la \in I$ satisfies $\langle \varphi \Psi, \Psi \rangle =
\bigo ( k^{-\infty}) $ with a $\bigo$ independent of $\la$ and
$\Psi$. For any $x \in U$, we can choose $\varphi$ equal to $1$ on a
neighborhood of $x$ and we deduce the existence of a compact neighborhood $V$ of $x$, such that
any $\Psi$ as above satisfies
$$\int_V | \Psi (x)|^2 d \mu (x) = \bigo ( k^{-\infty})$$
Writing $\Psi = \Pi_k \Psi$ and using that the Schwartz kernel of $\Pi_k$ is
in $\bigo ( k^{n})$ on $M^2$ and in $\bigo ( k^{-\infty})$ on compacts subset
of $M^2$ not intersecting the diagonal, we get that on a neighborhood of $x$
the pointwise norm of $\Psi$ is in $\bigo ( k^{-\infty})$. Since 
$ (\Pi_k 1_I (P_k) \Pi_k) (x,x)$ is the sum of the $ | \Psi_\ell (x) |^2$ where
$(\Psi_\ell)$ is an orthonormal basis of $\Hilb_k \cap \op{Im}  1_I (P_k)$ of
eigenvectors of $P_k$, and $\op{dim} \Hilb_k = \bigo ( k^{n})$, we deduce that 
$$ ( \Pi_k 1_I(P_k) \Pi_k) (x,x) = \bigo ( k^{-\infty}) , \qquad \forall x \in U
$$
with a $\bigo $ uniform on compact subsets of $U$. This ends the proof of the
first assertion.

For the second assertion, since the operator norm of $P_k$ is bounded independently of
$k$, we can assume that $g \in \Ci_0 ( \R, \C)$. We will apply the
Helffer-Sj\"ostrand formula, that we already used in a similar context for the
functional calculus of Toeplitz operators \cite[Proposition 12]{oimbt}. So for $\tilde{P}_k$ the
restriction of $P_k$ to $\Hilb_k$, we have
\begin{gather} \label{eq:Helffer_Sjostrand}
g ( \tilde P_k) = \frac{1}{2  \pi} \int_\C  (\partial_{\con{z}} \tilde{g} )(z)
\; (z
- \tilde P_k)^{-1} \;| dz d \con{z}|
\end{gather}
where $\tilde{g} \in \Ci_0 ( \C, \C)$ is an extension of
$g$ such that $\partial_{\con{z}}\tilde{g}$ vanishes to infinite
order along the real axis \cite[Theorem 14.8]{Zw}.

In the same way we proved \eqref{eq:pi_k-q_k-pi_k}, we can construct for any
$z \in \C \setminus \R$, $(Q_k (z)) \in \lag(A)$ such that $\Pi_k Q_k (z)\Pi_k = Q_k
(z) $ and 
\begin{gather} \label{eq:q_k--z}
Q_k ( z)  (z - P_k) = \Pi_k  - R_k (z) 
\end{gather}
with $(R_k (z) ) \in \lag _{\infty}(A)$. At the first step we set $Q_k (z)
= \Pi_k \tilde{Q}_k \Pi_k$ with $\tilde{Q}_k (z) $ in $\lag (A)$ having symbol
$(z - f )^{-1} \pi $. We obtain \eqref{eq:q_k--z} with $(R_k (z) ) \in
\lag_1(A)$. 
Then if $(R_k (z)) \in \lag_p ( A)$ and has symbol $ \si_p( R_k (z)) = r(z)$, we add to $Q_k$ the
operator $\Pi_k Q_k '(z) \Pi_k $ where $(Q_k'(z))$ is an operator of $\lag_p
(A)$ with symbol $ \si ( Q_k'(z) ) = r (z) (z- f)^{-1} \pi$. 

To apply this in \eqref{eq:Helffer_Sjostrand}, we need to control carefully
the dependence with respect to $z$. For $U$ an open set of $M$, introduce the
space $\mathcal{F} \Ci (U)$ consisting of family $(f(z, \cdot), \; z \in \C
\setminus \R)$ of $\Ci (U)$ having the form
$$ g(z,x) = \frac{ \sum_m a_{m} (x) z^m} { \sum_m b_m(x) z^m }$$
where the sums are finite, the coefficients $a_m$ and $b_m$ belongs to $\Ci
(U)$, and for any $x$ the poles of $g(\cdot, x)$ lie on the real axis.  Since
$\mathcal{F} \Ci (U)$ is a $\Ci (U) $-module, we can define $\mathcal{F} \Ci (
U, B)$ for any auxiliary bundle $B$ as the space of $z$-dependent section of
$B $ on $U$ with local representatives in $\mathcal{F}\Ci (U)$ for any
$z$-independent frame of $B$ on $U$.

Now, having in mind the construction of
$Q_k (z)$ in \eqref{eq:q_k--z}, observe that $(z - f)^{-1}$ belongs to
$\mathcal{F} \Ci ( M, \op{End} F)$. Moreover, $\mathcal{F} \Ci ( U)$ being closed
under product, for any $r(z) \in \mathcal{F} \Ci ( M, \op{End} F)$, $r(z)
 (z- f
)^{-1} \in \mathcal{F} \Ci ( M, \op{End} F)$.

Now introduce the space $\mathcal{F} \lag (A)$ consisting of families $(P_k (z)
, \; z \in \C \setminus \R)$ of $ \lag(A)$ such that in the asymptotic
expansion  \eqref{eq:expansion} satisfied by the Schwartz kernel of $P_k(z)$ the coefficients
 have the form 
\begin{gather} \label{eq:a_ell-z-x}
a_{\ell} (z,x,\xi) = \sum a_{\ell, \al } (z,x)
\xi^\al
\end{gather}
with $a_{\ell, \al} \in \mathcal{F} \Ci (U, \op{End} \C^r)$, and each remainder
$r_{N,k}$ is in $\bigo ( k^{n - \frac{N+1}{2}}) $ uniformly on $K \cap ( (\C
\setminus \R) \times U^2)$ where $K$ is any compact subset of
$\C \times U^2$. We claim that we can choose $Q_k (z) \in \mathcal{F} \lag
(A)$ in \eqref{eq:q_k--z}. To see this, it suffices to prove that 
\begin{gather} \label{eq:lem} 
S (z) \in \mathcal{F} \lag  (A) \Rightarrow \Pi_k S_k (z) \Pi_k (z - P_k )
\in \mathcal{F} \lag (A)
\end{gather} 
and then to use what we said before on $r(z)
\circ (z - f)^{-1}$.  To prove \eqref{eq:lem},  it suffices to show
that for any $ S(z) \in \mathcal{F} \lag (A)$ and $T \in \lag (A)$ independent
of $z$, $TS(z)$ and $S(z) T$ belongs to $\mathcal{F} \lag (A)$. To prove this,
we can assume
that the Schwartz kernel of $T(z)$ is contained in a compact subset of $U^2$
independent of $k$ and $z$,
where we have the expansion \eqref{eq:expansion}, and we can treat each term
of the expansion independently of the others. Suppose we only have $a_{\ell}
(z,x, \xi)$. Then by \eqref{eq:a_ell-z-x}, $ S(z) = \sum S_\al
a_{\ell, \al} (z, \cdot)  $ where the sum is finite, $S_{\al} \in \lag_{\ell} (A)$ and
does not depend on $z$.  Since $T S(z) = \sum (TS_{\al}) a_{\ell, \al}(z ,
\cdot )$ and
$TS_{\al} \in \lag_{\ell} (A)$ for any $\al$, $TS(z)$ belongs
to $\mathcal{F} \lag_{\ell}  (A)$. The product $S(z)T $ is more delicate to handle. By
the same proof as \cite[Lemma 5.11]{oim}, for any compact set $K$ of $U$, there exists a family $(T_{\be}, \be
\in \N^{2n})$ such that $ T_{\be} \in \lag_{|\be|}  (A)$, and for any $f \in \Ci_K
(U)$, $f T = \sum_{|\be| \leqslant N} T_{\be} (\partial^{\be} f) $ modulo
$\lag_{N+1} (A)$. Consequently
$$ S(z) T = \sum_{ \al} S_{\al} (a_{\ell,\al}(z, \cdot ) T) =  \sum_{ \al, |\be|
  \leqslant N} S_{\al} T_{\be} (\partial^\be a_{\ell, \al}(z, \cdot) ) $$
 modulo
$\lag_{N+1} (A)$. To conclude we use that $S_{\al} T_{\be} \in \lag _{\ell +
  |\be|} (A)$ and $\partial^\be a_{\ell, \al }  \in \mathcal{F}  \Ci ( U,
\op{End} \C^r)$.

Now the function $\xi (z) = (\op{Im} z)^{-1} \partial_{\con z}
\tilde{g} (z)$ vanishes to infinite order along the real axis and its
support is contained in the compact set $K = \op{supp} \tilde g$.  For any $ f \in \mathcal{F} \Ci (U)$, the
product $\xi (z) f(z, \cdot )$ extends smoothly to $\C$. We deduce that there exists a family $(S_k (z) )$ of $\lag (A)$ depending continuously of $z
\in K$
in the same sense as \eqref{eq:s_kla-in-lag}, and such that 
\begin{gather} \label{eq:pi_k-s_k-z}
\Pi_k S_k (z) \Pi_k = S_k (z), \qquad S_k (z) (z - P_k ) = \xi (z) \Pi_k +
\bigo ( k^{-\infty})
\end{gather}
with a $\bigo$ uniform with respect to $z$. 
Since $\| (z - \tilde{P}_k)^{-1} \| = \bigo (|\op{Im} z |^{-1})$, multiplying
the last equality by $( \op{Im} z ) (z- \tilde P_k)^{-1}$, we obtain
\begin{gather} \label{eq:part-z-tild}
\partial_{\con z}
\tilde{g} (z) (z- \tilde{P}_k)^{-1} \Pi_k = (\op{Im} z)
 S_k (z) + R_k (z) 
\end{gather}
with $R_k (z) = \bigo (k^{-\infty})$.
Since $\Pi_k R_k (z) \Pi_k = R_k (z)$ and the schwartz kernel of $\Pi_k $ is in
$\bigo ( k^n)$, this implies that the Schwartz kernel of $R_k(z) $ is in $\bigo (
k^{-\infty})$  uniformly with respect to $z$. 
Inserting \eqref{eq:part-z-tild} in \eqref{eq:Helffer_Sjostrand}, it comes that $(g ( \tilde{P}_k)\Pi_k )$ belongs to $\mathcal{L}(A)$. To see this,
we simply have to integrate with respect to $z$ the coefficients $a_{\ell}
(z,x,\xi)$ in the expansion \eqref{eq:expansion} of the Schwartz kernel of $(
\op{Im} z ) S_k (z)$. Since $\si_0 ( (\op{Im} z) S_k (z) ) = \partial_{\con z}
\tilde{g} (z) ( z- f)^{-1}
\pi$, we deduce also that
$$ \si_0 ( g ( \tilde{P}_k)\Pi_k ) = \frac{1}{2  \pi} \int_\C \partial_{\con z}
\tilde{g} (z) (z-f)^{-1}
\pi \; |dz d \con{z}| =  g ( f) \pi $$
  which concludes the proof.    
\end{proof}

\begin{cor} \label{cor:local-spectr-estim}
 Let $(\Delta_k)$ be a family of formally self-adjoint differential operators
  of the form 
  \eqref{eq:ass_op}. Let $\bs \in \R \setminus \Si$. Then for any $g \in \Ci (\R , \C)$ supported in $]-\infty, \bs]$, $(g ( k^{-1}
  \Delta_k ))$ belongs to $\lag ^+ (A)$ and has symbol $g ( \PP)$. 
\end{cor}

\begin{proof}
Since $g$ is supported in $]-\infty, \bs]$, $g ( k^{-1} \Delta_k ) = \Pi_k g
( k^{-1} \Delta_k \Pi_k) \Pi_k$ where $\Pi_k = 1_{]-\infty, \bs]} ( k^{-1}
\Delta_k)$. By Theorem \ref{theo:main_result_first} and Corollary
\ref{cor:dsed}, $( \Pi_k)$ and $k^{-1} \Delta_k \Pi_k$ belongs to $\lag^+ (A)$
with symbols $\pi = 1_{]-\infty, \bs ]} (\PP)$ and $ f = g ( \PP)$.
So the results follows from the second assertion of Theorem
\ref{theo:local-spectr-estim}.
\end{proof}
This proves the second part of Theorem \ref{theo:intro_proj}. We end this
section with the proof of the local Weyl laws, Theorem
\ref{theo:intro_weyl_local}. The proof works for any $(\Delta_k)$ of the form
\eqref{eq:ass_op}.
\begin{proof}[Proof of Theorem \ref{theo:intro_weyl_local}] We use the same
  notation as in Corollary \ref{cor:local-spectr-estim} and its proof.  Let $a,b  \in ]
  - \infty, \bs] \setminus \Si_y$.  We have $\spec (f(y) ) = \Si_y \cap
  ]-\infty, \bs]$. When $[a,b] \cap \Si_y$ is empty, the first part of Theorem
  \ref{theo:local-spectr-estim} implies that $N(y,a,b,k) = \bigo (
  k^{-\infty})$. 
To the contrary, assume that $[a,b] \cap \Si_y = \{ \la \}$. Then choose a
function $g \in \Ci_0 ( ]a,b[, \R)$ which is equal to $1$ on $]\la - \ep, \la
+ \ep[$ for some $\ep >0$. Since $N(y,a, \la - \ep, k ) = \bigo (
k^{-\infty})$ and $N(y, \la+ \ep , b, k) = \bigo ( k^{-\infty})$ by the first
part of the proof,
$$ N(y,a,b, k ) = g ( k^{-1} \Delta_k) (y,y) + \bigo ( k^{-\infty}).$$
Since $g( k^{-1} \Delta_k)$ is in $\lag^+ (A)$ and has symbol $g ( \PP)$ we
have by \cite[Theorem 2.2, Assertion 5 and Proposition 5.6]{oimpart1}
$$ g ( k^{-1} \Delta_k) (y,y)  = \Bigl( \frac{k}{2\pi} \Bigr)^n \sum_{\ell =0
}^{\infty} m_{\ell, \la } k^{-\ell} + \bigo ( k^{-\infty})$$
with $m_{0, \la} = \op{tr} g ( \PP) (y)$, so $m_{0,\la}$ is the multiplicity
of $\la$ as an eigenvalue of $\PP_y$. 
\end{proof}

\section{Miscellaneous proofs} \label{sec:misceleanous-proofs}

\begin{proof}[Proof of Lemma \ref{lem:Darboux}]
  This is essentially Darboux lemma with parameters. We can adapt the
   proof presented in  \cite[Section 3.2]{MS}. A more efficient
  approach based on \cite{BLM} is as follows. 
First, if $r$ is sufficiently small, for any $y$, the exponential map $\exp _y :T_yM
\rightarrow M$ restricts to an embedding from $B_y(r)$ into $M$. Identify $
U = \exp_y ( B_y (r))$ with an open set of $T_yM$. We are looking for a
diffeomorphism $\varphi $ defined on a neighborhood of the origin of $T_yM$
such that $\varphi ( 0 ) =0$, $T_0 \varphi = \op{id}$ and $\varphi ^* \om$ is
constant. The important point is to define $\varphi$ in such a way that it
depends smoothly on $y$. 

Let $\al$
be the primitive of $\om$ on $U$ obtained by radial homotopy. So
\begin{gather} \label{eq:al_x}
\al_x ( v) = \int_0^1 \om_{tx} ( tx, v) \; dt, \qquad x \in U, \;
v \in T_yM 
\end{gather}
and $d \al = \om$. Let $X$ be the vector field of $U$ such that $
\iota_X \om = 2 \al$. By Poincar\'{e} Lemma, $\mathcal{L}_X \om  = 2 \om$.
Furthermore, linearising $\al$ at the origin, we see that
$X = E + \bigo (2)$ with $E$  the Euler vector field of $T_yM$. Since $Z =
X -E$ vanishes to second order at the origin, the family $Z_t ( x) :=
\frac{1}{t^2}  Z  ( t x)$ extends smoothly at $t=0$. Let $\varphi_t$ be the
flow of the time-dependent vector field $Z_t$ of $U$, that is
$\varphi_0 ( x) = x$ and $\dot{\varphi}_t(x) = Z_t ( \varphi_t(x))$. Since $Z_t$
is zero at the origin, $\varphi_1$ is a germ of diffeomorphism of $(T_yM, 0)$. 
By the
proof of Lemma 2.4 in \cite{BLM}, $\varphi_1^* X = E$, where  the pull-back is
defined by $ \varphi_1^* X = ( \varphi_1^{-1})_* X$. So $\mathcal{L}_X \om = 2
\om$ implies that $\mathcal{L}_{E} \varphi^*_1 \om = 2 \varphi_1^* \om$. So
$\varphi_1^* \om$ is constant.  

To conclude, observe that $\varphi_1$ depends smoothly on $y$ because $\al$
given in \eqref{eq:al_x} depends smoothly on $y$, so the same holds for $X$
and $Z_t$, and the solution of a first order differential equation depending
smoothly on a parameter, is smooth with respect to the parameter. Finally the
radius $r_0$ is chosen so that $\varphi_1$ is defined on $ B_y (r_0)$. Since
$M$ is compact, we can choose $r_0>0$ independent of $y$.  
\end{proof}

\begin{proof}[Proof of Lemma \ref{lem:covering}]
Let $d$ be the geodesic distance of $M$ associated to our Riemannian metric.  Starting from $d ( y, \exp_y (\xi) ) =\| \xi \|$ when $\xi$ is sufficiently
close to the origin,  we get that 
\begin{gather} \label{eq:ineq_norme}
C^{-1} \| \xi \| \leqslant  d( y, \Psi_y ( \xi)) \leqslant C \| \xi \|  
\end{gather}
for any $\xi \in B_y(r_1)$ with $r_1$ sufficiently small. So if $B( y,r)$ is 
the open ball of the metric space $(M,d)$, then $\Psi_y ( B_y(r ) ) \subset B( y, 
r C))$ and $B( y, r ) \subset \Psi_y ( B_y ( rC))$.
Define
$$v_- ( \ep) = \inf \{ \op{vol} (B(y, \ep)), \; y \in M\}, \quad v_+ ( \ep ) =
\sup \{ \op{vol} (B(y, \ep)), \; y \in M\} $$
Then, replacing $C$ by a larger constant if necessary, when $\ep$ is
sufficiently small, $C^{-1} \ep ^{2n} \leqslant v_- ( \ep)$ and $v_+ ( \ep )
\leqslant C \ep ^{2n}$.

For any $\ep>0$, choose a maximal subset $J(\ep)$ of $M$ such that the balls $B( y,
\ep/2)$, $y \in J( \ep)$, are mutually disjoint. From the maximality, $M
\subset \bigcup_{y \in J(\ep)} B(y, \ep)$ so that the sets $ U_y (\ep):= \Psi_y ( B_y (
\ep C)) $, $ y \in J(\ep)$ cover $M$. For any $x \in M$, let $N(x, \ep)$ be the
number of $y \in J(\ep)$ such that $x\in U_y(\ep) $. If $x \in U_y(\ep)$, by triangle
inequality, $B(y, \ep/2) \subset B(x, \ep ( 1 + C^2))$. Since the balls $B(y,
\ep/2)$, $y \in J(\ep)$ are mutually disjoint, it comes that 
$$ N(x, \ep) v_{-} (\ep/2) \leqslant \op{vol} ( B( x, \ep ( 1+ C^2))) \leqslant
v_{+} ( \ep ( 1+ C^2 ))$$
So $N(x, \ep) \leqslant C^2(2 ( 1+ C^2))^{2n}$. So the multiplicity of the cover
$U_y(\ep)$, $y \in J(\ep)$ is bounded independently of $\ep$. 
\end{proof}

\bigskip
\noindent
\begin{tabular}{l}
Laurent Charles, \\  Sorbonne Universit\'e, CNRS, \\ 
 Institut de Math\'{e}matiques 
 de Jussieu-Paris Rive Gauche, \\
  F-75005 Paris, France. 
\end{tabular}

\end{document}